\documentclass[final]{colt2022}

\usepackage{enumitem}
\usepackage{ifthen}
\usepackage{times}
\newenvironment{thisnote}{\par\color{gray}}{\par}

\usepackage{mathtools}
\usepackage{mathrsfs} 
\makeatletter\def\@seccntformat#1{\protect\makebox[0pt][r]{\csname the#1\endcsname\hspace{12pt}}}\makeatother
\definecolor{mycolor1}{rgb}{0.105882,0.619608,0.466667}
\definecolor{mycolor2}{rgb}{0.85098,0.372549,0.00784314}
\definecolor{mycolor3}{rgb}{0.458824,0.439216,0.701961}
\definecolor{mycolor4}{rgb}{0.905882,0.160784,0.541176}
\definecolor{mycolor5}{rgb}{0.4,0.65098,0.117647}
\definecolor{mycolor6}{rgb}{0.65098,0.462745,0.113725}
\definecolor{mycolor7}{rgb}{0.901961,0.670588,0.00784314}
\definecolor{mycolor8}{rgb}{0.4,0.4,0.4}
\definecolor{mycolor9}{rgb}{0.301961,0,0.294118}
\definecolor{mycolor10}{rgb}{0.0313725,0.25098,0.505882}

\DeclareMathOperator{\arccosh}{arccosh}

\usepackage{bigfoot}

\newcommand{\cutchunk}[1]{}

\newcommand{\removesafe}[1]{}

\newcommand{\calA}{\mathcal{A}}
\newcommand{\calB}{\mathcal{B}}
\newcommand{\calP}{\mathcal{P}}

\newcommand{\calF}{\mathcal{F}}

\newcommand{\calK}{\mathcal{K}}

\newcommand{\calM}{\mathcal{M}}
\newcommand{\calN}{\mathcal{N}}
\newcommand{\calO}{\mathcal{O}}

\newcommand{\Rmcurv}{\mathrm{Rm}}

\newcommand{\diag}{\mathrm{diag}}
\newcommand{\Vol}{\mathrm{Vol}}

\newcommand{\dist}{\mathrm{dist}}

\newcommand{\floor}[1]{\lfloor #1 \rfloor}

\newcommand{\grad}{\mathrm{grad}}
\newcommand{\Hess}{\mathrm{Hess}}

\newcommand{\inner}[2]{\left\langle{#1},{#2}\right\rangle}

\newcommand{\Kup}{{K_{\mathrm{up}}}}
\newcommand{\Klo}{{K_{\mathrm{lo}}}}

\newcommand{\norm}[1]{\left\|{#1}\right\|}

\newcommand{\Rball}{R_{\mathrm{ball}}}

\newcommand{\Rd}{{\mathbb{R}^{d}}}

\newcommand{\reals}{{\mathbb{R}}}

\newcommand{\Sd}{\mathbb{S}^{d}}

\newcommand{\Sym}{\mathrm{Sym}}
\newcommand{\T}{\mathrm{T}}

\newcommand{\trace}{\mathrm{Tr}}

\newcommand{\xorigin}{x_{\mathrm{ref}}}

\newcommand{\aref}[1]{\hyperref[#1]{A\ref{#1}}}
\newtheorem{assumption}{A\ignorespaces} 

\title[Negative curvature obstructs acceleration for geodesically convex optimization]{Negative curvature obstructs acceleration for {strongly} geodesically convex optimization, even with exact first-order oracles}

\coltauthor{\Name{Christopher Criscitiello} \Email{christopher.criscitiello@epfl.ch}\AND
\Name{Nicolas Boumal} \Email{nicolas.boumal@epfl.ch}\\
\addr Ecole Polytechnique F\'ed\'erale de Lausanne (EPFL), Institute of Mathematics}

\begin{document}

\maketitle

\begin{abstract}
\citet{hamilton2021nogo} showed that, {in certain regimes}, it is not possible to accelerate Riemannian gradient descent in the {hyperbolic plane} if we restrict ourselves to algorithms which make queries in a (large) {bounded domain} and which receive gradients and function values corrupted by a (small) amount of {noise}.
We show that acceleration remains unachievable for any deterministic algorithm which receives exact gradient and function-value information (unbounded queries, no noise). 
Our results hold for a large class of Hadamard manifolds including hyperbolic spaces and the symmetric space $\mathrm{SL}(n) / \mathrm{SO}(n)$ of positive definite $n \times n$ matrices of determinant one.
{This cements a surprising gap between the complexity of convex optimization and geodesically convex optimization: for hyperbolic spaces, Riemannian gradient descent is optimal on the class of smooth and strongly geodesically convex functions (in the regime where the condition number scales with the radius of the optimization domain).}
The key idea for proving the lower bound consists of perturbing squared distance functions with sums of bump functions chosen by a resisting oracle.
\end{abstract}



\section{Introduction}
We consider optimization problems of the form
\begin{align}
	\min_{x \in \calM} f(x)
	\tag{P}
	\label{eq:P}
\end{align}
where $\calM$ is a Riemannian manifold and $f \colon \calM \rightarrow \reals$ is a smooth strongly geodesically convex (g-convex) function (we review technical geometric terms in Section~\ref{prelims}).
When $\calM$ is a Euclidean space, problem~\eqref{eq:P} amounts to smooth strongly convex optimization.

Several problems of interest are non-convex but can be recast as g-convex optimization problems, which means global solutions can be found efficiently.
{Examples and applications in data science, statistics and machine learning include: computing intrinsic means or medians on curved spaces~\citep{karcher1977riemannian,YuanHAG20} such as for computational anatomy~\citep{Fletcher2009TheGM} or phylogenetics~\citep[Ch.~8]{bacak2014hadamard}, metric learning~\citep{zadeh2016geometricmetriclearning}, computing optimistic likelihoods~\citep{nguyen2019calculating}, parameter estimation for mixture models~\citep{hosseini2015manifoldgmm}, robust covariance estimation and subspace recovery~\citep{auderset2005angular,wiesel2012gconvexity,zhang2012gconvexity,wiesel2015gconvexity,sra2015conicgeometricoptimspd,ciobotaru2020geometrical,franksmoitra2020}, estimation for matrix normal models ~\citep{tang2021integrated,amendola2021,franks2021neartyler}, sampling on Riemannian manifolds~\citep{goyalsampinggconvexity}, and landscape analysis such as for matrix completion~\citep{ahn2021riemannianmatrixfact}.  
In mathematics and theoretical computer science, applications of g-convex optimization include computing Brascamp--Lieb constants~\citep{sravishnoibracsampliebconstant}, the null cone membership problem and polynomial identity testing---see~\citep{allenzhuoperatorsplitting,burgissernoncommutativeoptimization,franksreichenback2021} and references therein.}
More generally, optimization on manifolds also has applications in scientific computing, imaging, communications and robotics~\citep{AMS08,hu2020introoptimmanifolds,boumal2020intromanifolds}.

Given these applications, it is natural to ask for fast algorithms for the g-convex optimization problem~\eqref{eq:P}.
We consider algorithms which have access to an oracle providing first-order information (function values and gradients), and consider the following computational task:
\begin{quote} \label{problem1}
\emph{Let $f \colon \calM \rightarrow \reals$ be a $\mu$-strongly g-convex function which is $L$-smooth in a geodesic ball $B$ of radius $r$ and whose minimizer $x^*$ lies in $B$ (see Sections~\ref{introproblemclass} and~\ref{gconvexfunctionsdetails}).
Find a point $x \in \calM$ within distance $\frac{r}{5}$ of $x^*$.}\footnote{It suffices to ask how many queries are required to reduce the uncertainty radius $r$ by a factor $\epsilon$, for any fixed $\epsilon \in (0, 1)$.  Throughout we take $\epsilon = \frac{1}{5}$, following~\citet{hamilton2021nogo}.}
\end{quote}
The radius $r$ represents our initial uncertainty about the location of the minimizer of $f$.
Thus we ask: how many queries are required to reduce our uncertainty by a constant factor (five in this case)?
When $\calM = \reals^d$ is Euclidean space, g-convexity is equivalent to convexity, and it is well known that (projected) gradient descent (GD) uses at most $O(\kappa)$ queries to solve this computational task, where  
the condition number $\kappa = \frac{L}{\mu}$ represents the {conditioning} of the problem.
In contrast, Nesterov's accelerated gradient method (NAG) (adapted to the ball domain) uses $\tilde{O}\big(\sqrt{\kappa}\big)$ queries~\citep[Thm.~6 and Sec.~5.1]{nesterovacceleratedgradientinconvexset2007}, and that is optimal~\citep[Ch. 2]{nesterov2004introductory}.\footnote{Throughout, $O$ and $\Omega$ do \emph{not} hide the parameter $r$.   Also, $\tilde{O}$ and $\tilde{\Omega}$ hide logarithmic factors in $\kappa=\frac{L}{\mu}$ and $r$.}

For the moment, let us consider the case where $\calM$ is a hyperbolic space, meaning it has constant negative curvature.
{\citet[App.~D]{davidmr} show that a version of projected Riemannian gradient descent (RGD) uses at most {$\tilde O(\kappa)$} gradient queries to solve the computational task described above.}\footnote{{In the published version of this paper, we cite~\citet{zhang2016complexitygeodesicallyconvex} (see Appendix~\ref{curvdependenceandbestratesforRGD}); however, as pointed out by~\citet{davidmr}, the proof in~\citep[Prop.~15]{zhang2016complexitygeodesicallyconvex} does not work in the constrained case.  \citet[App.~D]{davidmr} give a version of RGD for constrained optimization with rate $\tilde O(\kappa)$}.}
This matches the rate of gradient descent in Euclidean spaces.  We are led to the following question:
\begin{quote}
\emph{Is there an algorithm for g-convex optimization on hyperbolic spaces which solves the above computational task in $\tilde O(\sqrt{\kappa})$ queries?}
\end{quote}
\noindent In this paper, we show that \emph{no such accelerated algorithm exists}, and in fact a version of {RGD is optimal} for smooth strongly g-convex optimization on hyperbolic spaces, in the regime $r = \Theta(\kappa)$ (see Section~\ref{mainresultssection}).\footnote{To establish lower bounds when $\kappa \gg r \gg 1$, further ideas seem to be necessary.}
Indeed, a number of algorithms have been developed to address this question~\citep{liu2017gconvexacceleration,zhang2018estimatesequence,ahn2020nesterovs,jinsra2021riemannianacceleration,martinezrubio2021global,lezcanocasado2020adaptive,alimisis2019continuoustime,alimisis2021momentum,huang2021sparsepca,duruisseaux2021ODEacceleration,franca2021b,franca2021a}, but none are proven to achieve the fully accelerated rate of $\tilde O(\sqrt{\kappa})$.  

Our analysis builds on the recent work of~\citet{hamilton2021nogo}, who show that acceleration on the hyperbolic plane is impossible when function values and gradients are corrupted by noise, even when this noise is very small.
Their argument introduces a number of important ideas, most notably a key geometric property of the hyperbolic plane which we call the ``ball-packing property'': for $r > 0$ sufficiently large, any geodesic ball of radius $r$ in the hyperbolic plane contains $N = e^{\Theta(r)}$ disjoint open geodesic balls of radius $\frac{r}{4}$.
Using several of~\citet{hamilton2021nogo}'s ideas, plus additional ideas we introduce, we prove that acceleration is impossible even when function values and gradients are known exactly, i.e., not corrupted by noise.

In addition to proving lower bounds for queries yielding exact information (our main contribution), we improve upon the results of~\citet{hamilton2021nogo} in several ways.  In Section~\ref{subsecballpackingproperty}, we establish the ball-packing property for a large class of Hadamard manifolds, including the symmetric space $\mathrm{SL}(n) / \mathrm{SO}(n)$ of positive definite matrices with determinant one which is important in applications, at least in part because the Riemannian metric on $\mathrm{SL}(n) / \mathrm{SO}(n)$ is the Fisher--Rao information metric for covariance matrices of Gaussian distributions~\citep{skovgaard1984riemgeogaussians,statsofPDmatricesfisher2006}.  In turn, we show that acceleration is impossible on this large class of Hadamard manifolds.  
\citet{hamilton2021nogo} also restrict all algorithms to query in a bounded domain.  We remove this assumption using a reduction which, starting from hard functions designed for algorithms making bounded queries, produces hard functions for algorithms which can make unbounded queries.  
\cutchunk{
\item \citet{hamilton2021nogo} prove lower bounds for strongly g-convex optimization.  In Section~\ref{nonstronglyconvexcaseextension}, we extend our lower bounds to nonstrongly g-convex optimization.  {In the appropriate regime,} acceleration remains impossible.}


\subsection{Key ideas: building hard g-convex functions}
\citet{hamilton2021nogo} establish their lower bound by exhibiting a distribution on strongly g-convex functions which is challenging for any algorithm receiving function information corrupted by noise.
Amazingly, the hard distribution they consider is simply a uniform distribution over a finite number of Riemannian squared distance functions.
Intuitively, this distribution is difficult for algorithms because geodesics diverge rapidly in hyperbolic space (see Lemma~\ref{geodesicsdiverge}), so a small amount of noise in a gradient is magnified.
However, like in Euclidean spaces, the Riemannian gradient of a squared distance function points directly towards the function's minimizer.
Therefore, squared distance functions are not enough to go beyond noisy oracles. 

\cutchunk{
Classical lower bounds in smooth convex optimization---``worst function in the world'' arguments~\citep[Ch.~2]{nesterov2004introductory}---are based on quadratic functions.
Thus, we at first considered classes of quadratic-like functions on hyperbolic spaces which generalize the Riemannian squared distance functions.
Ultimately, we were unable to make this work.
}

The key idea we introduce is to use squared distance functions \emph{perturbed by a resisting oracle}, i.e., functions $f(x) = \frac{1}{2}\dist(x, x^*)^2 + H(x)$ with $\norm{\Hess H(x)}$ small.
The perturbations $H$ are not g-convex, but since their Hessian is small, the perturbed functions $f$ retain strong g-convexity.
Each perturbation is constructed as a \emph{sum of bump functions}, that is, $C^\infty$ functions with compact support.

\subsection{Algorithm and problem classes} \label{introproblemclass}
It is crucial to define the class of functions for which we prove lower bounds.
A natural function class to consider is the set of functions $f \colon \calM \rightarrow \reals$ which are $L$-smooth\footnote{We say a function is $L$-smooth if it has {$L$-Lipschitz Riemannian gradient} (Definition~\ref{definitionLsmoothness}).  When we say a function is smooth, we mean that it is $L$-smooth for some $L \geq 0$.  We say a function is $C^{\infty}$ if it is infinitely differentiable.} and $\mu$-strongly g-convex on all of $\calM$ (see Section~\ref{gconvexfunctionsdetails}).
Yet, if $\calM$ has sectional curvatures upper bounded by some $\Kup < 0$ or if $\calM = \mathrm{SL}(n) / \mathrm{SO}(n)$, then this class is empty.  It is impossible for a function to be both $L$-smooth and strongly g-convex on all of $\calM$ if $\Kup < 0$ or if $\calM = \mathrm{SL}(n) / \mathrm{SO}(n)$.\footnote{See Proposition~\ref{geometryinfluencesobjective} in Appendix~\ref{geometryinfluencesfunctions}, which is an extension of a result due to~\citet{hamilton2021nogo}.}

A simple remedy for this issue is to consider minimizing $\mu$-strongly g-convex functions which are $L$-smooth in a ball of finite radius $r$.
This is especially natural since whether acceleration is possible depends on how $r$ compares with $\kappa$---this will become clearer in Section~\ref{comparisontolit}.
Let $\calM$ be a Hadamard manifold, and let $B(\xorigin, r) \subseteq \calM$ denote the closed geodesic ball centered at $\xorigin \in \calM$ of radius $r$ (see Section~\ref{hadamardmanssections}).
We consider the following class of real-valued functions $f \colon \calM \rightarrow \reals$.
\begin{definition} \label{deffirstfctclass}
For $\kappa\geq 1, r > 0, \xorigin \in \calM$, let $\calF_{\kappa, r}^{\xorigin}(\calM)$ be the set of $C^{\infty}$ functions on $\calM$ which
\begin{itemize}
\item are $\mu$-strongly g-convex in all of $\calM$ with $\mu > 0$;
\item are $L$-smooth in $B(\xorigin, r)$ with $\kappa = \frac{L}{\mu}$; and
\item have a unique global minimizer $x^*$ which lies in the ball $B(\xorigin, \frac{3}{4} r)$.
\end{itemize}
\end{definition}
In the third item of Definition~\ref{deffirstfctclass}, we require $\frac{3}{4} r$ instead of $r$ to ensure that the ball $B(x^*, \frac{r}{5})$ is contained in the interior of $B(\xorigin, r)$.

We impose no restrictions on the algorithm except that it is deterministic.
A deterministic first-order algorithm $\calA$ on $\calM$ is an initial point $x_0$ and a sequence of maps $(\calA_{k} \colon (\reals \times \T \calM)^k \rightarrow \calM)_{k \geq 1}$.
Running an algorithm $\calA$ on a cost function $f \colon \calM \rightarrow \reals$ produces iterates $x_0, x_1, x_2, \ldots$ given by $x_{k} = \calA_k((f_0, (x_0, g_0)), \ldots, (f_{k-1}, (x_{k-1}, g_{k-1})))$, where $f_\ell = f(x_\ell)$ and $g_\ell = \grad f(x_\ell)$ constitute the past function value and gradient information gathered thus far.
It is an open question whether the lower bounds in this paper can be extended to randomized algorithms.

\subsection{Main results} \label{mainresultssection}
We now state our main results about the impossibility of acceleration for the function class $\calF_{\kappa, r}^{\xorigin}(\calM)$ in Definition~\ref{deffirstfctclass}.
There is some leeway in choosing the constants below.

\begin{theorem} \label{cor1}
Let $\calM$ be a Hadamard manifold of dimension $d \geq 2$ whose sectional curvatures are in the interval $[\Klo, \Kup]$ with $\Kup < 0$.
Let $\xorigin \in \calM$, $\kappa \geq 1000 \sqrt{\frac{{\Klo}}{{\Kup}}}$ and define $r > 0$ such that $\kappa = 12 r \sqrt{-\Klo}+9$.
For every deterministic first-order algorithm $\calA$, there is a function $f \in \calF_{\kappa, r}^{\xorigin}(\calM)$ such that algorithm $\calA$ requires at least
\begin{align*}
&\Bigg\lfloor \sqrt{\frac{\Kup}{\Klo}} \cdot \frac{\kappa}{1000 \log\big(10 \kappa\big)}\Bigg\rfloor = \tilde{\Omega}\bigg(\sqrt{\frac{\Kup}{\Klo}} \cdot\kappa\bigg)
\end{align*}
queries in order to find a point $x \in \calM$ within distance $\frac{r}{5}$ of the minimizer of $f$.
\end{theorem}
\begin{corollary}
Let $\calM$ be a hyperbolic space ($\Klo = \Kup = K < 0$), $\xorigin \in \calM$, $\kappa \geq 1000$ and define $r > 0$ such that $\kappa = 12 r \sqrt{-K}+9$.  Among deterministic first-order algorithms, the projected gradient descent method in~\citep[App.~D]{davidmr} is optimal (up to log factors) on the function class $\calF_{\kappa, r}^{\xorigin}(\calM)$.
\end{corollary}

The symmetric space $\mathrm{SL}(n) / \mathrm{SO}(n)$ does not have strictly negative curvature as required by Theorem~\ref{cor1}, but we can still show that acceleration is unachievable if $n \geq 2$ is held fixed as $\kappa$ grows.
\begin{theorem} \label{cor3}
Let $\xorigin \in \mathrm{SL}(n) / \mathrm{SO}(n)$, $\kappa \geq 1000 n$ and define $r > 0$ such that $\kappa = 6 r \sqrt{2} + 9$.
For every deterministic first-order algorithm $\calA$,
there is a function $f \in \calF_{\kappa, r}^{\xorigin}(\mathrm{SL}(n) / \mathrm{SO}(n))$
such that the algorithm $\calA$ requires at least $\big\lfloor \frac{1}{n} \cdot \frac{\kappa}{1000 \log(10 \kappa)}\big\rfloor = \tilde{\Omega}\big(\frac{1}{n} \cdot \kappa\big)$
queries in order to find a point $x$ within distance $\frac{r}{5}$ of the minimizer of $f$.
\end{theorem}

The lower bound $\tilde \Omega(\frac{\kappa}{n})$ also holds for the symmetric space $\calP_n$ of positive definite matrices with affine-invariant metric because it is isometric to $\reals \times \mathrm{SL}(n) / \mathrm{SO}(n)$ (see Appendix~\ref{PDmatricesappendix}).
It is an open question whether one can remove the factors $\sqrt{\frac{\Kup}{\Klo}}$ and $\frac{1}{n}$ in the lower bounds in Theorems~\ref{cor1} and~\ref{cor3}.

\cutchunk{
Whether acceleration is possible depends on how $r$ scales as the condition number $\kappa = \frac{L}{\mu}$ grows.
There are now known results for two regimes:
\begin{itemize}
\item ($r \leq O(\frac{1}{\kappa^{3/4}})$) In this regime, we can achieve acceleration due to the results of \citet{zhang2018estimatesequence} and \citet{ahn2020nesterovs} (also see Section~\ref{reductiontoconvexoptwhenrissmall}).
{For hyperbolic spaces in particular, \citet{martinezrubio2021global} shows that we can achieve acceleration when $r \leq O(1)$.}

\item ($r = \Theta(\kappa)$)  In this regime, our results show we cannot achieve acceleration.  
\end{itemize}
}


\subsection{Comparison to literature: best known upper bounds} \label{comparisontolit}
Let us review the best known upper bounds for smooth g-convex optimization (see Appendix~\ref{literature} for a more complete discussion of the literature).
\citet{ahn2020nesterovs} provide an algorithm {which is strictly faster than RGD}, and {requires only $\tilde{O}(\sqrt{\kappa})$ queries for the computational task described in the introduction when $r \leq O(\frac{1}{\kappa^{3/4}})$. 
Intuitively this makes sense because Riemannian manifolds are locally Euclidean, so in a small enough ball the effects of curvature are negligible.  When $r$ is not small, the algorithm of~\citet{ahn2020nesterovs} requires $\tilde O(\kappa)$ gradient queries.

The guarantees for the algorithm provided by~\citet{ahn2020nesterovs} hold for Hadamard manifolds of bounded curvature.  For hyperbolic spaces in particular, \citet{martinezrubio2021global} improves upon these guarantees by providing an algorithm requiring $e^{\tilde{O}(r)} \sqrt{\kappa}$ queries to solve the computational task; in particular, this algorithm is accelerated when $r \leq {O}(1)$.

\section{Preliminaries and the ball-packing property} \label{prelims}
We introduce the tools used to prove the main results.
For an introduction to Riemannian manifolds see~\citep{lee2012smoothmanifolds,lee2018riemannian}, or~\citep{AMS08,boumal2020intromanifolds} for an optimization perspective.

\subsection{Hadamard manifolds}\label{hadamardmanssections}
Throughout, $\calM$ denotes a smooth manifold which has tangent bundle $\T \calM$ and tangent spaces $\T_x \calM$.
We equip $\calM$ with a Riemannian metric: a smoothly-varying inner product $\inner{\cdot}{\cdot}_x$ on each tangent space $\T_x \calM$.
Throughout, we drop the subscript and denote these inner products by $\inner{\cdot}{\cdot}$.
The metric allows us to define the gradient $\grad f(x) \in \T_x \calM$ and Hessian $\Hess f(x) \colon \T_x \calM \rightarrow \T_x \calM$ of the cost function $f$ at each point $x$~\citep[Ch.~3,~5]{boumal2020intromanifolds}.  
We write $\norm{v} = \sqrt{\inner{v}{v}}$ for $v \in \T_x \calM$ and $\norm{A}$ for the operator norm of a linear operator $A \colon \T_x \calM \rightarrow \T_y \calM$.  We use $I$ to denote the identity linear operator from $\T_x \calM$ to $\T_x \calM$.

The Riemannian metric gives $\calM$ a notion of distance $\dist$ and geodesics.
The closed (geodesic) ball of radius $r$ centered at $x \in \calM$ is $B(x, r) = \{y \in \calM : \dist(y, x) \leq r\}$.
The closed ball in $\T_{x} \calM$ centered at $g \in \T_{x} \calM$ with radius $r$ is $B_x(g, r) = \{s \in \T_x \calM : \norm{s-g} \leq r\}$.  

The metric also provides a notion of \emph{intrinsic curvature}.
We focus on Hadamard manifolds:  
\begin{definition} \label{defhadamardmanifold}
A Riemannian manifold $\calM$ is a Hadamard manifold if $\calM$ is complete, simply connected and has nonpositive sectional curvature everywhere.
\end{definition}
By the Cartan--Hadamard Theorem, all $d$-dimensional Hadamard manifolds $\calM$ are diffeomorphic to $\Rd$~\citep[Thm. 12.8]{lee2018riemannian}.  
The Hopf--Rinow Theorem implies that the exponential map $\exp \colon \T \calM \rightarrow \calM$ is well defined on the entire tangent bundle, and moreover every pair of points can be connected by a unique geodesic and this geodesic is minimal~\citep[Prop. 12.9]{lee2018riemannian}.  
This means that the inverse of the exponential map $\exp_{x}^{-1} \colon \calM \rightarrow \T_x\calM$ is well defined for all $x \in \calM$.
We use $P_{x \rightarrow y} \colon \T_x \calM \rightarrow \T_y \calM$ to denote parallel transport along the geodesic connecting $x$ and $y$.

The next lemma is a direct consequence of the hyperbolic law of cosines and Toponogov's triangle comparison theorem---see Appendix~\ref{geolemmasandcharofgconvexity}.
It expresses the fact that when the underlying space is negatively curved, geodesics diverge quickly.
Lemma~\ref{geodesicsdiverge} forms the basis of Lemma~\ref{lemmaNbig} (spaces with sufficient negative curvature satisfy the ball-packing property), which is the most important geometric fact underlying Theorems~\ref{cor1} and~\ref{cor3}.  A proof of Lemma~\ref{geodesicsdiverge} can be found in Appendix~\ref{Appgeodesicsdiverge}.
\begin{lemma}[Geodesics diverge]\label{geodesicsdiverge}
Let $v_1, v_2$ be two tangent vectors at $\xorigin$ on a Hadamard manifold $\calM$ with identical norms $s = \norm{v_1} = \norm{v_2}$ and forming an angle at least $\theta$.
If the sectional curvatures of $\calM$ are upper bounded by $\Kup < 0$ and $\theta = e^{1 - \frac{2}{3} s \sqrt{-\Kup}}$, then $\dist(z_1, z_2) \geq \frac{2}{3} s$ where $z_i = \exp_{\xorigin}(v_i)$ for $i=1, 2$.
\end{lemma}
It is instructive to compare this lemma to the Euclidean case, where the law of cosines implies $\norm{z_1-z_2}^2 \leq 2 s^2 - 2 s^2 \cos(\theta) = O(s^2 \theta^2)$.  Therefore, if $\theta = \Theta(e^{-s})$, then $\norm{z_1-z_2} = O(s e^{-s})$.

\subsection{The ball-packing property} \label{subsecballpackingproperty}

To prove the lower bound, we require our space to satisfy the following geometric property.

\begin{assumption}[Ball-packing property]
\label{assumptionNbigWeak}
There is a point $\xorigin \in \calM$, an $\tilde{r} > 0$ and a $\tilde{c} > 0$ such that for all $r \geq \tilde{r}$, there exist $N \geq e^{\tilde{c} r}$ points $z_1, \ldots, z_N$ in the ball $B(\xorigin, \frac{3}{4} r)$ so that 
all pairs of points are separated by a distance of at least $\frac{r}{2}$: 
$\dist(z_i, z_j) \geq \frac{r}{2}$ for all $i \neq j$.
We say $\calM$ satisfies the ball-packing property with $\tilde r, \tilde c$ and $\xorigin \in \calM$.
\end{assumption}

We think of the points $z_1, \ldots, z_N$ as centers of disjoint open balls of radius $\frac{r}{4}$ contained in $B(\xorigin, r)$.
No Euclidean space $\reals^d$ satisfies a ball-packing property as the volume of a ball of radius $r$ scales polynomially as $r^d$, not exponentially.
\begin{assumption}[Strong ball-packing property]
\label{assumptionNbigStrong}
There is an $\tilde{r} > 0$ and a $\tilde{c} > 0$ such that $\calM$ satisfies the ball-packing property with $\tilde r, \tilde c$ and every $\xorigin \in \calM$.
\end{assumption}

\begin{lemma} \label{lemmaNbig}
Let $d \geq 2$ and $\calM$ be a $d$-dimensional Hadamard manifold whose sectional curvatures are in the interval $(-\infty, \Kup]$ with $\Kup < 0$.  
Then $\calM$ satisfies the strong ball-packing property~\aref{assumptionNbigStrong} for $\tilde{r} = \frac{4}{\sqrt{-\Kup}}$ and $\tilde{c} = d \frac{\sqrt{-\Kup}}{8}$.
\end{lemma}
\begin{proof}
Let $\xorigin \in \calM$.
Let $r \geq \tilde{r}$ and let $s = \frac{3}{4} r$.
Let $\theta = e^{1 - \frac{2}{3} s \sqrt{-\Kup}}$.
Consider the sphere $\mathbb{S}^{d-1}_{\xorigin}(s) = \{v \in \T_{\xorigin} \calM : \norm{v} = s\}.$
We have $\theta \leq \frac{\pi}{2}$ because $s\sqrt{-\Kup} \geq 3$.
Therefore using $d \geq 2$ and a standard covering number argument adapted to our setting (see Lemma~\ref{packingonsphere} in Appendix~\ref{appendixNbig}), we find there exist 
$$N \geq \theta^{-(d-1)} = e^{(d-1)(\frac{2}{3} s \sqrt{-\Kup} -1)} \geq e^{\frac{1}{2}d(\frac{1}{2} r \sqrt{-\Kup} -1)} \geq e^{\frac{1}{8} d r \sqrt{-\Kup}}$$
tangent vectors $v_1, \ldots, v_N \in \mathbb{S}^{d-1}_{\xorigin}(s)$ such that the angle between vectors $v_i$ and $v_j$ is at least $\theta$ for all $i \neq j$.
Define $z_j = \exp_{\xorigin}(v_j)$ for $j = 1, 2, \ldots, N$.  Therefore, $z_j \in B(\xorigin, \frac{3}{4} r)$ for all $j$.
Moreover, $\dist(z_i, z_j) \geq \frac{2}{3} s = \frac{r}{2}$ for all $i \neq j$ owing to Lemma~\ref{geodesicsdiverge} (geodesics diverge). 
\end{proof}

If $\calM$ is a hyperbolic space, we can instead argue Lemma~\ref{lemmaNbig} using a simple volume argument, combined with the fact that the covering number is less than the packing number~\citep[Lem.~4.2.8]{vershynin_2018}.
However, that argument does not hold for spaces with nonconstant curvature because we would need a bound on the ratio of $\Klo$ to $\Kup$.  Lemma~\ref{lemmaNbig} does not make such an assumption.

The symmetric spaces $\mathcal{SLP}_n = \mathrm{SL}(n) / \mathrm{SO}(n)$ and $\calP_n = \reals \times \mathcal{SLP}_n$ are Hadamard manifolds which do not have strictly negative curvature, so we cannot apply Lemma~\ref{lemmaNbig}.
However, $\mathcal{SLP}_n$ contains an $(n-1)$-dimensional totally geodesic submanifold isometric to a hyperbolic space~\citep[Ch.~II.10]{bridsonmetric}.
This allows us to prove Lemma~\ref{lemmaPDmatrices} in Appendix~\ref{PDmatricesappendix}, where we also argue the best possible $\tilde{c}$ for $\mathcal{SLP}_n$ satisfies $\tilde c \leq O(n^{3/2}) = o(\dim(\mathcal{SLP}_n))$.
\begin{lemma} \label{lemmaPDmatrices}
%
For $n \geq 3$, both $\mathcal{SLP}_n$ and $\mathcal{P}_n$ satisfy the strong ball-packing property~\aref{assumptionNbigStrong} with $\tilde{r} = 8 \sqrt{2}, \tilde{c} = \frac{n-1}{16 \sqrt{2}}$, and $\mathcal{SLP}_2$ and $\mathcal{P}_2$ satisfy the strong ball-packing property with $\tilde{r} = 4 \sqrt{2}, \tilde{c} = \frac{1}{4 \sqrt{2}}.$
\end{lemma}

\subsection{Geodesic convexity}\label{gconvexfunctionsdetails}  

A subset $D$ of a Hadamard manifold $\calM$ is g-convex if the geodesic segment connecting each pair of points in $D$ is contained in $D$~\citep{udriste1994convex}.  Geodesic balls are g-convex.
We study special functions on g-convex sets (Lemma~\ref{muHesslemma} has other characterizations of g-convexity and smoothness).


\cutchunk{\begin{definition}
Let $\calM$ be a Hadamard manifold, and let $D \subseteq \calM$ be g-convex.
A function $f \colon D \rightarrow \reals$ is said to be $\mu$-strongly g-convex if for all $x, y \in D$ the univariate function $f \circ \gamma \colon [0, 1] \rightarrow \reals$ is $\mu$-strongly convex, where $\gamma \colon [0, 1] \rightarrow \reals$ is the geodesic with $\gamma(0) = x, \gamma(1) = y$.
\end{definition}}


\begin{definition} \label{definitionLsmoothness}
Let $f\colon \calM \rightarrow \reals$ be a twice continuously differentiable function on a Hadamard manifold $\calM$, and let $D$ be a g-convex subset of $\calM$.  We say (somewhat restrictively) that:
\begin{itemize}
\item $f$ is $\mu$-strongly g-convex in $D$ if $\Hess f(x) \succeq \mu I$ for all $x \in D$.  
\item $f$ is $L$-smooth in $D$ if $\norm{\Hess f(x)} \leq L$ for all $x \in D$.
\end{itemize}
(If $f$ is $L$-smooth in $D$, then $\norm{\grad f(x) - P_{y \rightarrow x} \grad f(y)} \leq L \dist(x, y)$ for all $x, y \in D$.)
\end{definition}


\begin{lemma} \citep[Lem. 2 in App. B]{alimisis2019continuoustime} \label{Lipschitzgradsquareddistance}
Let $\calM$ be a Hadamard manifold with sectional curvatures in $[\Klo, 0]$.
Fix $z \in \calM$, and let $f \colon \calM \rightarrow \reals, f(x) = \frac{1}{2} \dist(x, z)^2$.  Then $f$ is $C^{\infty}$, $\grad f(x) = - \exp_x^{-1}(z)$, and $f$ is $1$-strongly g-convex in $\calM$ and $L$-smooth  in $B(z, r)$ for any $r > 0$ with $L = \frac{r \sqrt{-\Klo}}{\tanh(r \sqrt{-\Klo})} \leq 1 + r \sqrt{-\Klo}$.
%
%
\end{lemma}

\section{Technical version of the main theorem and proof of key lemma} \label{maintheoremandproof}
We are now ready to prove our main technical theorem, from which Theorems~\ref{cor1} and~\ref{cor3} in the introduction follow.
For ease of exposition, we state and prove the following slightly simpler theorem in the main part of the paper.  
For this theorem, we assume the algorithm only receives gradient information (no function values), and the algorithm always makes queries in a bounded domain.
For the statement below, recall the definition of the function class $\mathcal{F}_{{\kappa}, r}^{\xorigin}(\calM)$ from Section~\ref{introproblemclass}.

\begin{theorem} \label{theoremnoacceleration}
Let $\calM$ be a Hadamard manifold of dimension $d \geq 2$ which satisfies the ball-packing property~\aref{assumptionNbigWeak} with constants $\tilde{r}, \tilde{c}$ and point $\xorigin \in \calM$.  
Also assume $\calM$ has sectional curvatures in the interval $[\Klo, 0]$ with $\Klo < 0$.
Let  $r \geq \max\big\{\tilde{r}, \frac{8}{\sqrt{-\Klo}}, \frac{4(d+2)}{\tilde{c}}\big\}$.  
Define ${\kappa} = 4 r\sqrt{-\Klo} + 3.$
Let $\calA$ be any deterministic algorithm which only makes gradient queries, and assume that $\calA$ always queries in $B(\xorigin, \mathscr{R})$, with $\mathscr{R} \geq r$.

Then there is a function $f \in \mathcal{F}_{\kappa, r}^{\xorigin}(\calM)$ with minimizer $x^*$ such that running $\calA$ on $f$ yields iterates $x_0, x_1, x_2, \ldots$ satisfying $\dist(x_{k}, x^*) \geq \frac{r}{4}$ for all $k = 0, 1, \ldots, T-1$, where
\begin{align} \label{rangefork}
T =  \Bigg\lfloor \frac{\frac{1}{2}\tilde{c} d^{-1} r}{\log\big(2000 \cdot \frac{1}{2} \tilde{c} d^{-1} r (3 \mathscr{R} \sqrt{-\Klo} + 2)\big)} \Bigg\rfloor.
\end{align}
%
%
%
\end{theorem}
\noindent In the theorem, observe that $\kappa = \Theta(r)$ and so $T = \tilde \Theta(r) = \tilde \Theta(\kappa)$ (assuming $\mathscr{R} = \mathrm{poly}(r)$).

In Appendices~\ref{fctvalues} and~\ref{apppermittingqueriesoutsidedomain}, we state and prove Theorem~\ref{maintheoremunboundedqueries}: an extension of Theorem~\ref{theoremnoacceleration} which provides a lower bound for algorithms which can also make function-value queries as well as unbounded queries.
Theorems~\ref{cor1} and~\ref{cor3} follow directly from Theorem~\ref{maintheoremunboundedqueries} and the ball-packing properties established in Lemmas~\ref{lemmaNbig} and~\ref{lemmaPDmatrices}---see Appendices~\ref{helper1} and~\ref{helper2} for the details.

To allow for algorithms which make unbounded queries, the high-level idea is to modify all hard instances $f$ from Theorem~\ref{theoremnoacceleration} so that $f(x) = \frac{1}{2} \dist(x, \xorigin)^2$ for $x \not \in B(\xorigin, \mathscr{R})$ (recall $\mathscr{R} \geq r$).  
This way, the algorithm gains no information by querying outside the ball $B(\xorigin, \mathscr{R})$.
On the other hand, we still want the hard functions $f$ to remain untouched in the ball $B(\xorigin, r)$.
In the region between radii $r$ and $\mathscr{R}$, we smoothly interpolate between these two choices of functions.
We show that we can choose $\mathscr{R}$ appropriately so that the lower bound $\tilde{\Omega}(r)$ still holds and the modified functions are still strongly g-convex.
Technically, we do this via a reduction, which is depicted in Figure~\ref{figredboundedqueries} in Appendix~\ref{apppermittingqueriesoutsidedomain} with additional details.

\subsection{Key lemma: Pi\`ece de r\'esistance} \label{sec3dot2}
The main ingredient to prove Theorem~\ref{theoremnoacceleration} is Lemma~\ref{keylemma} stated below.
At a high-level, we show that as long as the algorithm has made at most $T$ queries, the oracle can always answer these queries in such a way that there exist two cost functions consistent with these queries and yet whose minimizers are significantly far away from each other.
Let us make this more precise.

Consider a Hadamard manifold $\calM$ satisfying the ball-packing property~\aref{assumptionNbigWeak} with $\tilde{r}, \tilde{c} > 0$ and $\xorigin \in \calM$. 
Let $z_1, z_2, \ldots, z_N,$ with $N \geq e^{\tilde{c} r},$ be points in $B(\xorigin, \frac{3}{4} r), r \geq \tilde{r},$ so that all pairs of points are separated by a distance of at least $\frac{r}{2}$.
Let $\calA$ be a first-order optimization algorithm.
One can even give the list of points $z_1, \ldots, z_N$ to the algorithm designer.
The algorithm $\calA$ queries points $x_0, x_1, \ldots$ and our job (as the resisting oracle) is to choose gradients $g_0, g_1, \ldots$ to return to $\calA$.

At each iteration $k \geq 0$, we maintain a list of ``active candidate functions'' $f_{j, k} \colon \calM \rightarrow \reals$ indexed by $j \in A_k \subseteq \{1, \ldots, N\}$.  
The notation $A_k$ stands for ``active'' set at iteration $k$.
Each of the functions $f_{j, k}, j \in A_k,$ is differentiable, strongly g-convex, and has minimizer at $z_j$ with $z_j$ a distance of at least $\frac{r}{4}$ from all queried points.
Additionally, the functions $f_{j, k}, j \in A_k,$ are consistent with the $k$ gradient queries $(x_0, g_0), \ldots, (x_{k-1}, g_{k-1})$ answered so far, meaning $\grad f_{j, k}(x_m) = g_m$ for all $m < k$ and $j \in A_k$.
Therefore, any of the functions $f_{j, k}$, with $j \in A_k$, can be the actual function being optimized.
Hence, any of the minimizers $z_j$, with $j \in A_k$, can be the minimizer of the actual function being optimized.
As long as $A_k$ is nonempty, we can conclude that the algorithm $\calA$ has not queried a point within distance $\frac{r}{4}$ of the minimizer up to iteration $k$.

The next set of active candidate functions $\{f_{j, k+1} : j \in A_{k+1}\}$ is chosen by modifying the current set of active candidate functions: $f_{j, k+1} = f_{j, k} + h_{j, k}$.  The modifications $h_{j, k}$ and the set $A_{k+1} \subseteq A_k$ are chosen so that $\grad f_{j, k+1}(x_k) = g_k$, where $g_k$ is the gradient chosen by the resisting oracle to return to the algorithm in response to the query $x_k$.
Given the queries $x_0, x_1, \ldots, x_k$ made by the algorithm and the current active set $A_k$, the resisting oracle chooses $g_k \in \T_{x_k} \calM$ in such a way that the algorithm gains as little information about the location of $x^*$ as possible.
This amounts to choosing $g_k$ so that the cardinality of $A_{k+1}$ is as large as possible.
For example, if $\calM$ is a $d$-dimensional hyperbolic space of curvature $-1$, we show 
$\left|A_{k+1}\right| \geq \tilde \Omega(\left|A_k\right| / r^d).$
Since $\left|A_0\right| \geq e^{\Omega(d r)}$ due to the ball-packing lemma~\ref{lemmaNbig}, this allows us to conclude the desired lower bound.


\begin{lemma} \label{keylemma}
Let $\calM$ be a Hadamard manifold of dimension $d \geq 2$ with sectional curvatures in the interval $[\Klo, 0]$ and $\Klo < 0$.
Let $\xorigin \in \calM$, $r \geq \frac{8}{\sqrt{-\Klo}}$, $\mathscr{R} \geq r$.  Let $z_1, \ldots, z_N \in B(\xorigin, \frac{3}{4} r)$ be distinct points in $\calM$ such that $\dist(z_i, z_j) \geq \frac{r}{2}$ for all $i \neq j$.  
Define $A_0 = \{1, 2, \ldots, N\}$.
Let $\calA$ be any first-order algorithm which only makes gradient queries and only queries points in $B(\xorigin, \mathscr{R})$.
Finally, let $w\geq 1$ (this is a tuning parameter we will set later).

For every $k = 0, 1, 2, \ldots, \floor{2w},$ algorithm $\calA$ queries
$x_k = \calA_k((x_0, g_0), \ldots, (x_{k-1}, g_{k-1}))$
and there exists a tangent vector $g_k \in \T_{x_k} \calM$ and a set $A_{k+1} \subseteq A_k$ satisfying 
\begin{align} \label{lowerboundonAk}
\left|A_{k+1}\right| \geq \frac{\left|A_k\right| - 1}{(2000 w(3 \mathscr{R} \sqrt{-\Klo} + 2))^d}
\end{align}
such that for each $j \in A_{k+1}$ there is a $C^{\infty}$ function $f_{j, k+1} \colon \calM \rightarrow \reals$ of the form
\begin{align} \label{definitionofHjk}
f_{j, k+1}(x) = \frac{1}{2} \dist(x, z_j)^2 + H_{j, k+1}(x)
\end{align}
satisfying:
\begin{enumerate} [label=\textbf{L\arabic*}]

\item \label{prop1} $f_{j, k+1}$ is $(1 - \frac{k+1}{4 w})$-strongly g-convex in $\calM$ and $[2 r \sqrt{-\Klo} + 1 + \frac{k+1}{4 w}]$-smooth in $B(\xorigin, r)$;

\item $\grad f_{j, k+1}(z_j) = 0$ (hence in particular, the minimizer of $f_{j, k+1}$ is $z_j$); \label{prop2}

\item $\grad f_{j, k+1}(x_{m}) = g_{m}$ for $m = 0, 1, \ldots, k$ ($f_{j, k+1}$ is compatible with all queries); \label{prop3}

\item $\dist(x_{m}, z_j) \geq \frac{r}{4}$ for all $m = 0, 1, \ldots, k$; \label{prop4}

\item $\norm{\grad H_{j, k+1}(x)} \leq \frac{k+1}{4 w \sqrt{-\Klo}}$ and $\norm{\Hess H_{j, k+1}(x)} \leq \frac{k+1}{4 w}$ for all $x \in \calM$.\label{prop5}\footnote{This last property is helpful for the induction used to prove Lemma~\ref{keylemma}.  It is not explicitly used to prove Theorem~\ref{theoremnoacceleration}.}
\end{enumerate}
\end{lemma}

\begin{proof}[Proof of Theorem~\ref{theoremnoacceleration}]
Let us apply Lemma~\ref{keylemma} to $\calM$ and $\calA$.
Let the points $z_1, \ldots, z_N$ be provided by the ball-packing property so that $N\geq e^{\tilde{c} r}$.
Set $w = \tilde{c} d^{-1} r/4$ in Lemma~\ref{keylemma}.

It is easy to show by induction that inequality~\eqref{lowerboundonAk} along with $\left|A_0\right| \geq e^{\tilde{c} r}$ and $r \geq \frac{4(d+2)}{\tilde{c}}$ imply $\left|A_k\right| \geq 2$ for all $k \leq \min\{T, \floor{2 w}\} = T$.
For completeness, we give a short proof in Appendix~\ref{apptechnicalfactfromproofofbabytheorem}.

Since $A_T$ is nonempty, we can choose $j \in A_T$ and let $f = f_{j, T}$.
By property~\ref{prop1} and $T \leq 2 w$, $f$ is $\frac{1}{2}$-strongly g-convex in $\calM$ and $[2 r \sqrt{-\Klo} + \frac{3}{2}]$-smooth in $B(\xorigin, r)$.
Property~\ref{prop2} implies $f$ has minimizer $z_j$ which is contained in $B(\xorigin, \frac{3}{4}r)$.
Thus, $f$ is in $\mathcal{F}_{\kappa, r}^{\xorigin}(\calM)$ with $\kappa=4 r\sqrt{-\Klo} + 3$.

On the other hand, properties~\ref{prop3} and~\ref{prop4} imply running $\calA$ on $f$ produces iterates $x_0, \ldots, x_{T-1}$ satisfying $\dist(x_k, z_j) \geq \frac{r}{4}$ for all $k = 0, 1, \ldots, T-1$: all are far from the minimizer.
\end{proof}

%

\subsection{Proof of the key lemma~\ref{keylemma}} \label{proofoflemmaSec3}
We prove by induction on $k$ that there is a set $A_{k+1}$ and functions $f_{j, k+1}, j \in A_{k+1},$ satisfying the properties \ref{prop1}, \ref{prop2}, \ref{prop3}, \ref{prop4}, \ref{prop5}.
As we do this, we also construct the gradients $g_0, g_1, \ldots$, and we show that inequality~\eqref{lowerboundonAk} holds for each $k \geq 0$.

Our \textbf{base case} is $k+1 = 0$.
At the start (no queries), we simply define
\begin{align} \label{buildfj0}
f_{j, 0}(x) = \frac{1}{2} \dist(x, z_j)^2, \quad \forall j \in A_0 = \{1, \ldots, N\}.
\end{align}
Clearly $\grad f_{j, 0}(z_j) = 0$ (\ref{prop2} is satisfied), and $f_{j, 0}$ is $1$-strongly g-convex and $[2 r \sqrt{-\Klo} + 1]$-smooth in $B(\xorigin, r)$ by Lemma~\ref{Lipschitzgradsquareddistance} (\ref{prop1} is satisfied).
At iteration 0, all the functions $f_{j, 0}$ are trivially consistent with the set of past queries because there are no past queries (\ref{prop3} and \ref{prop4} are satisfied).  Finally, $H_{j, 0}$ is identically zero (\ref{prop5} is satisfied).

Now let us move on to the \textbf{inductive step}.  The remainder of this section is devoted to the inductive step.
We are at iteration $k \in [0, \floor{2 w})$, and there have been $k$ past queries at the points $x_0, \ldots, x_{k-1}$, along with the $k$ gradients $g_0, \ldots, g_{k-1}$ returned by the oracle.

The algorithm queries a point $x_k$.
If $k \geq 1$, let $x_{\ell}$ be a previous query point closest to $x_k$, i.e.,
\begin{align} \label{closestpoint}
x_{\ell} \in {\arg\min}_{x \in \{x_0, x_1, \ldots, x_{k-1}\}} \dist(x_k, x).
\end{align}
If $x_k = x_\ell$ (i.e., the algorithm repeats a query), just return $g_k = g_{\ell}$, take $A_{k+1} = A_k$, and we are done.
Otherwise, we can assume $x_k \neq x_\ell$.

By the inductive hypothesis (\textbf{IH}), we have a set $A_k$ such that for each $j \in A_k$ there is an infinitely differentable function $f_{j, k}$ for which:
\begin{enumerate} [label=\textbf{IH\arabic*}]
\item \label{IHprop1} $f_{j, k}$ is $(1 - \frac{k}{4 w})$-strongly g-convex in $\calM$ and $[2 r \sqrt{-\Klo} + 1 + \frac{k}{4 w}]$-smooth in $B(\xorigin, r)$;

\item $\grad f_{j, k}(z_j) = 0$; \label{IHprop2}

\item $\grad f_{j, k}(x_{m}) = g_{m}$ for $m = 0, \ldots, k-1$; \label{IHprop3}

\item $\dist(x_{m}, z_j) \geq \frac{r}{4}$ for all $m = 0, 1, \ldots, k-1$; \label{IHprop4}

\item $\norm{\grad H_{j, k}(x)} \leq \frac{k}{4 w\sqrt{-\Klo}}$ and $\norm{\Hess H_{j, k}(x)} \leq \frac{k}{4 w}$ for all $x \in \calM$. \label{IHprop5}
\end{enumerate}

We want to choose a large set $A_{k+1} \subseteq A_k$, and for each $j \in A_{k+1}$ we must construct $f_{j, k+1}$ as
\begin{align}\label{buildfjkplus1}
f_{j, k+1} = f_{j, k} + h_{j, k},
\end{align}
where $h_{j, k} \colon \calM \rightarrow \reals$ is an appropriately chosen function.
What properties do we want the functions $h_{j, k}$ to satisfy?
Compare the properties \ref{IHprop1}, \ref{IHprop2}, \ref{IHprop3}, \ref{IHprop4}, \ref{IHprop5} satisfied by $f_{j, k}$, with the properties \ref{prop1}, \ref{prop2}, \ref{prop3}, \ref{prop4}, \ref{prop5} we want $f_{j, k+1}$ to satisfy.
Let us look at each property.

\begin{itemize}
\item In order to ensure $f_{j, k+1}$ satisfies \ref{prop4}, we simply need to choose $A_{k+1}$ so that $\dist(x_k, z_j) \geq \frac{r}{4}$ for all $j \in A_{k+1}$.
Define
\begin{align} \label{definetildeAk}
\tilde{A}_k = \bigg\{j \in A_k : \dist(x_k, z_j) \geq \frac{r}{4}\bigg\}.
\end{align}
Since any pair of minimizers $z_i, z_j$ are separated by a distance of at least $r/2$, there is at most one $j \in A_k$ such that $\dist(x_k, z_j) < r/4$.
Therefore $|\tilde{A}_k| \geq |A_k| - 1$.
Below we define $A_{k+1}$ as a particular subset of $\tilde{A}_k$.

\item Let us look at property~\ref{prop3}.
If $k=0$, then $f_{j, k+1}$ is trivially consistent with the past queries (because there are none).
Assume $k\geq 1$.
In order for $f_{j, k+1}$ to remain consistent with the past queries $x_0, \ldots, x_{k-1}$, it is sufficient to enforce that the closed support\footnote{We use the notation $\mathrm{supp}(f)= \{x \in \calM : f(x) \neq 0\}$ to denote the support of a function $f \colon \calM \rightarrow \reals$, and $\overline{S}$ to denote the closure of the set $S$.} $\overline{\mathrm{supp}(h_{j, k})}$ does not contain $x_0, \ldots, x_{k-1}$.
Of course $h_{j, k}$ vanishes identically on the complement of its closed support $\calM \setminus \overline{\mathrm{supp}(h_{j, k})}$.
Further, $\calM \setminus \overline{\mathrm{supp}(h_{j, k})}$ is an open set, so 
$$\grad h_{j, k}(x) = 0, \quad \Hess h_{j, k}(x) = 0 \quad \forall x \in \calM \setminus \overline{\mathrm{supp}(h_{j, k})}.$$
Using $\grad f_{j, k+1} = \grad f_{j, k} + \grad h_{j, k}$ and the inductive hypothesis~\ref{IHprop3}, this ensures that $f_{j, k+1}$ is consistent with the queries $x_0, \ldots, x_{k-1}$.

In order to gain control of the gradient of $f_{j, k+1}$ at $x_k$, we also want the support of $h_{j, k}$ to contain $x_k$.
So using that $x_{\ell}$~\eqref{closestpoint} is a past query point closest to $x_k$, it is enough to enforce that the support of $h_{j, k}$ remains in the ball $B(x_k, \frac{1}{4} \dist(x_k, x_{\ell}))$.  
We are not done with~\ref{prop3} but let us move on for now.

\item In the previous item we saw that if $k\geq 1$, we want the support of $h_{j, k}$ to be contained in a ball centered at $x_k$ and whose radius is no more than $\frac{1}{4} \dist(x_k, x_{\ell})$.
For~\ref{prop2}, it is convenient to require that this radius is no more than $\frac{r}{8}$.
Precisely, we shall ensure that the support of $h_{j, k}$ is contained in the ball $B(x_k, \Rball^{(k)})$ where
\begin{align} \label{Rkball}
\Rball^{(0)} = \frac{r}{8}, \quad \quad \Rball^{(k)} = \min\Big\{\frac{1}{4}\dist(x_k, x_{\ell}), \frac{r}{8}\Big\} \quad \text{if } k\geq 1.
\end{align}

Let us now show that this choice of $\Rball^{(k)}$ guarantees~\ref{prop2}, i.e., $\grad f_{j, k+1}(z_j) = 0$.
We know $\grad f_{j, k}(z_j) = 0$.  Therefore, to satisfy \ref{prop2} it is sufficient to impose that the closed support of $h_{j, k}$ does not contain $z_j$.
We know that $\dist(x_k, z_j) \geq \frac{r}{4}$ for all $j \in \tilde{A}_{k}$.
Since $\Rball^{(k)} \leq \frac{r}{8}$, we indeed have $z_j \not \in B(x_k, \Rball^{(k)})$ for all $j \in \tilde{A}_{k}$.


\item Since $f_{j, k+1} = f_{j, k} + h_{j, k}$~\eqref{buildfjkplus1} and $H_{j, k+1} = H_{j, k} + h_{j, k}$ due to~\eqref{definitionofHjk},
$$\Hess f_{j, k} -\norm{\Hess h_{j, k}} I \preceq \Hess f_{j, k+1} \preceq \Hess f_{j, k} + \norm{\Hess h_{j, k}} I,$$
$$\norm{\grad H_{j, k+1}} \leq \norm{\grad H_{j, k}} + \norm{\grad h_{j, k}}, \quad \norm{\Hess H_{j, k+1}} \leq \norm{\Hess H_{j, k}} + \norm{\Hess h_{j, k}}.$$
Therefore (using Definition~\ref{definitionLsmoothness}), for $f_{j, k+1}$ to satisfy \ref{prop1} and $H_{j, k+1}$ to satisfy \ref{prop5}, it is enough to require $\norm{\Hess h_{j, k}(x)} \leq \frac{1}{4 w}$ and $\norm{\grad h_{j, k}(x)} \leq \frac{1}{4 w \sqrt{-\Klo}}$ for all $x \in \calM$.
\end{itemize}

Since we are looking for a function $h_{j, k}$ with support contained in a ball, we are looking to construct a \textit{bump function}, that is a $C^\infty$ function on $\calM$ whose closed support is compact.
In Appendix~\ref{appendixbumpfunction}, we state and prove Lemma~\ref{lemmabumpfcts} which, given a point $x_k \in \calM$ and a radius $\Rball > 0$, provides a family of bump functions 
\begin{align} \label{familyofbumpseq}
\{h_g \colon \calM \rightarrow \reals\}_{g \in B_{x_k}(0, w^{-1} g_{\mathrm{norm}}(\Rball))},
\end{align}
such that for each $g$ with $\norm{g} \leq w^{-1} g_{\mathrm{norm}}(\Rball)$, the function $h_g$ is supported in $B(x_k, \Rball)$ (property~\ref{BFprop3} in Lemma~\ref{lemmabumpfcts}), has $\grad h_g(x_k) = g$ (property \ref{BFprop2}), and in addition satisfies the bounds $\norm{\Hess h_{g}(x)} \leq \frac{1}{4 w}$ and $\norm{\grad h_{g}(x)} \leq \frac{1}{4 w \sqrt{-\Klo}}$ for all $x \in \calM$ (property \ref{BFprop4}).
Here $g_{\text{norm}} \colon [0, \infty) \rightarrow \reals$ is a certain univariate function with a simple explicit formula (equation~\eqref{gnormeq}).
(Remember that $B_{x_k}(0, w^{-1} g_{\mathrm{norm}}(\Rball)) \subseteq \T_{x_k}\calM$ denotes a Euclidean ball, see Section~\ref{hadamardmanssections}.)

So far we have not chosen $g_k \in \T_{x_k}\calM$.  However, using the family of bump functions~\eqref{familyofbumpseq}, we have shown for any choice of $g \in B_{x_k}(0, w^{-1} g_{\mathrm{norm}}(\Rball^{(k)}))$ and $j \in \tilde{A}_k$, the function $f_{j, k+1} = f_{j, k} + h_{g}$ satisfies properties \ref{prop1}, \ref{prop2}, \ref{prop4}, \ref{prop5}, as well as $\grad f_{j, k+1}(x_m) = g_m$ for $m = 0, 1, \ldots k-1$ (this is part of property~\ref{prop3}).
It remains to choose $g_k \in \T_{x_k} \calM$ and $A_{k+1} \subseteq \tilde{A}_k$ so that $\grad f_{j, k+1}(x_k) = g_k$ for all $j \in A_{k+1}$, and so that inequality~\eqref{lowerboundonAk} is satisfied.

Around each gradient $\grad f_{j, k}(x_k), j \in \tilde{A}_k$, there is a small ball $B_{j, k}$ which is the set of possible gradients of $f_{j, k+1} = f_{j, k} + h_g$ at $x_k$.
More precisely, the balls $B_{j, k} \subseteq \T_{x_k} \calM$ are defined by 
\begin{align} \label{definingBjk}
B_{j, k} = B_{x_k}(\grad f_{j, k}(x_k), w^{-1}g_{\mathrm{norm}}(\Rball^{(k)})), \quad \quad \forall j \in \tilde A_k.
\end{align}
Some of these balls overlap, and we want to choose $g_k \in \T_{x_k} M$ which simultaneously lies in as many of the balls as possible\cutchunk{---see Figure~\ref{figureballsBkandBjk}}.
This is because the oracle only gets to pick one $g_k$ and we would like $g_k$ to be compatible with as many $f_{j, k+1}, j \in \tilde{A}_k,$ as possible.  
Therefore, choose
\begin{align} \label{definegk}
g_k \in {\arg\max}_{g \in \T_{x_k} \calM} \left|\big\{j \in \tilde{A}_k : g \in B_{j, k}\big\}\right|.
\end{align}
Define $A_{k+1} = \{j \in \tilde A_k : g_k \in B_{j, k}\}.$
The number of balls $\{B_{j, k}\}$ which intersect at the common vector $g_k$ equals $\left|A_{k+1}\right|$.
For each $j \in A_{k+1}$ the vector $g_{j, k} = g_k - \grad f_{j, k}(x_k)$ is in $B_{x_k}(0, w^{-1}g_{\mathrm{norm}}(\Rball^{(k)}))$ and so
$$\grad (f_{j, k} + h_{g_{j, k}})(x_k) = \grad f_{j, k}(x_k) + \grad h_{g_{j, k}}(x_k)  = \grad f_{j, k}(x_k) + g_{j, k} = g_k$$
using property \ref{BFprop2} of Lemma~\ref{lemmabumpfcts}.
Therefore, defining $h_{j, k} = h_{g_{j, k}},$ we have 
$$\grad f_{j, k+1}(x_k) = \grad (f_{j, k} + h_{j, k})(x_k) = g_k,\quad \quad \text{for all $j \in A_{k+1}$.}$$

It remains to show inequality~\eqref{lowerboundonAk}, i.e., a large enough subset of the balls $B_{j, k}$ do indeed intersect at a common point.
For this, we use a geometric lemma (short proof in Appendix~\ref{appsimplegeolemma}).
\begin{lemma} \label{smallandlargeballslemma}
Consider $n$ closed balls $B_1, \ldots, B_n \subseteq \reals^d$ of radius $q$ each, and assume each of the balls is also contained in a larger closed ball $B$ of radius $r$: $B_j \subseteq B$ for all $j=1,\ldots, n$.  
Choose $g \in \arg \max_{y \in B} \left|\{j \in \{1, \ldots, n\} : y \in B_j\}\right|$ and let $A = \{j \in \{1, \ldots, n\} : g \in B_j\}$.
Then 
$$\left|A\right| \geq n \Vol(B_1) / \Vol(B) = n q^d / r^d.$$
\end{lemma}
\cutchunk{
\begin{figure}[h]
    \centering
    \includegraphics[width=0.7\textwidth]{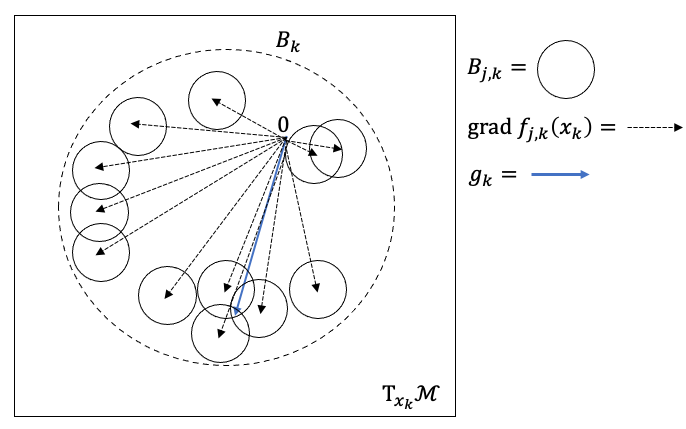}
    \caption{A depiction of how we choose the gradient $g_k$ in the proof of Lemma~\ref{keylemma} (Section~\ref{proofoflemmaSec3}).
    The balls $B_{j, k} \subset \T_{x_k}\calM$ are centered at the gradients $\{\grad f_{j, k}(x_k)\}_{j \in \tilde{A}_k}$ (the dashed arrows).
    The gradient $g_k \in \T_{x_k}\calM$ (the bolded blue arrow), is chosen in the intersection of as many balls $B_{j, k}$ as possible.
    The set of balls $B_{j, k}$ containing $g_k$ determines the subsequent active set $A_{k+1}$.  In the figure, $\left|A_{k+1}\right| = 3$.
    }
    \label{figureballsBkandBjk}
\end{figure}
}

To use Lemma~\ref{smallandlargeballslemma}, we need to find a ball $B_k \subseteq \T_{x_k}\calM$ containing all the balls $B_{j, k}, j \in \tilde A_k$, and we want the radius of $B_k$ to be small.
This is the last step of the proof.
Care has to be taken in bounding the radius of $B_k$ because the distance between $x_k$ and $x_\ell$~\eqref{closestpoint} can be arbitrarily small.
If $\dist(x_k, x_{\ell})$ is very small, then the radius of the balls $B_{j, k}$ is very small, and so we must show that the radius of $B_k$ is also sufficiently small for Lemma~\ref{smallandlargeballslemma} to be useful.
We upper bound the radius of the ball $B_k$ in the following two cases, showing that
\begin{align} \label{eqtobeshown}
\frac{\mathrm{Vol}(B_{j, k})}{\mathrm{Vol}(B_k)} \geq \frac{1}{(2000 w(3 \mathscr{R} \sqrt{-\Klo} + 2))^d}
\end{align}
holds in each case.
The most important part is to choose a good center for $B_k$:

\textbf{Case 1}: either $k=0$, or $k\geq 1$ and $\sqrt{-\Klo} \dist(x_k, x_{\ell}) > 4$.  
This captures the scenario where either there are no previous query points, or the algorithm queries $x_k$ not close to any previous query.  
{In this case, the overarching idea to upper bound the radius of $B_k$ is as follows: $f_{j, k}$ is a perturbed version of the squared distance function $x \mapsto \frac{1}{2} \dist(x, z_j)^2$.  
Therefore, the gradient of $f_{j, k}$ at $x_k$ (which is the center of $B_{j, k}$) is approximately $-\exp_{x_k}^{-1}(z_j)$ (see Lemma~\ref{Lipschitzgradsquareddistance}).
The points $\{z_j\}_{j \in \tilde{A}_k}$ are clustered around $\xorigin$ in a ball of radius $r$.  Therefore the vectors $\{-\exp_{x_k}^{-1}(z_j)\}_{j \in \tilde{A}_k}$ are clustered around $-\exp_{x_k}^{-1}(\xorigin)$.  
Consequently the same is true for the gradients $\{\grad f_{j, k}(x_k)\}_{j \in \tilde{A}_k}$.
We use this intuition to work out the details in Appendix~\ref{case1}.}

\textbf{Case 2}: $k \geq 1$ and $\sqrt{-\Klo} \dist(x_k, x_{\ell}) \leq 4$.
This captures the scenario where the algorithm queries $x_k$ close to a previous query.
{In this case, the overarching idea to bound the radius of a ball $B_k$ is as follows.  All the functions $f_{j, k}, j \in \tilde{A}_k$, have the same gradient at $x_\ell$, namely $g_\ell$.
Therefore, since $\dist(x_k, x_\ell)$ is small, gradient-Lipschitzness of the functions $f_{j, k}$ {(see Definition~\ref{definitionLsmoothness})} implies that the gradients $\{\grad f_{j, k}(x_k)\}_{j \in \tilde{A}_k}$ are all clustered around $P_{x_\ell \rightarrow x_k} g_\ell$ (the parallel transport of $g_\ell$ to $\T_{x_k} \calM$).
This intuition guides the details given in Appendix~\ref{case2}.}

After establishing inequality~\eqref{eqtobeshown}, we can use Lemma~\ref{smallandlargeballslemma} to show that $g_k$~\eqref{definegk} is contained in
\begin{align*}
\left|{A}_{k+1}\right| \geq |\tilde{A}_k| \frac{\Vol(B_{j, k})}{\Vol(B_k)} 
\geq \frac{|A_k| - 1}{(2000 w(3 \mathscr{R} \sqrt{-\Klo} + 2))^d}
\end{align*}
of the balls $B_{j, k}, j \in \tilde{A}_k$.
This concludes the inductive step, proving Lemma~\ref{keylemma}.

\acks{We thank David Mart\'inez-Rubio for helpful discussions and feedback on a version of this paper.}

\bibliography{../boumal_bibtex/boumal}

\appendix
\section{Further literature review} \label{literature}
{\citet{liu2017gconvexacceleration} were the first to claim acceleration on Riemannian manifolds.
However, their algorithm requires solving nonlinear equations at each iteration which a priori might be as difficult as the optimization problem itself.}

The results of~\citet{ahn2020nesterovs} mentioned previously are an improvement on the results of~\citet{zhang2018estimatesequence} who also show that acceleration is possible if the initial iterate is sufficiently close to the minimizer $x^*$.
\citet{ahn2020nesterovs}'s algorithm additionally converges globally.
\citet{jinsra2021riemannianacceleration} propose a framework for generating and analyzing eventually-accelerated algorithms; the algorithm of~\citet{ahn2020nesterovs} is an instance of this framework.

{\citet{martinezrubio2021global} presents algorithms for acceleration on spheres and hyperbolic spaces.  For hyperbolic spaces, \citet{martinezrubio2021global} proves the rates $e^{\tilde O(r)} \sqrt{\kappa}$ for the strongly g-convex case, and $e^{\tilde O(r)} \frac{1}{\sqrt{\epsilon}}$ for the nonstrongly g-convex case (to find a point $x$ satisfying $f(x) - f(x^*) \leq \epsilon \cdot \frac{1}{2} L r^2$).
The key idea consists of pulling back the optimization problem to a vector space via a geodesic map; the pullback satisfies a relaxed notion of convexity.  
This idea is similar to the method of trivializations, introduced in~\citep{pmlr-v97-lezcano-casado19a} and applied to momentum methods in~\citep{lezcanocasado2020adaptive}.}

\citet{alimisis2021momentum} tackle the problem of acceleration on the class of smooth nonstrongly g-convex functions.  
In certain scenarios (when $r$ and curvature are sufficiently small), their algorithm outperforms RGD; {however, in general their algorithm requires $O(\frac{r^2}{\epsilon})$ iterations to solve the problem, which is not better than RGD.}
Nevertheless, the experimental results of~\citet{alimisis2021momentum} show promise.  \citet{huang2021sparsepca} develop an algorithm for Riemannian optimization based on FISTA~\citep{beck2009fast} which also demonstrates promising experimental results.

\citet{alimisis2019continuoustime} construct an ordinary differential equation (ODE) to model a Riemmanian version of Nesterov's accelerated gradient method.
They prove that this ODE achieves an accelerated rate.  
It is unclear whether the discretization of this ODE preserves a similar acceleration.
Recently, techniques from dynamical systems and symplectic geometry have been used to derive ODEs~\citep{duruisseaux2021ODEacceleration} and discretize such ODEs to obtain algorithms on Riemannian manifolds~\citep{franca2021b,franca2021a}.
It is also unclear whether such algorithms achieve acceleration.


In stark contrast to the results presented in this paper, on non-convex functions it is possible to achieve acceleration \emph{for finding (first- and second-order) critical points} on Riemannian manifolds, even negatively curved manifolds~\citep{criscitiello2020accelerated}.


One can also {sometimes} model non-Euclidean geometries using a Bregman distance function; such geometries have different properties than Riemannian geometries.  {For example, the functions in consideration are still convex in the Euclidean sense (unlike g-convex functions); the Bregman geometry alters the notion smoothness and conditioning of these functions.}  \citet{acceleratingmirrordescentisnotpossible} recently showed that acceleration is not possible in this setting.  The techniques they use and the key geometric obstructions to acceleration are significantly different from the Riemannian setting.



%

\section{Useful geometric propositions, and characterizations of g-convexity and smoothness} \label{geolemmasandcharofgconvexity}
In the appendices, we use the following geometric propositions, which are consequences of the Euclidean law of cosines, the hyperbolic law of cosines~\citep[Thm.~3.5.3]{ratcliffe2019hyperbolic}, and Toponogov's triangle comparison theorem (see~\citep[Thm.~11.10]{lee2018riemannian}, ~\citep[Sec.~6.5]{buragoburago}, or~\citep[Thm.~8.13.3]{alexander2019alexandrov}).
In the appendices, we also use the equivalent characterizations of $\mu$-strong g-convexity and $L$-smoothness given in Lemma~\ref{muHesslemma}.

\begin{proposition} \label{TopogonovEuclidean}
Let $\calM$ be a Hadamard manifold.
Let $xyz$ be a geodesic triangle of $\calM$ with vertices $x, y, z\in \calM$ and side lengths $\dist(y, z) = a, \dist(x, z) = b, \dist(x, y) = c$.  
Also let the angle at $x$ be $\alpha$, i.e., $\alpha = \arccos\Big(\frac{\inner{\exp_x^{-1}(y)}{\exp_x^{-1}(z)}}{\dist(x, y) \dist(x, z)}\Big)$.
Then
$a^2 \geq b^2 + c^2 - 2 b c \cos(\alpha)$~\citep[Prop.~12.10]{lee2018riemannian}.
Equivalently, 
$$\dist(y, z)^2 \geq \dist(x, z)^2 + \dist(x, y)^2 - 2 \inner{\exp_{x}^{-1}(y)}{\exp_{x}^{-1}(z)} = \norm{\exp_{x}^{-1}(y) - \exp_x^{-1}(z)}^2.$$
\end{proposition}

\begin{proposition} \label{TopogonovHyperbolicBoundedAbove}
Consider the same setting as Proposition~\ref{TopogonovEuclidean}.
In addition, assume the sectional curvatures of $\calM$ are in the interval $(-\infty, \Kup]$ with $\Kup < 0$.
Then
$$\cosh(a \sqrt{-\Kup}) \geq \cosh(b \sqrt{-\Kup}) \cosh(c \sqrt{-\Kup}) - \sinh(b \sqrt{-\Kup}) \sinh(c \sqrt{-\Kup}) \cos(\alpha).$$
\end{proposition}

\begin{lemma} \label{muHesslemma}
Let $\calM$ be a Hadamard manifold, and $D \subseteq \calM$ be a g-convex set.
Let $f \colon \calM \rightarrow \reals$ be {twice continuously} differentiable.
With reference to Definition~\ref{definitionLsmoothness}:
\begin{itemize}
\item If $f$ is $\mu$-strongly g-convex in $D$ then $f(y) \geq f(x) + \inner{\grad f(x)}{\exp_x^{-1}(y)} + \frac{\mu}{2} \dist(x, y)^2$ for all $x, y \in D$.
\item If $f$ is $L$-smooth in $D$ then $\left|f(y) - f(x) -  \inner{\grad f(x)}{\exp_x^{-1}(y)}\right| \leq \frac{L}{2} \dist(x, y)^2$ for all $x, y \in D$.
\end{itemize}
\end{lemma}

\section{Proof of the simple geometric lemma~\ref{smallandlargeballslemma}} \label{appsimplegeolemma}
For each $x \in B$, let $N(x)$ be the number of smaller balls which contain $x$:
$$N(x) = \left|\{j \in \{1, \ldots, n\} : x \in B_j\}\right|.$$
Therefore, $g \in \arg \max_{y \in B} N(y)$.
The sum of the volumes of the smaller balls is 
$$n \Vol(B_1) = \int_{x \in B} N(x) dV(x) \leq \int_{x \in B} \Big(\max_{y \in B} N(y)\Big) dV(x) = \Big(\max_{y \in B} N(y)\Big) \Vol(B).$$
So $\left|A\right| = \max_{y \in B} N(y) \geq n \Vol(B_1) / \Vol(B) = n q^d / r^d.$

\section{Details for Cases 1 and 2 in proof of the key lemma~\ref{keylemma}}
\subsection{Case 1: $x_k$ is not close to any previous query point} \label{case1}
By Proposition~\ref{TopogonovEuclidean}, nonpositive curvature yields $\norm{\exp_{x_k}^{-1}(z_j) - \exp_{x_k}^{-1}(\xorigin)} \leq \dist(z_j, \xorigin) \leq r.$
By the inductive hypothesis and the assumptions $k \leq 2 w$ and $r\sqrt{-\Klo} \geq 8$, 
$$\norm{\grad H_{j, k}(x_k)} \leq k \frac{1}{4 w \sqrt{-\Klo}} \leq \frac{1}{2 \sqrt{-\Klo}} \leq \frac{r}{2}.$$
Therefore, using the definition of $H_{j, k}$ (equation~\eqref{definitionofHjk}), 
\begin{align*}
\norm{\grad f_{j, k}(x_k) - \exp_{x_k}^{-1}(\xorigin)} &= \norm{\exp_{x_k}^{-1}(z_j) + \grad H_{j, k}(x_k) - \exp_{x_k}^{-1}(\xorigin)}
\leq \frac{3 r}{2}.
\end{align*}

Since $g_{\mathrm{norm}}(\cdot)$~\eqref{gnormeq} from Lemma~\ref{lemmabumpfcts} is increasing on $[0, \infty)$,
$$w^{-1}g_{\mathrm{norm}}(\Rball^{(k)}) \leq w^{-1} \lim_{t \rightarrow \infty} g_{\mathrm{norm}}(t) = \frac{w^{-1}}{9 e^{1/3} \sqrt{-\Klo}} \leq \frac{1}{9 e^{1/3} \sqrt{-\Klo}} \leq \frac{r}{9 e^{1/3}} \leq \frac{r}{2}.$$
So each $B_{j, k}$ is contained in the ball 
$$B_{x_k}\bigg(\exp_{x_k}^{-1}(\xorigin), \frac{3r}{2} + w^{-1} g_{\mathrm{norm}}(\Rball^{(k)})\bigg) \subseteq B_{x_k}(\exp_{x_k}^{-1}(\xorigin), 2 r).$$
Defining, $B_k = B_{x_k}(\exp_{x_k}^{-1}(\xorigin), 2 r)$, we have
$\frac{\Vol(B_{j, k})}{\Vol(B_k)} = \frac{w^{-d} g_{\mathrm{norm}}(\Rball^{(k)})^d}{(2 r)^d}.$

Now assume $k=0$ or $\sqrt{-\Klo} \dist(x_k, x_{\ell}) > 4$, where $x_\ell$ is defined in equation~\eqref{closestpoint}.
Using $r\sqrt{-\Klo} \geq 8$ and the definition of $\Rball^{(k)}$~\eqref{Rkball}, this assumption implies $\Rball^{(k)} \geq 1/\sqrt{-\Klo}$.
Thus,
\begin{align*}
g_{\mathrm{norm}}(\Rball^{(k)}) \geq g_{\mathrm{norm}}(1 / \sqrt{-\Klo})
&\geq \frac{8/\sqrt{-\Klo}}{e^{1/3}(1485+72)} \geq \frac{1}{300 \sqrt{-\Klo}}.
\end{align*}
We conclude
\begin{align} \label{ratioofgradientvolumesCase1}
\frac{\Vol(B_{j, k})}{\Vol(B_k)} &= \frac{w^{-d} g_{\mathrm{norm}}(\Rball^{(k)})^d}{(2 r)^d}
\geq \frac{w^{-d}}{(300 \sqrt{-\Klo})^d \cdot (2 r)^d} = \frac{1}{(600 w r \sqrt{-\Klo})^d}.
\end{align}

\subsection{Case 2: $x_k$ is close to a previous query point} \label{case2}
Assume $\sqrt{-\Klo} \dist(x_k, x_{\ell}) \leq 4$, where $x_{\ell}$ is defined in equation~\eqref{closestpoint}.
Since $r \geq 8 / \sqrt{-\Klo}$, we have
$\Rball^{(k)} = \frac{1}{4} \dist(x_k, x_{\ell}).$

By assumption, $\dist(x_k, \xorigin) \leq \mathscr{R}$.
Since $\dist(z_j, \xorigin) \leq r$ and $\dist(x_k, x_\ell) \leq 4 / \sqrt{-\Klo} \leq r / 2$, the triangle inequality implies $\dist(x_k, z_j)$ and $\dist(x_\ell, z_j)$ are both at most $3 \mathscr{R}$.
The inductive hypothesis~\ref{IHprop5} implies $f_{j, k}(x) = \frac{1}{2} \dist(x, z_j)^2 + H_{j, k}(x)$ and $\norm{\Hess H_{j, k}(x)} \leq \frac{k}{4 w} \leq \frac{1}{2}$ for all $x \in \calM$, using $k \leq 2 w$.
So by Lemma~\ref{Lipschitzgradsquareddistance}, 
\begin{align} \label{Case2boundonhessianoffjk}
\norm{\Hess f_{j, k}(x)} \leq \max\{\dist(x_k, z_j), \dist(x_\ell, z_j)\} \sqrt{-\Klo} + \frac{3}{2} \leq 3 \mathscr{R} \sqrt{-\Klo} + \frac{3}{2}
\end{align}
for all $x \in B(z_j, \max\{\dist(x_k, z_j), \dist(x_\ell, z_j)\})$.
Additionally, the inductive hypothesis~\ref{IHprop3} implies $\grad f_{j, k}(x_{\ell}) = g_{\ell}$ for all $j \in \tilde A_k$.
Therefore, by Definition~\ref{definitionLsmoothness}: 
\begin{align*}
\norm{\grad f_{j, k}(x_k) - P_{x_\ell \rightarrow x_k} g_{\ell}}
&\leq \Big(3 \mathscr{R} \sqrt{-\Klo} + \frac{3}{2}\Big) \dist(x_k, x_{\ell}) \quad \text{for all $j \in \tilde A_k$}.
\end{align*}
We have shown that all the gradients $\grad f_{j, k}(x_k), j \in \tilde{A}_k,$ are contained in a ball in $\T_{x_k} \calM$ centered at $P_{x_\ell \rightarrow x_k} g_{\ell}$ with radius $(3 \mathscr{R} \sqrt{-\Klo} + \frac{3}{2}) \dist(x_k, x_{\ell})$.

Recall the definition of the balls $B_{j, k}$ defined in equation~\eqref{definingBjk}.
We conclude all the balls $B_{j, k}, j \in \tilde A_k$, are contained in a ball $B_k \subseteq \T_{x_k} \calM$ centered at $P_{x_\ell \rightarrow x_k} g_{\ell}$ with radius 
\begin{align*}
&\Big(3 \mathscr{R} \sqrt{-\Klo} + \frac{3}{2}\Big) \dist(x_k, x_{\ell}) + w^{-1} g_{\mathrm{norm}}\Big(\frac{1}{4} \dist(x_k, x_\ell)\Big) 
\leq (3 \mathscr{R} \sqrt{-\Klo} + 2) \dist(x_k, x_{\ell}),
\end{align*}
using that $w \geq 1$ and $g_{\mathrm{norm}}\Big(\frac{1}{4} \dist(x_k, x_\ell)\Big) \leq \frac{8 \dist(x_k, x_\ell) / 4}{e^{1/3}(1485)}  \leq \frac{\dist(x_k, x_{\ell})}{1000}.$
Therefore,
\begin{equation} \label{ratioofgradientvolumesCase2modded}
\begin{split}
\frac{\Vol(B_{j, k})}{\Vol(B_k)} &= \frac{w^{-d} g_{\mathrm{norm}}(\dist(x_k, x_{\ell}) / 4)^d}{((3 \mathscr{R} \sqrt{-\Klo} + 2) \dist(x_k, x_{\ell}))^d} \geq \frac{w^{-d} \dist(x_k, x_{\ell})^d}{(2000(3 \mathscr{R} \sqrt{-\Klo} + 2) \dist(x_k, x_{\ell}))^d} \\
&= \frac{1}{(2000 w(3 \mathscr{R} \sqrt{-\Klo} + 2))^d},
\end{split}
\end{equation}
using
$g_{\mathrm{norm}}\Big(\frac{1}{4}\dist(x_k, x_\ell)\Big) \geq \frac{2 \dist(x_k, x_\ell)}{e^{1/3}(1485 + 72)} \geq \frac{\dist(x_k, x_\ell)}{2000},$
which is due to $\sqrt{-\Klo} \dist(x_k, x_{\ell}) \leq 4$.


\section{Bump functions} \label{appendixbumpfunction}
\begin{lemma}[Family of bump functions] \label{lemmabumpfcts}
Let $\calM$ be a Hadamard manifold with sectional curvatures in the interval $[\Klo, 0]$ with $\Klo < 0$.
Let $\Rball > 0, w > 0$, $x_k \in \calM$. 
Define 
\begin{align} \label{defoffctA}
a \colon [0, \infty) \rightarrow \reals, \quad \quad a(R) = \frac{R}{4(4\sqrt{-\Klo} + 55 / R)},
\end{align}
\begin{align} \label{gnormeq}
g_{\mathrm{norm}} \colon [0, \infty) \rightarrow \reals, \quad \quad g_{\mathrm{norm}}(R) = \frac{8 R}{e^{1/3}(1485 + 72 R \sqrt{-\Klo})}.
\end{align}
There is a family of functions $\{h_g \colon \calM \rightarrow \reals\}$, indexed by $g \in B_{x_k}(0, w^{-1} g_{\mathrm{norm}}(\Rball))$,
satisfying for each $g \in B_{x_k}(0, w^{-1} g_{\mathrm{norm}}(\Rball)) \subseteq \T_{x_k} \calM$:
\begin{enumerate} [label=\textbf{BF\arabic*}]
\item $\grad h_g(x_k) = g$; \label{BFprop2}
\item the support of each $h_g$ is contained in $B(x_k, \Rball)$; \label{BFprop3}
\item $\norm{\grad h_g(x)} \leq \frac{1}{4 w \sqrt{-\Klo}}$, $\norm{\Hess h_g(x)} \leq \frac{1}{4 w}$ for all $x \in \calM$; \label{BFprop4}

\item $h_g(x_k) = \frac{3}{8} \norm{g} {g}_{\mathrm{norm}}^{-1}\big(w \norm{g}\big)$, and $\left|h_g(x)\right| \leq w^{-1} a(\Rball)$ for all $x \in \calM$.
 \label{BFprop1}
\end{enumerate}
\end{lemma}
\subsection{Proof of Lemma \ref{lemmabumpfcts}} \label{appendixbumpfunctionmain}
Define $\phi_R \colon \reals \rightarrow \reals$ by
$$\phi_R(t) = e\cdot \exp\bigg(-1/\Big(1-\frac{2t}{R^2}\Big)\bigg) = \exp(2 t / (2 t - R^2)) \quad \quad \text{for $t \in (-\infty, R^2/2)$},$$
and $\phi_R(t) = 0$ elsewhere.  The function $\phi_R\colon \reals \rightarrow \reals$ is $C^\infty$~\citep[Lem.~2.20]{lee2012smoothmanifolds}.
We consider bump functions supported in $B(p, R)$ of the form $h(x) = a \cdot \phi_R(\dist(x, p)^2/2)$ for $a \in \reals.$
As a composition of $C^{\infty}$ functions, these bump functions are also $C^\infty$.
\begin{remark}
Since $\dist(x, p)^2/2 \geq 0$, the values of $\phi_R$ on $(-\infty, 0)$ are irrelevant.
All that matters is that $\phi_R$ is infinitely differentiable in a neighborhood of the origin.
\end{remark}
We have $\phi(R^2/8) = e^{-1/3}$, and for $t \in (-\infty, R^2/2)$
$$\phi_R'(t) = - \phi_R(t) \cdot 2R^2 / (R^2 - 2 t)^2, \quad \quad \phi_R''(t) = \phi_R(t) \cdot 4 R^2 (4 t - R^2) / (R^2 - 2 t)^4.$$
We have partitioned the proof of Lemma~\ref{lemmabumpfcts} into several subsections: ~\ref{bumpfunctionconstrictionproof},~\ref{boundinggradofbumpfunctionsproof} and~\ref{boudinghessianofbumpfunctionsproof}.

\subsubsection{Bump function construction} \label{bumpfunctionconstrictionproof}
For each $p \in B(x_k, \Rball/3)$: let $R = 2 \dist(x_k, p)$ and define $\tilde{h}_{x_k, p} : \calM \rightarrow \reals$ by 
$$\tilde{h}_{x_k, p}(x) = w^{-1} a(R) \phi_R(\dist(x, p)^2 / 2).$$ 
By construction, $\tilde{h}_{x_k, p}$ is supported in the closed ball $B(p, R) \subseteq B(x_k, \Rball)$.  

We have $$\grad \tilde{h}_{x_k, p}(x) = - w^{-1} a(R) \phi_R'(\dist(x, p)^2/2) \exp_{x}^{-1}(p).$$  So using that $\norm{\exp_{x_k}^{-1}(p)} = \dist(x_k, p) = R/2,$
\begin{equation} \label{usedbelow}
\begin{split}
\norm{\grad \tilde{h}_{x_k, p}(x_k)} &= w^{-1} a(R) |\phi_R'(R^2/8)| R/2 
\\ &= w^{-1} a(R) \phi(R^2/8) \frac{32}{9 R^2} R/2 = w^{-1} e^{-1/3} a(R) \frac{16}{9 R} \\
&= w^{-1} \frac{4 e^{-1/3}}{9 R(4\sqrt{-\Klo}/R + 55 / R^2)} = w^{-1} g_{\mathrm{norm}}\Big(\frac{3}{2} R\Big)
\end{split}
\end{equation}
where $R$ can take any value in $[0, 2 \Rball / 3]$.
Since the function ${g}_{\mathrm{norm}}$ is strictly increasing on $[0, \infty)$ and ${g}_{\mathrm{norm}}(0) = 0$, we see that $\norm{\grad \tilde{h}_{x_k, p}(x_k)}$ takes all values in the interval $[0, w^{-1} g_{\mathrm{norm}}(\Rball)]$ (as $p$ varies).

On the other hand, $\grad \tilde{h}_{x_k, p}(x_k) = \norm{\grad \tilde{h}_{x_k, p}(x_k)} \frac{\exp_{x_k}^{-1}(p)}{\norm{{\exp_{x_k}^{-1}(p)}}}.$
Therefore,
$$\{\grad \tilde{h}_{x_k, p}(x_k) : p \in B(x_k, \Rball/3)\} = B_{x_k}(0, w^{-1} g_{\mathrm{norm}}( \Rball)).$$
More precisely, for each $g \in B_{x_k}(0, w^{-1} g_{\mathrm{norm}}( \Rball))$ there is exactly one $p \in B(x_k, \Rball/3)$ such that $g = \grad \tilde{h}_{x_k, p}(x_k)$, and vice versa.  Finally, define 
$$h_{\grad \tilde{h}_{x_k, p}(x_k)} := \tilde{h}_{x_k, p} \quad \forall p \in B(x_k, \Rball/3).$$
This defines the family of functions in Lemma~\ref{lemmabumpfcts}, and also establishes property~\ref{BFprop2}.
By construction, property~\ref{BFprop3} is also satisfied.

We calculate
\begin{align*}
h_{\grad \tilde{h}_{x_k, p}(x_k)}(x_k) &= \tilde{h}_{x_k, p}(x_k) = w^{-1} a(R) \phi_R(\dist(x_k, p)^2 / 2) = w^{-1} a(R) \phi_R(R^2 / 8) \\
&= w^{-1} a(R) e^{-1/3} = \norm{\grad \tilde{h}_{x_k, p}(x_k)} \frac{9 R}{16} \\
& = \norm{\grad \tilde{h}_{x_k, p}(x_k)} \frac{9}{16} \cdot \frac{2}{3} {g}_{\mathrm{norm}}^{-1}\bigg(w \norm{\grad \tilde{h}_{x_k, p}(x_k)}\bigg)
\end{align*}
where we used equation~\eqref{usedbelow} for the last two equalities.
This shows the first part of property~\ref{BFprop1}.

\subsubsection{Bounding the function values and gradients of the bump functions} \label{boundinggradofbumpfunctionsproof}
The maximum of $\left|\tilde{h}_{x_k, p}(x)\right|$ is attained when $x = p$ and equals $w^{-1} a(R)$, which is at most $w^{-1} a(2 \Rball / 3) \leq w^{-1} a(\Rball)$.
This shows the second part of property~\ref{BFprop1}.

By Section~\ref{bumpfunctionconstrictionproof}, we know that
$$\grad \tilde{h}_{x_k, p}(x) = - w^{-1} a(R) \phi_R'(\dist(x, p)^2/2) \exp_{x}^{-1}(p)$$
if $\dist(x, p) \leq R$ and $\grad \tilde{h}_{x_k, p}(x) = 0$ otherwise.  Therefore, for any $x\in\calM$,
\begin{align*}
\norm{\grad \tilde{h}_{x_k, p}(x)} &\leq w^{-1} a(R) \max_{t \in [0, R]} \left|\phi_R'(t^2/2)\right| t \\
&= w^{-1} a(R) \max_{t \in [0, R]} t \cdot \phi_R(t^2/2) \cdot 2R^2 / (R^2 -  t^2)^2
\end{align*}
It is easy to see that the maximizer of this problem is $t = R / 3^{1/4}$, which yields
\begin{align*}
\norm{\grad \tilde{h}_{x_k, p}(x)} &\leq w^{-1} a(R) R^{-1} \sqrt{36 + 21 \sqrt{3}} e^{-1/2-\sqrt{3}/2} \leq 3 w^{-1} a(R) R^{-1} \\
& \leq w^{-1} \frac{R}{55 + 4 R \sqrt{-\Klo}} \leq \frac{1}{4 w \sqrt{-\Klo}}.
\end{align*}
This proves the first part of property~\ref{BFprop4}.

\subsubsection{Bounding the Hessian of the bump functions} \label{boudinghessianofbumpfunctionsproof}
For $v \in \T_x M$ and $\dist(x, p) \leq R$, we have
\begin{align*}
\langle v, \Hess \tilde{h}_{x_k, p}(x) v \rangle &= w^{-1} a(R) \phi_R''(t^2/2) \langle v, -\exp_x^{-1}(p) \rangle^2 &&+ w^{-1} a(R) \phi_R'(t^2/2) \langle v, \mathscr{H}(x) v\rangle 
\\ &= (\text{term 1}) &&+ (\text{term 2})
\end{align*}
where $t = \dist(x, p)$ and $\mathscr{H}(x)$ is the Hessian of the function $x \mapsto \dist(x, p)^2/2$.

We have $t = \dist(x, p) = \norm{\exp_x^{-1}(p)}$ and 
$$\norm{\mathscr{H}(x)} \leq 1 + \dist(x, p) \sqrt{-\Klo} = 1 + t \sqrt{-\Klo},$$ 
by Lemma~\ref{Lipschitzgradsquareddistance}.
So for $t \in [0, R]$ we have that
\begin{align*}
\left|\text{term 1}\right| &\leq (w^{-1} a(R) \norm{v}^2)\bigg[\phi(t^2/2) \frac{4R^2 \left| 2 t^2 - R^2 \right|}{(R^2-t^2)^4} t^2\bigg] 
\leq (w^{-1} a(R) \norm{v}^2)\bigg[4 R^6 \frac{\phi(t^2/2)}{(R^2-t^2)^4}\bigg] \\
&\leq (w^{-1} a(R) \norm{v}^2)\bigg[4 R^6 \cdot 256 e^{-3} / R^8\bigg] \leq (w^{-1} a(R) \norm{v}^2)\bigg[51 / R^2\bigg], \text{and}
\end{align*}
\begin{align*}
\left|\text{term 2}\right| &\leq (w^{-1} a(R) \norm{v}^2) \bigg[\phi(t^2/2) (1+t \sqrt{-\Klo}) \frac{2 R^2}{(R^2 - t^2)^2}\bigg] \\
&\leq (w^{-1} a(R) \norm{v}^2)\bigg[2 (1+R \sqrt{-\Klo}) R^2 \frac{\phi(t^2/2)}{(R^2 - t^2)^2}\bigg] \\
&\leq (w^{-1} a(R) \norm{v}^2)\bigg[2 (1+R \sqrt{-\Klo}) R^2 \cdot 4 e^{-1} / R^4\bigg] \\
&\leq (w^{-1} a(R) \norm{v}^2) \bigg[4 \sqrt{-\Klo} / R + 4 / R^2\bigg].
\end{align*}
Of course $\|\Hess \tilde{h}_{x_k, p}(x)\| = 0$ for $x \not \in B(p, R)$.  So for all $x$, we have
$\|\Hess \tilde{h}_{x_k, p}(x)\| \leq w^{-1} a(R) (4\sqrt{-\Klo}/R + 55 / R^2) = \frac{1}{4 w}.$
We have proven the second part of property~\ref{BFprop4}.

\subsection{Bump functions for function-value queries} \label{bumpfunctionfctvalqueriesproof}
In this section, we construct a family of bump functions parametrized by both function value and gradient, as explained in Appendix~\ref{fctvalues}.
To do this, we first prove Lemma~\ref{lemmabumpfunctionsfunctionvalues}.
Then we prove Lemma~\ref{bumpfctsfctvalsandgrads} (stated in Appendix~\ref{fctvalues}).

\begin{lemma} \label{lemmabumpfunctionsfunctionvalues}
Let $\calM$ be a Hadamard manifold with sectional curvatures bounded below by $\Klo < 0$.
Let $\Rball > 0, w > 0, x_k \in \calM$. 
Define $a \colon [0, \infty) \rightarrow \reals$ as in equation~\eqref{defoffctA}.
There is a family of bump functions
$$\{\hat{h}_f \colon \calM \rightarrow R\}_{f \in [-w^{-1} a(\Rball), w^{-1} a(\Rball)]},$$
satisfying for each $f \in [-w^{-1} a(\Rball), w^{-1} a(\Rball)]$:
\begin{itemize}
\item $\hat{h}_f(x_k) = f$;
\item $\grad \hat{h}_f(x_k) = 0$;
\item the support of each $\hat{h}_f$ is contained in $B(x_k, \Rball)$;
\item $\left|\hat{h}_f(x)\right| \leq w^{-1} a(\Rball)$, $\norm{\grad \hat{h}_f(x)} \leq \frac{w^{-1}}{4 \sqrt{-\Klo}}$, $\norm{\Hess \hat{h}_f(x)} \leq \frac{1}{4 w}$ for all $x \in \calM$.
\end{itemize}
\end{lemma}
\begin{proof}
We use the notation established in Section~\ref{appendixbumpfunctionmain}.
For $f \in [-w^{-1} a(\Rball), w^{-1} a(\Rball)]$, define the smooth functions $\hat{h}_f \colon \calM \rightarrow \reals$ as follows:
$$\hat{h}_{c w^{-1} a(\Rball)}(x) = c w^{-1} {a}(\Rball) \phi_{\Rball}(\dist(x, x_k)^2/2), \quad \quad \forall c \in [-1, 1], \forall x \in \calM.$$
For each $f \in [-w^{-1} a(\Rball), w^{-1} a(\Rball)]$, we know:
\begin{itemize}
\item $\hat{h}_f$ is supported in $B(x_k, \Rball)$;
\item $\hat{h}_f(x_k) = c w^{-1} a(\Rball) = f$;
\item $\grad \hat{h}_f(x_k) = 0$, since $x \mapsto \dist(x, x_k)^2/2$ has zero gradient at $x = x_k$;
\item using the calculations from Section~\ref{boundinggradofbumpfunctionsproof} and $\left|c\right|\leq 1$,
\begin{align*}
\norm{\grad \hat{h}_f(x)} &= \left|c\right| w^{-1} a(\Rball) \left|\phi_{\Rball}'(\dist(x, x_k)^2/2)\right| \cdot \norm{\exp_x^{-1}(x_k)} \\
&\leq \left|c\right| w^{-1} a(\Rball) \max_{t \in [0, \Rball]}{\left|\phi_{\Rball}'(t^2/2)\right| \cdot t} 
\leq w^{-1} \frac{1}{4 \sqrt{-\Klo}};
\end{align*}
\item using the calculations from Section~\ref{boudinghessianofbumpfunctionsproof} and $\left|c\right|\leq 1$,
$$\norm{\Hess \hat{h}_f(x)} \leq \left|c\right| w^{-1} a(\Rball) (4\sqrt{-\Klo}/\Rball + 55/\Rball^2) \leq \frac{1}{4 w}.$$
\end{itemize}
\end{proof}

\begin{proof}[Proof of Lemma~\ref{bumpfctsfctvalsandgrads}]
For all
$$\hat{f} \in [-w^{-1} a(\Rball), w^{-1} a(\Rball)], \quad \quad g \in B_{x_k}(0, w^{-1} g_{\mathrm{norm}}(\Rball)),$$
Lemmas~\ref{lemmabumpfcts} (property~\ref{BFprop1}) and~\ref{lemmabumpfunctionsfunctionvalues} imply
$(\hat{h}_{\hat{f}} + h_g)(x_k) = \hat{f} + \frac{3}{8} \norm{g} {g}_{\mathrm{norm}}^{-1}\Big(w \norm{g}\Big)$
and
$\grad (\hat{h}_{\hat{f}} + h_g)(x_k) = 0 + g = g.$

We know that $\norm{g} \in [0, w^{-1} g_{\mathrm{norm}}(\Rball)]$.
Additionally, $g_{\mathrm{norm}}(0) = 0$ and $g_{\mathrm{norm}}$ is strictly increasing.
Therefore (introducing the change of variables $g_{\mathrm{norm}}(t) = w \norm{g}$)
$$\min_{t \in [0, \Rball]} \frac{3}{8} w^{-1} g_{\mathrm{norm}}(t) \cdot t \leq \frac{3}{8} \norm{g} {g}_{\mathrm{norm}}^{-1}\bigg(w \norm{g}\bigg) \leq \max_{t \in [0, \Rball]} \frac{3}{8} w^{-1} g_{\mathrm{norm}}(t) \cdot t.$$
Using that $t \mapsto g_{\mathrm{norm}}(t)$ is increasing,
$$0 \leq \frac{3}{8} \norm{g} {g}_{\mathrm{norm}}^{-1}\bigg(w \norm{g}\bigg) \leq \frac{3}{8} w^{-1} \Rball g_{\mathrm{norm}}(\Rball).$$
Therefore, for any
$f \in [-w^{-1} a(\Rball) + \frac{3}{8} w^{-1} \Rball g_{\mathrm{norm}}(\Rball), w^{-1} a(\Rball)]$
and for any $g \in B_{x_k}(0, w^{-1} g_{\mathrm{norm}}(\Rball)),$
we can define 
$h_{f, g} = \hat{h}_{f - \frac{3}{8} \norm{g} {g}_{\mathrm{norm}}^{-1}(w \norm{g})} + h_g.$
By construction, we have $h_{f, g}(x_k) = f$ and $\grad h_{f, g}(x_k) = g$.
Moreover, Lemmas~\ref{lemmabumpfcts} and~\ref{lemmabumpfunctionsfunctionvalues} imply
\begin{itemize}
\item the support of each $h_{f, g}$ is contained in $B(x_k, \Rball)$;
\item $\norm{\grad h_{f, g}(x)} \leq \frac{1}{4 w \sqrt{-\Klo}} + \frac{1}{4 w \sqrt{-\Klo}} \leq \frac{1}{2 w \sqrt{-\Klo}}$ for all $x \in \calM$;
\item $\norm{\Hess h_{f, g}(x)} \leq \frac{1}{4 w} + \frac{1}{4 w} \leq \frac{1}{2 w}$ for all $x \in \calM$.
\end{itemize}

\end{proof}

\section{Incorporating function-value queries} \label{fctvalues}
We want to extend the lower bound in Theorem~\ref{theoremnoacceleration} to algorithms which can make function-value and unbounded queries.
We do this in two steps.
First in Appendix~\ref{fctvalues} (this section), we prove Theorem~\ref{theoremnoaccelerationinfull} below, an extension of Theorem~\ref{theoremnoacceleration} providing a lower bound for algorithms using function-values but making bounded queries.
Second in Appendix~\ref{apppermittingqueriesoutsidedomain}, we prove Theorem~\ref{maintheoremunboundedqueries}, an extension of Theorem~\ref{theoremnoaccelerationinfull} providing a lower bound for algorithms making function-value and unbounded queries.


\begin{theorem} \label{theoremnoaccelerationinfull}
Let $\calM$ be a Hadamard manifold of dimension $d \geq 2$ which satisfies the ball-packing property~\aref{assumptionNbigWeak} with constants $\tilde{r}, \tilde{c}$ and point $\xorigin \in \calM$.  
Also assume $\calM$ has sectional curvatures in the interval $[\Klo, 0]$ with $\Klo < 0$.
Let  $r \geq \max\big\{\tilde{r}, \frac{8}{\sqrt{-\Klo}}, \frac{4(d+2)}{\tilde{c}}\big\}$.  
Define
${\kappa} = 4 r\sqrt{-\Klo} + 3.$
Let $\calA$ be any deterministic algorithm, and assume that $\calA$ always queries in $B(\xorigin, \mathscr{R})$, with $\mathscr{R} \geq r$.


There is a function $f \in \mathcal{F}_{\kappa, r}^{\xorigin}(\calM)$ with minimizer $x^*$ such that running $\calA$ on $f$ yields iterates $x_0, x_1, x_2, \ldots$ satisfying $\dist(x_{k}, x^*) \geq \frac{r}{4}$ for all $k = 0, 1, \ldots, T-1$, where
\begin{align} \label{rangeforkinfull}
T = \Bigg\lfloor \frac{\tilde{c} (d+2)^{-1} r}{\log\big(2000 \tilde{c} (d+2)^{-1} r (3 \mathscr{R} \sqrt{-\Klo} + 2)\big)}\Bigg\rfloor.
\end{align}

Moreover, $f$ is of the form $f(x) = \frac{1}{2} \dist(x, x^*)^2 + H(x)$ where $x^* \in B(\xorigin, \frac{3}{4} r)$ and $H \colon \calM \rightarrow \reals$ is a $C^\infty$ function satisfying
\begin{align} \label{usefulinequalitiesforprovingmaintheorem}
\left|H(x)\right| &\leq \frac{r}{64 \sqrt{-\Klo}}, & \norm{\grad H(x)} &\leq \frac{1}{2 \sqrt{-\Klo}}, & \norm{\Hess H(x)} &\leq \frac{1}{2}, \quad \forall x\in \calM.
\end{align}
%
%
%
\end{theorem}
\noindent Inequalities~\eqref{usefulinequalitiesforprovingmaintheorem} are included because they are useful for the proof for Theorem~\ref{maintheoremunboundedqueries} (see {Section~\ref{secwhereabbreviatesrR}}).

Before continuing to the details, let us first sketch the main ideas needed for this proof.
\cutchunk{In Sections~\ref{sec3dot2} and~\ref{proofoflemmaSec3}, we proved Theorem~\ref{theoremnoacceleration}, which provides a lower bound for algorithms which use gradient information but do not use function-value information.
In this section, we sketch the main ideas to prove Theorem~\ref{theoremnoaccelerationinfull}, i.e., to prove a lower bound for algorithms which use both function-value and gradient information.  Details are provided in Appendix~\ref{fctvalues}.}
In the case where the oracle only returns a gradient, Lemma~\ref{lemmabumpfcts} gives us a family of bump functions indexed by vectors $g$ such that for each $g$ in a ball of $\T_{x_k}\calM$ there is a bump function $h_g \colon \calM \rightarrow \reals$ satisfying $\grad h_g(x_k) = g$ (and a number of other properties).
This allowed us to use Lemma~\ref{smallandlargeballslemma} to choose the vector $g_k$.
In the case where the oracle also returns function values, we need a lemma which gives us a family of bump functions indexed by pairs $(f, g)$ lying in a \emph{cylinder} $I \times B$, where $I$ is a closed interval of the real line and $B$ is a closed ball in $\T_{x_k} \calM$.
For each pair $(f, g)$, the lemma should provide a bump function $h_{f, g} \colon \calM \rightarrow \reals$ satisfying $h_{f, g}(x_k) = f$ and $\grad h_{f, g}(x_k) = g$.  Lemma~\ref{bumpfctsfctvalsandgrads} in Appendix~\ref{fctvalues} does exactly this.

Lemma~\ref{bumpfctsfctvalsandgrads} works by constructing a bump function $h_{f, g}$ as a sum of two bump functions $\hat{h}_f$ and $h_g$ (the latter from Lemma~\ref{lemmabumpfcts}), the first controlling the function value of $h_{f, g}$, the second controlling its gradient.
We then use Lemma~\ref{smallandlargecylinderslemma}, which is analogous to Lemma~\ref{smallandlargeballslemma}, to choose a pair $(f_k, g_k)$ to return to the algorithm.
Since we are now comparing the volumes of sets of the form $I \times B$ which live in a space of dimension larger than $d$, we end up showing that
$\left|A_{k+1}\right| \geq \Omega(\left|A_k\right| / (\mathscr{R} \sqrt{-\Klo})^{d+2})$
instead of $\left|A_{k+1}\right| \geq \Omega(\left|A_k\right| / (\mathscr{R} \sqrt{-\Klo})^{d}).$
This is not an issue because $d+2 = \Theta(d)$.

\cutchunk{
\begin{remark}
Would allowing algorithms to query Hessian-vector products change the picture?  We expect not because, with exact arithmetic, a finite difference of two gradient vectors is enough to approximate a Hessian-vector product to \emph{arbitrary} precision~\citep[Ch. 10]{boumal2020intromanifolds}.
\end{remark}
}


\subsection{Proof of Theorem~\ref{theoremnoaccelerationinfull}}

Lemma~\ref{keylemmawithfctvals}, which is analogous to Lemma~\ref{keylemma}, forms the backbone of the proof of Theorem~\ref{theoremnoaccelerationinfull}.

\begin{lemma} \label{keylemmawithfctvals}
Let $\calM$ be a Hadamard manifold of dimension $d \geq 2$ with sectional curvatures in the interval $[\Klo, 0]$ with $\Klo < 0$.
Let $\xorigin \in \calM$, $r \geq 8 / \sqrt{-\Klo}$, $\mathscr{R} \geq r$.  Let $z_1, \ldots, z_N \in B(\xorigin, \frac{3}{4} r)$ be distinct points in $\calM$ such that $\dist(z_i, z_j) \geq \frac{r}{2}$ for all $i \neq j$.  
Define $A_0 = \{1, 2, \ldots, N\}$.
Let $\calA$ be any first-order algorithm which only queries points in $B(\xorigin, \mathscr{R})$.
Finally, let $w \geq 1$ (this is a tuning parameter we will set later).

For every nonnegative integer $k = 0, 1, 2, \ldots, \floor{w},$
the algorithm $\calA$ makes the query $x_k = \calA_k((f_0, (x_0, g_0)), \ldots, (f_{k-1}, (x_{k-1}, g_{k-1})))$ and there exists $f_k \in \reals$, $g_k \in \T_{x_k} \calM$ and a set $A_{k+1} \subseteq \{1, \ldots, N\}$ satisfying 
\begin{align} \label{L2lowerboundonAk}
\left|A_{k+1}\right| \geq \frac{12000 w(\left|A_k\right| - 1)}{\Big(2000 w(3 \mathscr{R} \sqrt{-\Klo} + 2)\Big)^{d+2}}
\end{align}
such that for each $j \in A_{k+1}$ there is a $C^{\infty}$ function $f_{j, k+1} \colon \calM \rightarrow \reals$ of the form
$$f_{j, k+1}(x) = \frac{1}{2} \dist(x, z_j)^2 + H_{j, k+1}(x)$$
satisfying:
\begin{enumerate} [label=\textbf{Lfv\arabic*}]
\item $f_{j, k+1}$ is $(1 - \frac{k+1}{2 w})$-strongly g-convex in $\calM$ and $[2r \sqrt{-\Klo} + 1 + \frac{k+1}{2 w}]$-smooth in $B(\xorigin, r)$; \label{L2prop1}

\item $\grad f_{j, k+1}(z_j) = 0$ (hence in particular the minimizer of $f_{j, k+1}$ is $z_j$); \label{L2prop2}

\item $f_{j, k+1}(x_{m}) = f_{m}$ and $\grad f_{j, k+1}(x_{m}) = g_{m}$ for $m = 0, 1, \ldots, k$; \label{L2prop3}

\item $\dist(x_{m}, z_j) \geq \frac{r}{4}$ for all $m = 0, 1, \ldots, k$. \label{L2prop4}

\item $\left|H_{j, k+1} (x) \right| \leq \frac{(k+1) r}{64 w \sqrt{-\Klo}}, \norm{\grad H_{j, k+1}(x)} \leq \frac{(k+1) }{2 w \sqrt{-\Klo}}, \norm{\Hess H_{j, k+1}(x)} \leq \frac{k+1}{2 w}$ for all $x \in \calM$. \label{L2prop5}
\end{enumerate}
\end{lemma}
\begin{proof}[Proof of Theorem~\ref{theoremnoaccelerationinfull}]
Let us apply Lemma~\ref{keylemmawithfctvals} to the manifold $\calM$ and algorithm $\calA$.
Let the points $z_1, \ldots, z_N$ be provided by the ball-packing property so that $N\geq e^{\tilde{c} r}$.

Set $w = \tilde{c} r (d+2)^{-1}$ in Lemma~\ref{keylemmawithfctvals}, and observe that
$$\min\{\floor{w}, T\} = \min\bigg\{\floor{\tilde{c} r (d+2)^{-1}}, \bigg\lfloor \frac{\tilde{c} r (d+2)^{-1}}{\log\big(2000 w(3 \mathscr{R} \sqrt{-\Klo} + 2)\big)} \bigg\rfloor\bigg\} = T$$
because $w = \tilde{c} r (d+2)^{-1} \geq 4$ and $\mathscr{R} \sqrt{-\Klo} \geq 8$.

For the same reasons as in the proof of Theorem~\ref{theoremnoacceleration}, it suffices to show that $\left|A_k\right| \geq 2$ for all $k \leq T$.
We induct on $k$.
(\textbf{Base case}) By the ball-packing property, $\left|A_0\right| \geq e^{\tilde{c} r} \geq 2$ since $r \geq 4 (d+2) / \tilde{c}$.
(\textbf{Inductive hypothesis}) Assume $\left|A_m\right| \geq 2$ for all $m \leq k$ and $k+1 \leq T$.
Therefore, $\left|A_{m}\right| - 1 \geq \left|A_{m}\right| / 2$ for all $m \leq k$.

Lemma~\ref{keylemmawithfctvals} implies
$$\left|A_{m+1}\right| \geq \frac{6000 w \left|A_m\right|}{\Big(2000 w(3 R \sqrt{-\Klo} + 2)\Big)^{d+2}} \geq \frac{2 \left|A_m\right|}{\Big(2000 w(3 \mathscr{R} \sqrt{-\Klo} + 2)\Big)^{d+2}} \quad \forall m \leq k.$$
Unrolling these inequalities and using $\left|A_0\right| \geq e^{\tilde{c} r}$, we get
$$\left|A_{k+1} \right| \geq \frac{e^{\tilde{c} r} 2^{k+1}}{ \big(2000 w(3 \mathscr{R} \sqrt{-\Klo} + 2)\big)^{(k+1) (d+2)}}.$$
On the other hand, $k+1 \leq T$ implies
$$\frac{e^{\tilde{c} r}}{ \big(2000 w(3 \mathscr{R} \sqrt{-\Klo} + 2)\big)^{(k+1) (d+2)}} \geq 1.$$
So $\left|A_{k+1}\right| \geq 2$.
Lastly, note that Lemma~\ref{keylemmawithfctvals} and our choice of $T$ implies for all $x \in \calM$
$$\left|H_{j, T}(x)\right| \leq \frac{r}{64 \sqrt{-\Klo}}, \quad \quad \norm{\grad H_{j, T}(x)} \leq \frac{1}{2\sqrt{-\Klo}}, \quad \quad \norm{\Hess H_{j, T}(x)} \leq \frac{1}{2}.$$ \end{proof}

\subsection{Proof of Lemma~\ref{keylemmawithfctvals}}
The proof approach for Lemma~\ref{keylemmawithfctvals} is very similar to the proof presented in Section~\ref{proofoflemmaSec3}, so we are more succinct, focusing on the additional analysis needed to handle function-value queries.
Before we prove this lemma, we state two lemmas which we use.
The following lemma is analogous to Lemma~\ref{lemmabumpfcts}, and its proof can be found in Appendix~\ref{bumpfunctionfctvalqueriesproof}.
\begin{lemma} \label{bumpfctsfctvalsandgrads}
Let $\calM$ be a Hadamard manifold with sectional curvatures in the interval $[\Klo, 0]$ with $\Klo < 0$.
Let $\Rball > 0, w > 0, x_k \in \calM$. 
Define $a \colon [0, \infty) \rightarrow \reals$ and $g_{\mathrm{norm}} \colon [0, \infty) \rightarrow \reals$ as in Lemma~\ref{appendixbumpfunction}.
There is a family of bump functions $\{{h}_{f, g} \colon \calM \rightarrow R\}$ indexed by $(f, g)$ satisfying
$$f \in \Big[-w^{-1} a(\Rball) + \frac{3}{8} w^{-1} \Rball g_{\mathrm{norm}}(\Rball), w^{-1} a(\Rball)\Big],  \quad g \in B_{x_k}(0, w^{-1} g_{\mathrm{norm}}(\Rball))$$
such that for each such $(f, g)$:
\begin{itemize}
\item ${h}_{f, g}(x_k) = f$;
\item $\grad {h}_{f, g}(x_k) = g$;
\item the support of each ${h}_{f, g}$ is contained in $B(x_k, \Rball)$;
\item $\left|h_{f, g}(x)\right| \leq 2w^{-1} a(\Rball)$, $\norm{\grad {h}_{f, g}(x)} \leq \frac{1}{2 w \sqrt{-\Klo}}$, $\norm{\Hess {h}_{f, g}(x)} \leq \frac{1}{2 w}$ for all $x \in \calM$.
\end{itemize}
The interval 
$[-w^{-1} a(\Rball) + \frac{3}{8} w^{-1} \Rball g_{\mathrm{norm}}(\Rball), w^{-1} a(\Rball)]$
has length at least $w^{-1} a(\Rball)$.
\end{lemma}

By \emph{cylinder} we mean any subset $C$ of a Euclidean space of the form $C = I \times B$, where $I$ is a closed interval of $\reals$ and $B$ is a closed Euclidean ball.  
Note that these cylinders include their interior---there is a distinction between a cylinder and the surface of a cylinder.
The height of a cylinder $C = I \times B$ is the length of $I$; the radius of $C$ is the radius of the ball $B$.

\begin{lemma} \label{smallandlargecylinderslemma}
In the following $I, I_j$ denote closed intervals of $\reals$, and $B, B_j$ denote closed $d$-dimensional balls of Euclidean space.

Let $C_1 = I_1 \times B_1, \ldots, C_n = I_n \times B_n$ be $n$ cylinders of radius $q$ and height $a$ each.
Assume each of the cylinders is also contained in a larger cylinder $C = I \times B$ of radius $r$ and height $b$: $C_j \subseteq C$ for all $j=1, \ldots, n$.  
Choose
$$(f, g) \in \arg\max_{(t, y) \in C} \left|\{j \in \{1, \ldots, n\} : (t, y) \in C_j\}\right|,$$
and let $A = \{j \in \{1, \ldots, n\} : (f, g) \in C_j\}$.
Then
$$\left|A\right| \geq n\frac{\Vol(C_1)}{\Vol(C)} = n\frac{\Vol(B_1)}{\Vol(B)}\cdot\frac{a}{b} = n \frac{q^d}{r^d} \cdot \frac{a}{b}.$$
\end{lemma}
\noindent The proof of Lemma~\ref{smallandlargecylinderslemma} is essentially identical to the proof of the analogous Lemma~\ref{smallandlargeballslemma}.  

Let us now prove Lemma~\ref{keylemmawithfctvals}.
We construct the function values $f_0, f_1, \ldots$, gradients $g_0, g_1, \ldots$, sets $A_0, A_1, \ldots$, and functions $f_{j, 0}, f_{j, 1}, \ldots$ inductively.
We prove the claim by induction on $k$.
The \textbf{base case} is the same as in Section~\ref{proofoflemmaSec3}.  
In particular, we define $f_{j, 0}(x) = \frac{1}{2} \dist(x, z_j)^2$ for all $j \in A_0 = \{1, \ldots, N\}$.

Let's consider the \textbf{inductive step}.  We are at iteration $k \geq 0$, and we assume properties \ref{L2prop1}, \ref{L2prop2}, \ref{L2prop3}, \ref{L2prop4}, \ref{L2prop5} hold with $k$ replacing $k+1$ in all expressions (the inductive hypothesis).
The algorithm queries a point $x_k$.
If $k \geq 1$, let $x_{\ell}, \ell < k,$ be a previous query point closest to $x_k$.
We can assume $x_\ell \neq x_k$.
Define $\tilde{A}_k$ as in equation~\eqref{definetildeAk}.

We shall define $f_{j, k+1} = f_{j, k} + h_{j, k}$
where $h_{j, k}$ is an appropriately chosen bump function.
We want $h_{j, k}$ to be a bump function whose support is contained in $B(x_k, \Rball^{(k)})$ where $\Rball^{(k)}$ is defined by equation~\eqref{Rkball}.
With this choice for $\Rball^{(k)}$, we set $h_{j, k}$ to be one of the bump functions $h_{f, g}$ supplied by Lemma~\ref{bumpfctsfctvalsandgrads} (which one remains to be determined).
With this setup, we immediately know the function $f_{j, k+1} = f_{j, k} + h_{f, g}$ satisfies properties \ref{L2prop2} and \ref{L2prop4}, as well as $f_{j, k+1}(x_m) = f_m$ and $\grad f_{j, k+1}(x_m) = g_m$ for $m=0, 1, \ldots, k-1$, for the reasons given in Section~\ref{proofoflemmaSec3}.
Additionally, using
$$2 w^{-1} a(\Rball^{(k)}) \leq 2 w^{-1} a\Big(\frac{r}{8}\Big) \leq 2 w^{-1} \frac{r/8}{4(4\sqrt{-\Klo})} = \frac{r}{64 w \sqrt{-\Klo}},$$
we see the function $f_{j, k+1} = f_{j, k} + h_{f, g}$ satisfies properties \ref{L2prop1} and \ref{L2prop5}.

It remains to choose $A_{k+1} \subseteq \tilde{A}_k$, $f_k \in \reals$ and $g_k \in \T_{x_k} \calM$ so that $f_{j, k+1}(x_k) = f_k, \grad f_{j, k+1}(x_k) = g_k$ for all $j \in A_{k+1}$, and inequality~\eqref{L2lowerboundonAk} is satisfied.

Consider the cylinders in $\reals \times \T_{x_k} \calM$ defined by 
$C_{j, k} = I_{j, k} \times B_{j, k}$
where we define
\begin{align} \label{definingIjk}
I_{j, k} = \Big[f_{j, k}(x_k) -w^{-1} a(\Rball^{(k)}) + \frac{3}{8} w^{-1} \Rball^{(k)} g_{\mathrm{norm}}(\Rball^{(k)}), f_{j, k}(x_k) + w^{-1} a(\Rball^{(k)})\Big]
\end{align}
and recall that $B_{j, k} = B_{x_k}(\grad f_{j, k}(x_k), w^{-1} g_{\mathrm{norm}}(\Rball^{(k)})).$
Let
$$(f_k, g_k) \in {\arg \max}_{(f, g) \in \reals \times \T_{x_k}\calM} \left|\{j \in \tilde{A}_k : (f, g) \in C_{j, k}\}\right|.$$
Define $A_{k+1} = \{j \in \tilde{A}_k : (f_k, g_k) \in C_{j, k}\}.$

For each $j \in A_{k+1}$ define $g_{j, k} = g_k - \grad f_{j, k}(x_k)$ and $f_{j, k}^{\text{(f.v.)}} = f_k - f_{j, k}(x_k)$.
Then Lemma~\ref{bumpfctsfctvalsandgrads} implies for each $j \in A_{k+1}$ there is a bump function $h_{j, k} := h_{f_{j, k}^{\text{(f.v.)}}, g_{j, k}}$ satisfying
$$h_{j, k}(x_k) = f_{j, k}^{\text{(f.v.)}} = f_k - f_{j, k}(x_k) \quad \text{ and } \quad \grad h_{j, k}(x_k) = g_{j, k} = g_k - \grad f_{j, k}(x_k).$$
Therefore for all $j\in A_{k+1}$, $f_{j, k+1}(x_k) = f_{j, k}(x_k) + h_{j, k}(x_k) = f_k$ and $\grad f_{j, k+1}(x_k) = \grad f_{j, k}(x_k) + \grad h_{j, k}(x_k) = g_k$.

It remains to verify inequality~\eqref{L2lowerboundonAk}.  To do so, we use Lemma~\ref{smallandlargecylinderslemma}.
To use this lemma, we need (a) a good upper bound on the radius of a ball $B_k \subseteq \T_{x_k} \calM$ containing the balls $B_{j, k}, j \in \tilde{A}_k$, and (b) a good upper bound for the length of an interval $I \subseteq \reals$ containing the intervals $I_{j, k}, j \in \tilde{A}_k$.
We've already done (a) in the proof from Section~\ref{proofoflemmaSec3}. 
Recall that we showed (using lines~\eqref{ratioofgradientvolumesCase1} and~\eqref{ratioofgradientvolumesCase2modded})
$$\frac{\Vol(B_{j, k})}{\Vol(B_k)} \geq \frac{1}{(2000 w(3 \mathscr{R} \sqrt{-\Klo} + 2))^d}.$$
For (b), we upper bound the length of an interval $I_k$ containing $I_{j, k}, j \in \tilde{A}_k$ in two cases, as in Section~\ref{proofoflemmaSec3}:

\textbf{Case 1}: either $k=0$, or $k\geq 1$ and $\sqrt{-\Klo} \dist(x_k, x_{\ell}) > 4$.  

\textbf{Case 2}: $k \geq 1$ and $\sqrt{-\Klo} \dist(x_k, x_{\ell}) \leq 4$.

In each case, we show that
$$\frac{\mathrm{Length}(I_{j, k})}{\mathrm{Length}(I_k)} \geq \frac{12000 w}{(2000 w (3 \mathscr{R} \sqrt{-\Klo} + 2))^2}.$$
Therefore, using Lemma~\ref{smallandlargecylinderslemma}, $(f_k, g_k)$ is contained in
\begin{align*}
&\left|A_{k+1}\right| \geq \left|\tilde{A}_k\right| \frac{\Vol(B_{j, k})}{\Vol(B_k)} \cdot \frac{\mathrm{Length}(I_{j, k})}{\mathrm{Length}(I_k)} \geq \frac{12000 w(\left|A_k\right| - 1)}{\Big(2000 w(3 \mathscr{R} \sqrt{-\Klo} + 2)\Big)^{d+2}}
\end{align*}
of the cylinders $C_{j, k}, j \in \tilde{A}_k$.
This concludes the inductive step, proving Lemma~\ref{keylemmawithfctvals}.

\subsubsection{Case 1 (for function-value queries)}
We have
\begin{equation} \label{previoususefulineqforfctvalscase1}
\begin{split}
&\left| \dist(x_k, z_j)^2 - \dist(x_k, \xorigin)^2 \right| \\
&= (\dist(x_k, z_j) + \dist(x_k, \xorigin)) \left| \dist(x_k, z_j) - \dist(x_k, \xorigin) \right| \\
&\leq (\dist(x_k, z_j) + \dist(x_k, \xorigin)) \dist(z_j, \xorigin) \\
&\leq (2 \dist(x_k, \xorigin) + \dist(z_j, \xorigin)) \dist(z_j, \xorigin) 
\leq (2 \dist(x_k, \xorigin) + r) r.
\end{split}
\end{equation}
By the inductive hypothesis and $k \leq w$, $\left| H_{j, k}(x_k) \right| \leq \frac{k w^{-1} r}{64 \sqrt{-\Klo}} \leq \frac{r}{64 \sqrt{-\Klo}}.$
Combining this with inequality~\eqref{previoususefulineqforfctvalscase1}, we conclude
\begin{equation} \label{ineqconclusioncase1fctvals}
\begin{split}
&\left|f_{j, k}(x_k) - \frac{1}{2} \dist(x_k, \xorigin)^2\right| = \left|H_{j, k}(x_k) + \frac{1}{2}\dist(x_k, z_j)^2 - \frac{1}{2} \dist(x_k, \xorigin)^2\right| \\
& \leq \frac{r}{64 \sqrt{-\Klo}} + \frac{1}{2}(2 \dist(x_k, \xorigin) + r) r \leq (\dist(x_k, \xorigin) + r) r \leq (\mathscr{R} + r) r
\end{split}
\end{equation}
using $r \sqrt{-\Klo} \geq 8$ for the penultimate inequality.
Therefore, all function values $f_{j, k}(x_k), j \in \tilde{A}_k,$ are contained in an interval centered at $\frac{1}{2} \dist(x_k, \xorigin)^2$ of length at most $2(\mathscr{R} + r) r$.
This implies that all the intervals $I_{j, k}, j \in \tilde{A}_k,$ are contained in an interval $I_k$ centered at $\frac{1}{2} \dist(x_k, \xorigin)^2$ of length at most 
\begin{align*}
2(\mathscr{R} + r) r + 2 w^{-1} a(\Rball^{(k)}) &\leq 2(\mathscr{R} + r) r + 2 w^{-1} a\Big(\frac{r}{8}\Big) \leq 2(\mathscr{R} + r) r + \frac{r}{64 \sqrt{-\Klo}} \\
&\leq 2(\mathscr{R} + 2 r) r \leq 6 \mathscr{R} r
\end{align*}
using that $w \geq 1$ and $r \sqrt{-\Klo} \geq 8$.

Now assume $k=0$ or $\sqrt{-\Klo} \dist(x_k, x_{\ell}) > 4$, where $x_\ell$ is defined in equation~\eqref{closestpoint}.
Using $r\sqrt{-\Klo} \geq 8$ and the definition of $\Rball^{(k)}$~\eqref{Rkball}, this assumption implies $\Rball^{(k)} \geq 1/\sqrt{-\Klo}$.
Therefore,
\begin{align*}
\frac{\mathrm{Length}(I_{j, k})}{\mathrm{Length}(I_k)} &\geq \frac{w^{-1} a(\Rball^{(k)})}{6 \mathscr{R} r} \geq \frac{w^{-1} a(1/\sqrt{-\Klo})}{6 \mathscr{R} r} \geq \frac{1}{2000 w (\mathscr{R} \sqrt{-\Klo}) (r \sqrt{-\Klo})}.
\end{align*}

\subsubsection{Case 2 (for function-value queries)}
Assume $k \geq 1$ and $\sqrt{-\Klo} \dist(x_k, x_\ell) \leq 4$.
Since $r \sqrt{-\Klo} \geq 8$, $\Rball^{(k)} = \dist(x_k, x_\ell) / 4$.
The analysis for this case is similar to Case 2 in Section~\ref{case2}.

The inductive hypothesis implies $\norm{\Hess H_{j, k}(x)} \leq \frac{k}{2 w} \leq \frac{1}{2}$ for all $x \in \calM$,
using $k \leq w$.
So as in equation~\eqref{Case2boundonhessianoffjk}, we know $$\norm{\Hess f_{j, k}(x)} \leq 3 \mathscr{R} \sqrt{-\Klo} + \frac{3}{2}, \quad\quad \forall x \in B(z_j, \max\{\dist(x_k, z_j), \dist(x_\ell, z_j)\}).$$
Additionally, the inductive hypothesis implies $f_{j, k}(x_{\ell}) = f_{\ell}$ and $\grad f_{j, k}(x_{\ell}) = g_{\ell}$ for all $j \in \tilde A_k$.
Therefore, by Lemma~\ref{muHesslemma}
\begin{align*}
\norm{f_{j, k}(x_k) - f_\ell - \inner{g_\ell}{\exp_{x_\ell}^{-1}(x_k)}} \leq \Big(3 \mathscr{R} \sqrt{-\Klo} + \frac{3}{2}\Big) \dist(x_k, x_{\ell})^2, \quad \quad  \text{$\forall j \in \tilde A_k$.}
\end{align*}
We have shown that all the function-values $f_{j, k}(x_k), j \in \tilde{A}_k,$ are contained in an interval centered at $f_\ell + \inner{g_\ell}{\exp_{x_\ell}^{-1}(x_k)}$ with length at most $2 (3 \mathscr{R} \sqrt{-\Klo} + \frac{3}{2}) \dist(x_k, x_{\ell})^2$.

Therefore, all the intervals $I_{j, k}, j \in \tilde A_k$, are contained in an interval $I_k \subseteq \reals$ of length 
\begin{align*}
&2 \Big(3 \mathscr{R} \sqrt{-\Klo} + \frac{3}{2}\Big) \dist(x_k, x_{\ell})^2 + 2 w^{-1} a(\dist(x_k, x_\ell) / 4) \leq 2 \Big(3 \mathscr{R} \sqrt{-\Klo} + 2\Big) \dist(x_k, x_{\ell})^2
\end{align*}
using $w \geq 1$ and
$a(\dist(x_k, x_\ell) / 4) = \frac{(\dist(x_k, x_\ell) / 4)^2}{4(\sqrt{-\Klo} \dist(x_k, x_\ell) + 55)} \leq \frac{\dist(x_k, x_\ell)^2}{64(55)} .$
Therefore,
\begin{align*}
&\frac{\mathrm{Length}(I_{j, k})}{\mathrm{Length}(I_k)} \geq \frac{w^{-1} a(\dist(x_k, x_\ell)/4)}{ 2 \Big(3 \mathscr{R} \sqrt{-\Klo} + 2\Big) \dist(x_k, x_{\ell})^2 }
\geq \frac{w^{-1} \dist(x_k, x_\ell)^2}{8000 \Big(3 \mathscr{R} \sqrt{-\Klo} + 2\Big) \dist(x_k, x_{\ell})^2} \\
&= \frac{3 \mathscr{R} \sqrt{-\Klo} + 2}{24}\cdot\frac{12000 w}{\Big(2000 w \Big(3 \mathscr{R} \sqrt{-\Klo} + 2\Big)\Big)^2} \geq \frac{12000 w}{\Big(2000 w \Big(3 \mathscr{R} \sqrt{-\Klo} + 2\Big)\Big)^2}
\end{align*}
using $\Rball^{(k)} = \dist(x_k, x_\ell) / 4$ and 
$$a(\dist(x_k, x_\ell)/4) = \frac{(\dist(x_k, x_\ell) / 4)^2}{4(\sqrt{-\Klo} \dist(x_k, x_\ell) + 55)} \geq \frac{(\dist(x_k, x_\ell) / 4)^2}{4(4 + 55)} \geq \frac{\dist(x_k, x_\ell)^2}{4000},$$
which itself follows from $\dist(x_k, x_\ell) \leq 4 / \sqrt{-\Klo}$.


\section{Unbounded queries} \label{apppermittingqueriesoutsidedomain}
We now want to extend the lower bound from Theorem~\ref{theoremnoaccelerationinfull}, which holds for algorithms querying only in $B(\xorigin, \mathscr{R})$, to algorithms which can query anywhere.
That is, we want to prove Theorem~\ref{maintheoremunboundedqueries}.

\begin{theorem} \label{maintheoremunboundedqueries}
Let $\calM$ be a Hadamard manifold of dimension $d \geq 2$ which satisfies the ball-packing property~\aref{assumptionNbigWeak} with constants $\tilde{r}, \tilde{c}$ and point $\xorigin \in \calM$.  
Also assume $\calM$ has sectional curvatures in the interval $[\Klo, 0]$ with $\Klo < 0$.
Let  $r \geq \max\big\{\tilde{r}, \frac{8}{\sqrt{-\Klo}}, \frac{4(d+2)}{\tilde{c}}\big\}$.  
Define
${\kappa} = 4 r\sqrt{-\Klo} + 3.$
Let $\calA$ be any deterministic algorithm.

There is a function $f \in \mathcal{F}_{{3 \kappa}, r}^{\xorigin}(\calM)$ with minimizer $x^*$ such that running $\calA$ on $f$ yields iterates $x_0, x_1, x_2, \ldots$ satisfying $\dist(x_{k}, x^*) \geq \frac{r}{4}$ for all $k = 0, 1, \ldots, {T}-1$, where
\begin{align*} 
{T} = \Bigg\lfloor \frac{\tilde{c} (d+2)^{-1} r}{\log\big(2000 \tilde{c} (d+2)^{-1} r (3 \mathscr{R} \sqrt{-\Klo} + 2)\big)}\Bigg\rfloor \geq \Bigg\lfloor \frac{\tilde{c} (d+2)^{-1} r}{\log\big(2\cdot 10^6 \cdot \tilde{c} (d+2)^{-1} r (r \sqrt{-\Klo})^2\big)}\Bigg\rfloor
\end{align*}
with $\mathscr{R} = 2^9 r \log(r \sqrt{-\Klo})^2$.
\cutchunk{Moreover, $f$ is $\frac{1}{6}$-strongly g-convex in $\calM$ and $[2 \mathscr{R} \sqrt{-\Klo} + \frac{3}{2}]$-smooth in the ball $B(\xorigin, \mathscr{R})$.
Outside of the ball $B(\xorigin, \mathscr{R})$, $f$ is simply half the squared distance to $\xorigin$, i.e., $f(x) = \frac{1}{2} \dist(x, \xorigin)^2$ for all $x \not\in B(\xorigin, \mathscr{R})$.}
\end{theorem}

To prove Theorem~\ref{maintheoremunboundedqueries}, the high-level idea is to modify all hard instances $f$ from Theorem~\ref{theoremnoaccelerationinfull} so that $f(x) = \frac{1}{2} \dist(x, \xorigin)^2$ for $x \not \in B(\xorigin, \mathscr{R})$ (recall $\mathscr{R} \geq r$).  
This way, the algorithm gains no information by querying outside the ball $B(\xorigin, \mathscr{R})$.
On the other hand, we still want the hard functions $f$ to remain untouched in the ball $B(\xorigin, r)$.
In the region between radii $r$ and $\mathscr{R}$, we smoothly interpolate between these two choices of functions.
We show that we can choose $\mathscr{R}$ appropriately so that the lower bound $\tilde{\Omega}(r)$ still holds and the modified functions are still strongly g-convex.
Technically, we do this via a reduction, which is depicted in Figure~\ref{figredboundedqueries}.
This argument was inspired by~\citep[Sec.~5.2]{carmon2017lower}.

\subsection{Proof of Theorem~\ref{maintheoremunboundedqueries}: a reduction from Theorem~\ref{theoremnoaccelerationinfull}} \label{secpermittingqueriesoutsidedomain}
Define $\mathscr{D} \colon \calM \rightarrow \reals$ by $\mathscr{D}(x) = \frac{1}{2} \dist(x, \xorigin)^2$.
Given any $f \colon \calM \rightarrow \reals$ (think from Theorem~\ref{theoremnoaccelerationinfull}), define the function $f_{r, \mathscr{R}} \colon \calM \rightarrow \reals$ by
\begin{equation} \label{eqforfrR}
\begin{split}
f_{r, \mathscr{R}}(x) =& s_{r, \mathscr{R}}\big(\mathscr{D}(x)\big) f(x) + \Big[1 - s_{r, \mathscr{R}}\big(\mathscr{D}(x)\big)\Big] \mathscr{D}(x)
\end{split}
\end{equation}
where $s_{r, \mathscr{R}} \colon \reals \rightarrow \reals$ is a $C^{\infty}$ function which is $1$ on $(-\infty, \frac{1}{2}r^2]$ and $0$ on $[\frac{1}{2} \mathscr{R}^2, \infty)$.
More precisely, following~\citet[Lem.~2.20,~2.21]{lee2012smoothmanifolds} we define the $C^\infty$ function $t \colon \reals \rightarrow \reals$ by
$$\quad \quad t(\tau) = \begin{cases} 
      1 & \text{for } \tau \in (-\infty, 0]; \\
      \frac{e^{-\frac{1}{1-\tau}}}{e^{-\frac{1}{1-\tau}} + e^{-\frac{1}{\tau}}} & \text{for } \tau \in (0,1); \\
      0 & \text{for } \tau \in [1, \infty)
\end{cases}$$
and define $s_{r, \mathscr{R}} \colon \reals \rightarrow \reals$ by
$$s_{r, \mathscr{R}}(\mathscr{D}) = t\bigg(\frac{\mathscr{D} -\frac{1}{2} r^2}{\frac{1}{2} \mathscr{R}^2 - \frac{1}{2} r^2}\bigg), \quad \quad \text{ for all $\mathscr{D} \in \reals.$}$$
In Appendix~\ref{secwhereabbreviatesrR}, we show that if we set $\mathscr{R} = 2^9 r \log(r \sqrt{-\Klo})^2$ and if $f \in \mathcal{F}_{\kappa, r}^{\xorigin}(\calM)$ is from Theorem~\ref{theoremnoaccelerationinfull}, then $f_{r, \mathscr{R}} \in \mathcal{F}_{3 \kappa, r}^{\xorigin}(\calM)$.

\begin{definition} \label{oracledefinition}
The (first-order) oracle for a differentiable function $f \colon \calM \rightarrow \reals$ is the map $\calO_f \colon \calM \rightarrow \reals \times \T \calM$ given by $\calO_f(x) = (f(x), (x, \grad f(x))).$
\end{definition}

Given the oracle $\calO_f$ of any function $f$, we can use $\calO_f$ to emulate the oracle $\calO_{f_{r, \mathscr{R}}}$ using equation~\eqref{eqforfrR}, and the following formula for $\grad f_{r, \mathscr{R}}$:
\begin{align*}
\grad f_{r, \mathscr{R}}(x) =
\begin{cases} 
      -\exp_x^{-1}(\xorigin) & \text{if } \mathrm{d}(x, \xorigin) > \mathscr{R}; \\
      \grad f(x) & \text{if }\mathrm{d}(x, \xorigin) \leq r; \\
      -s_{r, \mathscr{R}}'\big(\mathscr{D}(x)\big) \big(f(x) - \mathscr{D}(x)\big) \exp_x^{-1}(\xorigin) \\ \quad- \big(1 - s_{r, \mathscr{R}}\big(\mathscr{D}(x)\big)\big) \exp_x^{-1}(\xorigin) \\ \quad + s_{r, \mathscr{R}}\big(\mathscr{D}(x)\big)\grad f(x) & \text{otherwise}. 
\end{cases}
\end{align*}
See Appendix~\ref{secwhereabbreviatesrR} for the derivation of this formula for $\grad f_{r, \mathscr{R}}$.

To prove a lower bound for an algorithm $\calB$ querying anywhere, we make $\calB$ interact with the oracle $\calO_{f_{r, \mathscr{R}}}$ (which we simulate using $\calO_f$).
This implicitly defines an algorithm $\calA$ which interacts with $\calO_f$---see Figure~\ref{figredboundedqueries}.
Explicitly, the algorithm $\calA$ internally runs the algorithm $\calB$ as a subroutine as follows:
\begin{itemize}
\item if $\calB$ outputs $y_k \not \in B(\xorigin, \mathscr{R})$, $\calA$ does not query the oracle $\calO_f$, but simply passes 
$$(f_{r, \mathscr{R}}(y_k), \grad f_{r, \mathscr{R}}(y_k)) = \Big(\frac{1}{2} \dist(y_k, \xorigin)^2, -\exp_{y_k}^{-1}(\xorigin)\Big)$$
to $\calB$; this corresponds to path 1'-2'-3' in Figure~\ref{figredboundedqueries};

\item if $\calB$ outputs $y_k \in B(\xorigin, \mathscr{R})$, $\calA$ queries $\calO_f$ at $x_i = y_k$, receives $(f(x_i), \grad f(x_i))$ from $\calO_f$, and passes $(f_{r, \mathscr{R}}(x_i), \grad f_{r, \mathscr{R}}(x_i))$ to $\calB$ (which it computes using \\ $(f(x_i), \grad f(x_i))$); this corresponds to path 1-2-3-4-5 in Figure~\ref{figredboundedqueries}.
%
\end{itemize}
Inside of $\calA$, the algorithm $\calB$ outputs the sequence $y_0, y_1, y_2, \ldots$.
The algorithm $\calA$ produces the sequence of queries ${x}_0 = y_{k_0}, {x}_1 = y_{k_1}, \ldots \in B(\xorigin, \mathscr{R})$ where $0 \leq k_0 < k_1 < k_2 < \ldots$ are integers.  Let $\calK_T = \{k_0, k_1, \ldots, k_{T-1}\}$.
\begin{figure}[h]
    \centering
    \includegraphics[width=1\textwidth]{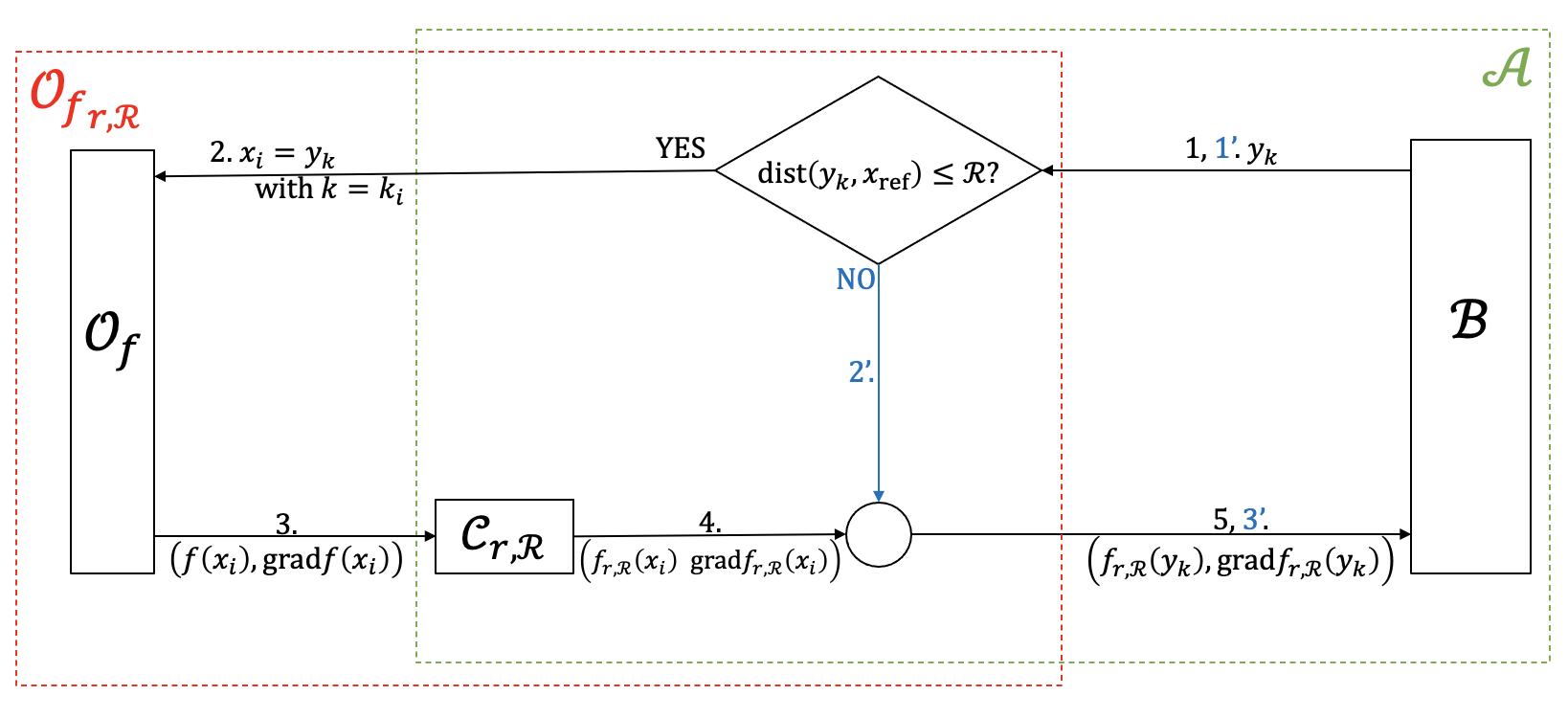}
    \caption{A diagram of the reduction used in Section~\ref{secpermittingqueriesoutsidedomain}.  
    The algorithm $\calA$ first internally runs $\calB$ to get an iterate $y_k$.  Then, depending on the distance between $y_k$ and $\xorigin$, $\calA$ either queries the oracle $\calO_f$ (path 1-2-3-4-5) or does not query the oracle (path 1'-2'-3').
    The box $\mathscr{C}_{r, \mathscr{R}}$ represents a map which when given a pair $(f(x), \grad f(x))$, outputs $(f_{r, \mathscr{R}}(x), \grad f_{r, \mathscr{R}}(x))$.  
    At step 2', $\calA$ computes $(\frac{1}{2} \dist(y_k, \xorigin)^2, -\exp_{y_k}(\xorigin))$ (not shown for clarity) which it then returns to $\calB$.}
    \label{figredboundedqueries}
\end{figure}

By design, algorithm $\calA$ makes queries only in $B(\xorigin, \mathscr{R})$.
Therefore, we can apply Theorem~\ref{theoremnoaccelerationinfull} to $\calA$.
That theorem implies there is a function $f \in \mathcal{F}_{\kappa, r}^{\xorigin}(\calM)$ with minimizer $x^*$ for which running $\calA$ on $f$ yields ${x}_0 = y_{k_0}, {x}_1 = y_{k_1}, \ldots, x_{T-1} = y_{k_{T-1}} \in B(\xorigin, \mathscr{R})$ satisfying $\dist(x^*, x_{k}) \geq \frac{r}{4}$ for all $k =0, 1, \ldots T-1$.
Here, $\kappa = 4 r \sqrt{-\Klo}+3$, and $T$ is given by Theorem~\ref{theoremnoaccelerationinfull} with $\mathscr{R} = 2^9 r \log(r \sqrt{-\Klo})^2$, that is,
$$T = \Bigg\lfloor \frac{\tilde{c} (d+2)^{-1} r}{\log\big(2000 \tilde{c} (d+2)^{-1} r (3 \mathscr{R} \sqrt{-\Klo} + 2)\big)}\Bigg\rfloor 
\geq \Bigg\lfloor \frac{\tilde{c} (d+2)^{-1} r}{\log\big(2\cdot 10^6 \cdot \tilde{c} (d+2)^{-1} r (r \sqrt{-\Klo})^2\big)}\Bigg\rfloor,$$
using that $r \sqrt{-\Klo} \geq 8$.  
In other words, $\text{$\dist(x^*, y_{k}) \geq \frac{r}{4}$ for all $k  \in \mathcal{K}_T.$}$

On the other hand, we know that $\text{$\dist(\xorigin, y_{k}) \geq \mathscr{R}$ for all $k \in \{0, 1, \ldots, T-1\} \setminus \calK_T$.}$  
Therefore, using that $\mathscr{R} \geq r$ and $x^* \in B(\xorigin, \frac{3}{4} r)$, 
$$\text{$\dist(x^*, y_{k}) \geq \frac{r}{4}$ for all $k \in \{0, 1, \ldots, T-1\} \setminus \calK_T$.}$$
We conclude $\dist(x^*, y_{k}) \geq \frac{r}{4}$ for all $k = 0, 1, ..., T-1$.
Finally, observe that (by our construction of $\calA$) if we run $\calB$ on the function $f_{r, \mathscr{R}}$ then we get exactly the sequence $y_0, y_1, \ldots, y_{T-1}$.  
Since $f_{r, \mathscr{R}} \in \mathcal{F}_{3 \kappa, r}^{\xorigin}(\calM)$ if $\mathscr{R} = 2^9 r \log(r \sqrt{-\Klo})^2$, as stated above, this proves Theorem~\ref{maintheoremunboundedqueries}.

\subsection{Verifying $f_{r, \mathscr{R}}$ is in the function class} \label{secwhereabbreviatesrR}
In this section, we abbreviate $s = s_{r, \mathscr{R}}$.
To finish the proof of Theorem~\ref{maintheoremunboundedqueries} (from Section~\ref{secpermittingqueriesoutsidedomain}), it remains to show that if $f \in \mathcal{F}_{\kappa, r}^{\xorigin}(\calM)$ is a hard function from Theorem~\ref{theoremnoaccelerationinfull}, then $f_{r, \mathscr{R}} \in \mathcal{F}_{3 \kappa, r}^{\xorigin}(\calM)$ for a suitable choice of $\mathscr{R}$.  
To do this, we use that a hard function $f$ from Theorem~\ref{theoremnoaccelerationinfull} is $\frac{1}{2}$-strongly g-convex in $\calM$ and $[2 r \sqrt{-\Klo} + \frac{3}{2}]$-smooth in $B(\xorigin, r)$.
We also use that $f$ has the form
$$f(x) = \frac{1}{2} \dist(x, x^*)^2 + H(x), \quad \text{with } x^* \in B\Big(\xorigin, \frac{3}{4}r\Big),$$
and, from inequalities~\eqref{usefulinequalitiesforprovingmaintheorem}, for all $x \in \calM$ we have
\begin{itemize}
\item $\norm{\grad H(x)} \leq \frac{1}{2 \sqrt{-\Klo}} \leq \frac{r}{16}$ (since we assume $r \sqrt{-\Klo} \geq 8$); and
\item $\left|H(x)\right| \leq \frac{r}{64 \sqrt{-\Klo}} \leq \frac{r^2}{512}$ (again since we assume $r \sqrt{-\Klo} \geq 8$).
\end{itemize}

Recall Definition~\ref{deffirstfctclass} for the function classes $\mathcal{F}_{\kappa, r}^{\xorigin}(\calM)$ and $\mathcal{F}_{3 \kappa, r}^{\xorigin}(\calM)$.  Since $f(x) = f_{r, \mathscr{R}}(x)$ for all $x \in B(\xorigin, r)$, it suffices to show that $\Hess f_{r, \mathscr{R}}(x) \succeq \frac{1}{6} I$ for all $x \in \calM$.
We just need to check that this is true when $r \leq \dist(x, \xorigin) \leq \mathscr{R}$.
Let's compute $\grad f_{r, \mathscr{R}}(x)$ and $\Hess f_{r, \mathscr{R}}(x)$ when $r \leq \dist(x, \xorigin) \leq \mathscr{R}$.  

Let $\gamma(t)$ be a geodesic with $\gamma(0) = x, \gamma'(0) = v$ and $\norm{v} = 1$.  
For the moment, define $\mathscr{D}(x) = \frac{1}{2} \dist(x, \xorigin)^2$, keeping in mind that $\mathscr{D} \colon \calM \rightarrow \reals$ depends on $\xorigin$.
Additionally, define 
$\tau(x) = \frac{\mathscr{D}(x) - \frac{1}{2} r^2}{\frac{1}{2} \mathscr{R}^2 - \frac{1}{2} r^2}$
so that 
$$s(\mathscr{D}(x)) = t(\tau(x)), \quad s'(\mathscr{D}(x)) = \frac{1}{\frac{1}{2} \mathscr{R}^2 - \frac{1}{2} r^2} t'(\tau(x)), \quad s''(\mathscr{D}(x)) = \frac{1}{(\frac{1}{2} \mathscr{R}^2 - \frac{1}{2} r^2)^2} t''(\tau(x)).$$

For the gradient, we have:
\begin{align*}
\inner{v}{\grad f_{r, \mathscr{R}}(x)} =& \frac{d}{dt}\Big[f_{r, \mathscr{R}}(\gamma(t))\Big]_{t=0} = \frac{d}{dt}\Big[s(\mathscr{D}(\gamma(t))) f(\gamma(t)) + [1 - s(\mathscr{D}(\gamma(t)))] \mathscr{D}(\gamma(t))\Big]_{t=0} \\
=& \frac{d}{dt}\Big[s(\mathscr{D}(\gamma(t)))\Big]_{t=0} (f(x) - \mathscr{D}(x)) + s(\mathscr{D}(x)) \frac{d}{dt}\Big[f(\gamma(t))\Big]_{t=0} \\
&+ [1 - s(\mathscr{D}(x))] \frac{d}{dt}\Big[\mathscr{D}(\gamma(t))\Big]_{t=0} \\
=& s'(\mathscr{D}(x)) \inner{v}{-\exp_x^{-1}(\xorigin)} (f(x) - \mathscr{D}(x)) + s(\mathscr{D}(x)) \inner{v}{\grad f(x)} \\
&+ [1 - s(\mathscr{D}(x))] \inner{v}{-\exp_x^{-1}(\xorigin)}.
\end{align*}

For the Hessian, we have:
\begin{align*}
\inner{v}{\Hess f_{r, \mathscr{R}}(x) v} &= \frac{d^2}{dt^2}\Big[f_{r, \mathscr{R}}(\gamma(t))\Big]_{t=0}
\\ &= \frac{d^2}{dt^2}\Big[s(\mathscr{D}(\gamma(t))) f(\gamma(t)) + [1 - s(\mathscr{D}(\gamma(t)))] \mathscr{D}(\gamma(t))\Big]_{t=0} \\
&=  \frac{d^2}{dt^2}\Big[s(\mathscr{D}(\gamma(t)))\Big]_{t=0} f(x) + 2 \frac{d}{dt}\Big[s(\mathscr{D}(\gamma(t)))\Big]_{t=0} \frac{d}{dt}\Big[f(\gamma(t))\Big]_{t=0} \\
&\quad + s(\mathscr{D}(x)) \frac{d^2}{dt^2}\Big[f(\gamma(t))\Big]_{t=0} - \frac{d^2}{dt^2}\Big[s(\mathscr{D}(\gamma(t)))\Big]_{t=0} \mathscr{D}(x) \\
&\quad - \frac{d}{dt}\Big[s(\mathscr{D}(\gamma(t)))\Big]_{t=0} \frac{d}{dt}\Big[\mathscr{D}(\gamma(t))\Big]_{t=0} + [1 - s(\mathscr{D}(x))] \frac{d^2}{dt^2}\Big[\mathscr{D}(\gamma(t))\Big]_{t=0}.
\end{align*}
Further simplifying yields:
\begin{align*}
\inner{v}{\Hess f_{r, \mathscr{R}}(x) v} &= s(\mathscr{D}(x)) \frac{d^2}{dt^2}\Big[f(\gamma(t))\Big]_{t=0} + [1 - s(\mathscr{D}(x))] \frac{d^2}{dt^2}\Big[\mathscr{D}(\gamma(t))\Big]_{t=0} \\
&\quad+  \frac{d^2}{dt^2}\Big[s(\mathscr{D}(\gamma(t)))\Big]_{t=0} (f(x) - \mathscr{D}(x)) \\
&\quad + 2 \frac{d}{dt}\Big[s(\mathscr{D}(\gamma(t)))\Big]_{t=0} \bigg(\frac{d}{dt}\Big[f(\gamma(t))\Big]_{t=0} - \frac{d}{dt}\Big[\mathscr{D}(\gamma(t))\Big]_{t=0}\bigg).
\end{align*}
Using that $f(x) = \frac{1}{2} \dist(x, x^*)^2 + H(x)$,
\begin{align*}
&\inner{v}{\Hess f_{r, \mathscr{R}}(x) v} = s(\mathscr{D}(x)) \inner{v}{\Hess f(x) v} + [1 - s(\mathscr{D}(x))] \inner{v}{\Hess \mathscr{D}(x)v} \\
&\quad+ \frac{d^2}{dt^2}\Big[s(\mathscr{D}(\gamma(t)))\Big]_{t=0} \bigg(\frac{1}{2}\dist(x, x^*)^2 + H(x) - \mathscr{D}(x)\bigg) \\
&\quad- 2 \frac{d}{dt}\Big[s(\mathscr{D}(\gamma(t)))\Big]_{t=0} \bigg(\inner{v}{\exp_{x}^{-1}(x^*) - \exp_{x}^{-1}(\xorigin)} - \inner{v}{\grad H(x)}\bigg).
\end{align*}
Rearranging yields
\begin{align*}
&\inner{v}{\Hess f_{r, \mathscr{R}}(x) v} = t(\tau(x)) \inner{v}{\Hess f(x) v} + [1 - t(\tau(x))] \inner{v}{\Hess \mathscr{D}(x)v} \\
&\quad+ \frac{d^2}{dt^2}\Big[s(\mathscr{D}(\gamma(t)))\Big]_{t=0} \frac{1}{2} \bigg(\dist(x, x^*) - \dist(x, \xorigin)\bigg)\bigg(\dist(x, x^*) + \dist(x, \xorigin)\bigg) \\
&\quad- 2 \frac{d}{dt}\Big[s(\mathscr{D}(\gamma(t)))\Big]_{t=0} \inner{v}{\exp_{x}^{-1}(x^*) - \exp_{x}^{-1}(\xorigin)} \\
&\quad+ \frac{d^2}{dt^2}\Big[s(\mathscr{D}(\gamma(t)))\Big]_{t=0} H(x) + 2 \frac{d}{dt}\Big[s(\mathscr{D}(\gamma(t)))\Big]_{t=0} \inner{v}{\grad H(x)}.
\end{align*}
Using $r \leq \dist(x, \xorigin) \leq \mathscr{R}$, we have $\left|\inner{v}{\exp_x^{-1}(\xorigin)}\right| \leq \dist(x, \xorigin) \leq \mathscr{R}.$
Using Proposition~\ref{TopogonovEuclidean},
$$\left|\inner{v}{\exp_x^{-1}(x^*) - \exp_x^{-1}(\xorigin)}\right| \leq \norm{\exp_x^{-1}(x^*) - \exp_x^{-1}(\xorigin)} \leq \dist(x^*, \xorigin) \leq r.$$
Using the triangle inequality,
\begin{align*}
&\left|\bigg(\dist(x, x^*) - \dist(x, \xorigin)\bigg)\bigg(\dist(x, x^*) + \dist(x, \xorigin)\bigg)\right| \\
&\leq \dist(x^*, \xorigin) \bigg(\dist(\xorigin, x^*) + 2\dist(x, \xorigin)\bigg) \leq r(r + 2 \mathscr{R}) \leq 3 r \mathscr{R}.
\end{align*}
Therefore, we have
\begin{align*}
\inner{v}{\Hess f_{r, \mathscr{R}}(x) v} &\geq t(\tau(x)) \inner{v}{\Hess f(x) v} + [1 - t(\tau(x))] \inner{v}{\Hess \mathscr{D}(x)v} \\
&\quad- \frac{3}{2} r \mathscr{R} \left|\frac{d^2}{dt^2}\Big[s(\mathscr{D}(\gamma(t)))\Big]_{t=0}\right| - 2 r \left|\frac{d}{dt}\Big[s(\mathscr{D}(\gamma(t)))\Big]_{t=0}\right| \\
&\quad- \frac{r^2}{512} \left|\frac{d^2}{dt^2}\Big[s(\mathscr{D}(\gamma(t)))\Big]_{t=0}\right| - \frac{r}{8} \left|\frac{d}{dt}\Big[s(\mathscr{D}(\gamma(t)))\Big]_{t=0}\right| \\
&\geq t(\tau(x)) \inner{v}{\Hess f(x) v} + [1 - t(\tau(x))] \inner{v}{\Hess \mathscr{D}(x)v} \\
&\quad-\frac{8}{5} r \mathscr{R} \left|\frac{d^2}{dt^2}\Big[s(\mathscr{D}(\gamma(t)))\Big]_{t=0}\right| - 3 r \left|\frac{d}{dt}\Big[s(\mathscr{D}(\gamma(t)))\Big]_{t=0}\right|.
\end{align*}
Additionally, we have
\begin{align*}
\left|\frac{d}{dt}\Big[s(\mathscr{D}(\gamma(t)))\Big]_{t=0}\right| &= \left|s'(\mathscr{D}(x)) \frac{d}{dt}\Big[\mathscr{D}(\gamma(t))\Big]_{t=0}\right|  = \frac{1}{\frac{1}{2} \mathscr{R}^2 - \frac{1}{2} r^2} \left|t'(\tau(x)) \inner{v}{-\exp_x^{-1}(\xorigin)}\right| \\&\leq \frac{\mathscr{R}}{\frac{1}{2} \mathscr{R}^2 - \frac{1}{2} r^2} \left|t'(\tau(x)) \right|,
\end{align*}
and
\begin{align*}
&\left|\frac{d^2}{dt^2}\Big[s(\mathscr{D}(\gamma(t)))\Big]_{t=0}\right| = \left|s''(\mathscr{D}(x)) \Bigg(\frac{d}{dt}\Big[\mathscr{D}(\gamma(t))\Big]_{t=0}\Bigg)^2 + s'(\mathscr{D}(x)) \frac{d^2}{dt^2}\Big[\mathscr{D}(\gamma(t))\Big]_{t=0}\right| \\
&\leq \frac{\mathscr{R}^2}{(\frac{1}{2} \mathscr{R}^2 - \frac{1}{2} r^2)^2} \left|t''(\tau(x))\right| + \frac{1}{\frac{1}{2} \mathscr{R}^2 - \frac{1}{2} r^2} \left|t'(\tau(x))\right| \inner{v}{\Hess \mathscr{D}(x) v}.
\end{align*}

In the following, we set $\mathscr{R} = 2^{9} r \log(r \sqrt{-\Klo})^2)$. 
This choice of $\mathscr{R}$ and $r\sqrt{-\Klo} \geq 8$ implies
$\mathscr{R} = 2^{9} r \log(r \sqrt{-\Klo})^2) \geq 2^{9} \log(8)^2 r \geq 2^{11} r.$
Since $\mathscr{R} \geq 2^{11} r$, we conclude
\begin{align*}
\inner{v}{\Hess f_{r, \mathscr{R}}(x) v} &\geq t(\tau(x)) \inner{v}{\Hess f(x) v} \\
&\quad+ \bigg[1 - t(\tau(x)) - \frac{2 r \mathscr{R}}{\frac{1}{2} \mathscr{R}^2 - \frac{1}{2} r^2} \left|t'(\tau(x))\right| \bigg] \inner{v}{\Hess \mathscr{D}(x)v} \\
&\quad- \frac{\frac{8}{5} r \mathscr{R}^3}{(\frac{1}{2} \mathscr{R}^2 - \frac{1}{2} r^2)^2} \left|t''(\tau(x))\right| - \frac{3 r \mathscr{R}}{\frac{1}{2} \mathscr{R}^2 - \frac{1}{2} r^2} \left|t'(\tau(x)) \right| \\
&\geq t(\tau(x)) \inner{v}{\Hess f(x) v} \\
&\quad + \bigg[1 - t(\tau(x)) - \frac{2 r \mathscr{R}}{\frac{1}{2}(1-2^{-22}) \mathscr{R}^2} \left|t'(\tau(x))\right| \bigg] \inner{v}{\Hess \mathscr{D}(x)v} \\
&\quad- \frac{\frac{8}{5} r \mathscr{R}^3}{(\frac{1}{2}(1-2^{-22}) \mathscr{R}^2)^2} \left|t''(\tau(x))\right| - \frac{3 r \mathscr{R}}{\frac{1}{2}(1-2^{-22})\mathscr{R}^2} \left|t'(\tau(x)) \right| \\
&= t(\tau(x)) \inner{v}{\Hess f(x) v} \\
&\quad+ \bigg[1 - t(\tau(x)) - \frac{4.5 r}{\mathscr{R}} \left|t'(\tau(x))\right| \bigg] \inner{v}{\Hess \mathscr{D}(x)v} \\
&\quad- \frac{\frac{32}{5} r}{(1-2^{-22})^2 \mathscr{R}} \left|t''(\tau(x))\right| - \frac{7 r}{\mathscr{R}} \left|t'(\tau(x)) \right|.
\end{align*}
One can check that $-2 \leq t'(\tau) \leq 0$ and $\left|t(\tau)''(0)\right| \leq 16$ for all $\tau \in (0,1)$.
So using $\mathscr{R} \geq 2^{11} r$,
\begin{align*}
\inner{v}{\Hess f_{r, \mathscr{R}}(x) v} &\geq t(\tau(x)) \inner{v}{\Hess f(x) v} + \bigg[1 - t(\tau(x)) - \frac{4.5 r}{\mathscr{R}} \left|t'(\tau(x))\right| \bigg] \inner{v}{\Hess \mathscr{D}(x)v} - \frac{1}{16}.
\end{align*}
Next we make two observations about the univariate function $t$:
\begin{itemize}
\item if $\tau \in [\frac{1}{2}, 1)$, then $1 - t(\tau) - \frac{4.5}{2^{11}} \left|t'(\tau)\right| \geq \frac{1}{2} - \frac{4.5}{2^{10}}$;
\item if $\tau \in (0, \frac{1}{2})$ and $\mathscr{R} \geq 9\cdot 4.5 r$, then $1 - t(\tau) - \frac{4.5 r}{\mathscr{R}} \left|t'(\tau)\right| \geq -2 e^{-\sqrt{\frac{\mathscr{R}}{9 r}}}$.  We prove this fact in the next Section~\ref{technicalfactsabouttforunboundedqueries}.
\end{itemize}
Using these facts, if $\tau(x) \in [\frac{1}{2}, 1)$ then (using $\inner{v}{\Hess \mathscr{D}(x) v} \geq 1$ and $\mathscr{R} \geq 2^{11} r$)
\begin{align*}
\inner{v}{\Hess f_{r, \mathscr{R}}(x) v} &\geq \frac{1}{2} - \frac{4.5}{2^{10}} - \frac{1}{16} \geq \frac{1}{6}.
\end{align*}
If $\tau(x) \in (0, \frac{1}{2})$, then (using $\inner{v}{\Hess f(x)v} \geq \frac{1}{2}$ and $\inner{v}{\Hess \mathscr{D}(x) v} \leq 2 \mathscr{R} \sqrt{-\Klo}$)
\begin{align*}
\inner{v}{\Hess f_{r, \mathscr{R}}(x) v} &\geq \frac{1}{2} \cdot \frac{1}{2} - 2 e^{-\sqrt{\frac{\mathscr{R}}{9 r}}} \cdot 2 \mathscr{R} \sqrt{-\Klo} - \frac{1}{16} \geq \frac{1}{2} \cdot \frac{1}{2} - \frac{1}{2^{6}} - \frac{1}{16} \geq \frac{1}{6},
\end{align*}
where the last inequality follows from choosing $\mathscr{R} = 2^{9} r \log(r \sqrt{-\Klo})^2$ and $r\sqrt{-\Klo} \geq 8$: 
\begin{align*}
2 e^{-\sqrt{\frac{\mathscr{R}}{9 r}}} \cdot 2 \mathscr{R} \sqrt{-\Klo} &= 4 e^{-{{\sqrt{\frac{2^9}{9}}} \log(r \sqrt{-\Klo})}} \cdot 2^{9} r \sqrt{-\Klo} \log(r \sqrt{-\Klo})^2 \\
&= 2^{11} (r \sqrt{-\Klo})^{-\sqrt{\frac{2^9}{9}}} r\sqrt{-\Klo} \log(r \sqrt{-\Klo})^2 \leq \frac{1}{2^{6}}.
\end{align*}

\subsection{Technical fact about the function $t \colon \reals \rightarrow \reals$ in the interval $(0, 1)$} \label{technicalfactsabouttforunboundedqueries}
\begin{lemma}
If $\tau \in (0, \frac{1}{2})$ and $c \geq 9$, then 
$1 - t(\tau) - \frac{1}{c} \left|t'(\tau)\right| \geq -{2}{e^{-{\sqrt{c/2}}}}.$
\end{lemma}
\begin{proof}
We have
$t'(\tau) = \frac{e^{\frac{1}{\tau -\tau ^2}}}{\left(e^{\frac{1}{1-\tau }}+e^{\frac{1}{\tau }}\right)^2} \cdot \frac{\left(-2 \tau ^2+2 \tau -1\right)}{(\tau -1)^2 \tau ^2} \leq 0$,
so
\begin{align} \label{teqone}
\left|t'(\tau)\right| = \frac{e^{\frac{1}{\tau -\tau ^2}}}{\left(e^{\frac{1}{1-\tau }}+e^{\frac{1}{\tau }}\right)^2} \cdot \frac{\left(2 \tau ^2-2 \tau +1\right)}{(\tau -1)^2 \tau ^2}.
\end{align}
Consider $\tau \in (0, \frac{1}{2})$ and also take $c \geq 9$.
Using~\eqref{teqone} we find
\begin{align*}
1 - t(\tau) - \frac{1}{c} \left|t'(\tau)\right| 
&= \frac{e^{\frac{1}{\tau -\tau ^2}} }{\left(e^{\frac{1}{1-\tau }}+e^{\frac{1}{\tau }}\right)^2}
   \frac{\left(c \tau ^4-2 c \tau ^3+c \tau ^2-2 \tau ^2+2 \tau -1\right)}{c (\tau
   -1)^2 \tau ^2}
   +\frac{e^{\frac{2}{1-\tau }}}{\left(e^{\frac{1}{1-\tau }}+e^{\frac{1}{\tau }}\right)^2}
   \\
   &\geq \frac{e^{\frac{1}{\tau -\tau ^2}} }{\left(e^{\frac{1}{1-\tau }}+e^{\frac{1}{\tau }}\right)^2}
   \frac{\left(c \tau ^4-2 c \tau ^3+c \tau ^2-2 \tau ^2+2 \tau -1\right)}{c (\tau
   -1)^2 \tau ^2}
   \\
   &\geq \frac{e^{\frac{1}{\tau -\tau ^2}} }{\left(e^{\frac{1}{1-\tau }}+e^{\frac{1}{\tau }}\right)^2}
   \left(\frac{c-1}{2 c}-\frac{1}{c \tau ^2}\right)
   \geq \frac{e^{\frac{1}{\tau -\tau ^2}} }{\left(e^{\frac{1}{1-\tau }}+e^{\frac{1}{\tau }}\right)^2}
   \min\left\{\frac{c-1}{2 c}-\frac{1}{c \tau ^2}, 0\right\}
\end{align*}
where for the penultimate inequality we used the fact
$$\frac{c \tau ^4-2 c \tau ^3+c \tau ^2-2 \tau ^2+2 \tau -1}{c (\tau -1)^2 \tau ^2}\geq \frac{c-1}{2 c}-\frac{1}{c \tau ^2} \quad \quad \forall \tau \leq \frac{1}{2}, c \geq 7.$$
This algebraic inequality can be verified by a computer algebra system, such as Mathematica.

For $\tau \in (0, \frac{1}{2}]$, we have
\begin{align} \label{teqtwo}
\frac{e^{\frac{1}{\tau -\tau ^2}}}{\left(e^{\frac{1}{1-\tau }}+e^{\frac{1}{\tau }}\right)^2} \leq \frac{e^{\frac{1}{\tau -\tau ^2}}}{\left(e^{\frac{1}{\tau }}\right)^2} = e^{\frac{1}{\tau -\tau ^2} - \frac{2}{\tau}} \leq e^{2 - \frac{1}{\tau}}
\end{align}
where for the last inequality we used that $\frac{1}{\tau -\tau ^2} - \frac{2}{\tau} \leq 2 - \frac{1}{\tau}$ for $\tau \in [0, \frac{1}{2}]$.
Therefore, using~\eqref{teqtwo},
\begin{align*}
1 - t(\tau) - \frac{1}{c} \left|t'(\tau)\right| 
   &\geq e^{2 - \frac{1}{\tau}}
   \min\left\{\frac{c-1}{2 c}-\frac{1}{c \tau ^2}, 0\right\} =
   \min\left\{e^{2 - \frac{1}{\tau}}\Big(\frac{c-1}{2 c}-\frac{1}{c \tau ^2}\Big), 0\right\}.
\end{align*}
We know $\tau \in (0, \frac{1}{2}]$ and $\frac{c-1}{2 c}-\frac{1}{c \tau ^2} \leq 0$ if and only if $0 < \tau \leq \frac{\sqrt{2}}{\sqrt{c-1}} \leq \frac{1}{2}$.
Additionally, $\lim_{\tau \rightarrow 0^+} e^{2 - \frac{1}{\tau}}\Big(\frac{c-1}{2 c}-\frac{1}{c \tau ^2}\Big) = 0$.
Therefore the minimum of $\tau \mapsto \min\left\{e^{2 - \frac{1}{\tau}}\Big(\frac{c-1}{2 c}-\frac{1}{c \tau ^2}\Big), 0\right\}$ for $\tau \in (0, \frac{1}{2})$ must occur at a critical point of $\tau \mapsto e^{2 - \frac{1}{\tau}}\Big(\frac{c-1}{2 c}-\frac{1}{c \tau ^2}\Big)$.
Let's compute that point:
$$0 = \frac{d}{d\tau}\Big[e^{2 - \frac{1}{\tau}}\Big(\frac{c-1}{2 c}-\frac{1}{c \tau ^2}\Big)\Big] = \frac{e^{2-\frac{1}{\tau }} \left(c \tau ^2-\tau ^2+4 \tau -2\right)}{2 c \tau ^4} \implies \tau = \frac{1}{\frac{\sqrt{c+1}}{\sqrt{2}}+1}.$$
Therefore,
\begin{align*}
1 - t(\tau) - \frac{1}{c} \left|t'(\tau)\right| 
   &\geq \min\left\{\Big[e^{2 - \frac{1}{\tau}}\Big(\frac{c-1}{2 c}-\frac{1}{c \tau ^2}\Big)\Big]_{\tau = \frac{1}{\frac{\sqrt{c+1}}{\sqrt{2}}+1}}, 0\right\} = -\frac{2 e^2 \left( 1 + \frac{\sqrt{c+1}}{\sqrt{2}}\right) }{c e^{1+\frac{\sqrt{c+1}}{\sqrt{2}}}} \\
   &\geq -\frac{2}{e^{\frac{\sqrt{c+1}}{\sqrt{2}}}} \geq -\frac{2}{e^{{\sqrt{c/2}}}}.
\end{align*}
\end{proof}

\section{Technical details for the ball-packing property}
\subsection{Geodesics diverge: Proof of Lemma~\ref{geodesicsdiverge}} \label{Appgeodesicsdiverge}
The angle between $v_1$ and $v_2$ is in the interval $[0, \pi]$; therefore, the statement of the lemma requires a proof only for $\theta \in [0, \pi]$.  
We split this into two cases.  
For both we use the following consequence of Proposition~\ref{TopogonovHyperbolicBoundedAbove}:
\begin{align} \label{topinequalityweuseforlargeball}
\cosh(\dist(z_1, z_2) \sqrt{-\Kup}) \geq \cosh(s \sqrt{-\Kup})^2 - \sinh(s \sqrt{-\Kup})^2 \cos(\theta).
\end{align}

If $\theta > \frac{\pi}{2}$, then 
$$\cosh(\dist(z_1, z_2) \sqrt{-\Kup}) \geq \cosh(s \sqrt{-\Kup})^2 \geq \cosh(s \sqrt{-\Kup}),$$
and so $\dist(z_1, z_2) \geq s \geq \frac{2}{3}s$.
So we can assume that $\theta \leq \frac{\pi}{2}$.

Note that
$e^{1 - \frac{2}{3} t} \geq \sqrt{3\bigg(1 - \frac{\cosh(t)^2 - \cosh(2 t / 3)}{\sinh(t)^2}\bigg)}  \text{ for all $t \geq 0$}.$
Therefore, 
\begin{align*}
\theta &= e^{1 - \frac{2}{3} s \sqrt{-\Kup}} \geq \sqrt{3\bigg(1 - \frac{\cosh(s \sqrt{-\Kup})^2 - \cosh(2 s \sqrt{-\Kup} / 3)}{\sinh(s \sqrt{-\Kup})^2}\bigg)}
\end{align*}
which implies
$\frac{\cosh(s \sqrt{-\Kup})^2 - \cosh(2 s \sqrt{-\Kup} / 3)}{\sinh(s \sqrt{-\Kup})^2} \geq 1 - \frac{1}{3} \theta^2 \geq \cos(\theta)$ (since $\theta \in (0, \frac{\pi}{2}]$).
Rearranging this inequality and applying inequality~\eqref{topinequalityweuseforlargeball},
$$\cosh(2 s \sqrt{-\Kup} / 3) \leq \cosh(s \sqrt{-\Kup})^2 - \sinh(s \sqrt{-\Kup})^2 \cos(\theta) \leq \cosh(\dist(z_1, z_2) \sqrt{-\Kup}).$$
We conclude $\dist(z_1, z_2) \geq \frac{2}{3} s$.

\subsection{Placing well-separated points on the unit sphere} \label{appendixNbig}
To prove Lemma~\ref{lemmaNbig}, we used the following lemma about placing well-separated points on the unit sphere $\mathbb{S}^{d-1} = \{x \in \reals^{d} : \norm{x}=1 \}$.
For $x, y \in \mathbb{S}^{d-1}$, $\dist_{\mathbb{S}^{d-1}}(x, y)$ equals the angle between the vectors $x$ and $y$: $\dist_{\mathbb{S}^{d-1}}(x, y) = \arccos(x^\top y)$.

Below, we use $\Vol(\mathbb{S}^{d-1})$ to denote the volume of the ``surface'' of the sphere (with the usual metric).  Note that $\Vol(\mathbb{S}^{d-1})$ does \emph{not} denote the volume of the unit Euclidean ball in $\reals^d$.  
\begin{lemma} \label{packingonsphere}
For any $d \geq 2$ and $\theta \in \big(0, \frac{\pi}{2}\big]$, there are $N \geq \frac{1}{\theta^{d-1}}$
vectors $v_1, \ldots, v_{N}$ on the $d-1$-dimensional unit sphere $\mathbb{S}^{d-1}$ satisfying
$\dist_{\mathbb{S}^{d-1}}(v_i, v_j) \geq \theta \quad \forall i \neq j.$
\end{lemma}
\begin{proof}
The sphere $\mathbb{S}^{d-1}$ is a metric space.  
The packing number on any metric space is lower bounded by the covering number~\citep[Lem.~4.2.8]{vershynin_2018}.
More precisely~\citep[Lem.~4.2.8]{vershynin_2018} imples there exist $N$ distinct vectors $\mathcal{V} = \{v_1, v_2, \ldots, v_N\}$ with $v_j \in \Sd$ such that 
\begin{enumerate}
\item $\dist_{\Sd}(v_i, v_j) \geq \theta$ for all $i \neq j$;
\item and moreover the geodesic balls on the sphere (spherical caps) of radius $\theta$ centered at $v \in \mathcal{V}$ cover $\mathbb{S}^{d-1}$, i.e., $\bigcup_{v \in \mathcal{V}} B_v^{\mathbb{S}^{d-1}}(\theta) \supseteq \mathbb{S}^{d-1}$.
\end{enumerate}
(The set $\mathcal{V}$ is said to be a maximally $\theta$-separated net.)
Therefore, the sum of the volumes of the balls $\{B_v^{\mathbb{S}^{d-1}}(\theta)\}_{v \in \mathcal{V}}$ must at least be the volume of the unit sphere, i.e., 
$$N \cdot \Vol\Big(B^{\mathbb{S}^{d-1}}(\theta)\Big) = \sum_{v \in \mathcal{V}} \Vol\Big(B_v^{\mathbb{S}^{d-1}}(\theta))\Big) \geq \Vol(\mathbb{S}^{d-1}) = \frac{2 \pi^{d/2}}{\Gamma(\frac{d}{2})}.$$
The last equality is the standard formula for the surface area of a sphere in Euclidean space.
The volume of a geodesic ball of radius $\theta$ on $\mathbb{S}^{d-1}$ is
$$\Vol\Big(B^{\mathbb{S}^{d-1}}(\theta)\Big) = \Vol(\mathbb{S}^{d-2}) \int_0^\theta \sin^{d-2}(\eta) d\eta = \frac{2 \pi^{\frac{d-1}{2}}}{\Gamma(\frac{d-1}{2})} \int_0^\theta \sin^{d-2}(\eta) d\eta,$$
see~\citep[Cor.~10.17]{lee2018riemannian} or~\citep[p.~314]{noncompactsymmetricspacevolumes}.
Using that $\theta \leq \frac{\pi}{2}$ and $\sin(\eta) \leq \eta$ for all $\eta \in [0, \frac{\pi}{2}]$, $\Vol\Big(B^{\mathbb{S}^{d-1}}(\theta)\Big) \leq \frac{2 \pi^{\frac{d-1}{2}}}{\Gamma(\frac{d-1}{2})} \frac{1}{d-1} \theta^{d-1}$.
Therefore,
$$N \geq \pi^{1/2} \frac{{(d-1)\Gamma(\frac{d-1}{2})}}{{\Gamma(\frac{d}{2})}} \frac{1}{\theta^{d-1}} \geq \pi^{1/2} \frac{{(d-1)\Gamma(\frac{d-1}{2})}}{{\Gamma(\frac{d}{2})}} \frac{1}{\theta^{d-1}} \geq \frac{1}{\theta^{d-1}}.$$


%
\end{proof}
\section{Geometry influences the objective function} \label{geometryinfluencesfunctions}
\citet{hamilton2021nogo} show that there is no strongly g-convex function which has bounded condition number on all of the hyperbolic plane.
This statement is of course not true in Euclidean space.
{Using a different technique, \citet{martinezrubio2021global} proves a similar result (see Proposition C.6 therein).}
We extend the result of \citet{hamilton2021nogo} to Hadamard spaces with sectional curvature upper bounded by $\Kup < 0$.

\begin{proposition} \label{geometryinfluencesobjective}
Let $\calM$ be a Hadamard manifold whose sectional curvatures are in the interval $(-\infty, \Kup]$ with $\Kup < 0$.
Let $f \colon \calM \rightarrow \reals$ be $L$-smooth and $\mu$-strongly g-convex in a ball $B(\xorigin, r)$.
Then  $\frac{L}{\mu} \geq \frac{1}{8}\Big(r \sqrt{-\Kup} - 1\Big)$
provided $r \geq \frac{4}{\mu} \norm{\grad f(\xorigin)} + \frac{1}{\sqrt{-\Kup}}$.
\end{proposition}

Before proceeding to the proof of Proposition~\ref{geometryinfluencesobjective}, we note that the bound $\kappa \geq \Omega(r)$ also applies to the symmetric spaces $\mathcal{SLP}_n$ and $\calP_n$, even though neither have strictly negative curvature. This is an immediate corollary of the result of \citet{hamilton2021nogo} because for every point $x$ in those spaces, there is always a totally geodesic submanifold containing $x$ and which is isometric to a hyperbolic plane (see Appendix~\ref{PDmatricesappendix}).
\\
\begin{proof}
The proof is very similar to the proof of~\citet{hamilton2021nogo}.
The main difference is we have to be a little careful because the manifold no longer necessarily has the same symmetries as a hyperbolic space.
Denote $\partial B(\xorigin, r) = \{x \in \calM : \dist(x, \xorigin) = r\}$.

Let $c = \frac{1}{\sqrt{-\Kup}}$.  
Let $x \in {\arg\min}_{y \in \partial B(\xorigin, r-c)} f(y)$.
Geodesic convexity of $f$ implies
$$\frac{r-c}{r} f(y) + \frac{c}{r} f(\xorigin) \geq f\bigg(\exp_{\xorigin}\Big(\frac{r-c}{r}\exp_{\xorigin}^{-1}(y)\Big)\bigg) \geq f(x)\quad  \forall y \in \partial B(\xorigin, r).$$
Therefore, 
\begin{align} \label{usefuleq1forgeometryinffct1}
f(y) - f(\xorigin) \geq \frac{r}{r-c} (f(x) - f(\xorigin)) \quad \forall y \in \partial B(\xorigin, r).
\end{align}
On the other hand, $\mu$-strong g-convexity of $f$ implies 
\begin{equation} \label{usefuleq1forgeometryinffct2}
\begin{split}
f(x) - f(\xorigin) &\geq \inner{\grad f(\xorigin)}{\exp_{\xorigin}^{-1}(x)} + \frac{\mu}{2} (r-c)^2 \\
&\geq - \norm{\grad f(\xorigin)} (r-c) + \frac{\mu}{2} (r-c)^2 \geq \frac{\mu}{4}(r-c)^2
\end{split}
\end{equation}
provided $r-c \geq \frac{4}{\mu}\norm{\grad f(\xorigin)}$.

Consider any geodesic $\gamma \colon \reals \rightarrow \calM$ with $\gamma(0) = x, \norm{\gamma'(0)} = 1$ and $\inner{\gamma'(0)}{\exp_{\xorigin}^{-1}(x)} = 0$.  
We claim $\gamma(\reals)$ intersects $\partial B(\xorigin, r)$ in at least two distinct points $y_+, y_-$.
By Proposition~\ref{TopogonovEuclidean},
\begin{align*}
\dist(\gamma(t), \xorigin)^2 &\geq \dist(\gamma(t), x)^2 + (r-c)^2 - 2 \inner{\exp_{x}^{-1}(\gamma(t))}{\exp_x^{-1}(\xorigin)} \\
&= \dist(\gamma(t), x)^2 + (r-c)^2 = t^2 + (r-c)^2.
\end{align*}
Choosing $t$ so that $t^2 + (r-c)^2 > r$, continuity of $\gamma$ implies that we must have $\gamma(t_+), \gamma(t_-) \in \partial B(\xorigin, r)$ for some $t_+ > 0$ and $t_- < 0$.
Let $y_+ = \gamma(t_+)$ and $y_- = \gamma(t_-)$.
It is clear that $y_+ \neq y_-$ as geodesics do not form closed loops in Hadamard manifolds \citep[Prop.~12.9]{lee2018riemannian}.
Observe 
\begin{align} \label{usefuleq1forgeometryinffct3}
\exp_x^{-1}(y_+) = t_+ \gamma'(0), \quad \text{ and } \quad \exp_x^{-1}(y_+) = t_- \gamma'(0).
\end{align}

By $L$-smoothness of $f$,
\begin{align*}
f(y_+) &\leq f(x) + \inner{\grad f(x)}{\exp_x^{-1}(y_+)} + \frac{L}{2} \dist(x, y_+)^2,\\
f(y_-) &\leq f(x) + \inner{\grad f(x)}{\exp_x^{-1}(y_-)} + \frac{L}{2} \dist(x, y_-)^2
\end{align*}
which summed yield
\begin{align*}
\frac{-t_-}{t_+ - t_-} f(y_+) + \frac{t_+}{t_+ - t_-} f(y_-) 
&\leq f(x) + \frac{L}{2}\Bigg(\frac{-t_-}{t_+ - t_-} \dist(x, y_+)^2 + \frac{t_+}{t_+ - t_-} \dist(x, y_-)^2\Bigg) \\
& \leq f(x) + \frac{L}{2} \Big( \dist(x, y_+)^2 + \dist(x, y_-)^2 \Big)
\end{align*}
where we have used~\eqref{usefuleq1forgeometryinffct3} to cancel the terms $\inner{\grad f(x)}{\exp_x^{-1}(y_{\pm})}$.
Using inequality~\eqref{usefuleq1forgeometryinffct1},
\begin{align*}
\frac{r}{r-c}(f(x) - f(\xorigin)) &\leq \frac{-t_-}{t_+ - t_-} (f(y_+) - f(\xorigin)) + \frac{t_+}{t_+ - t_-} (f(y_-) - f(\xorigin)) \\
&\leq f(x) -f(\xorigin) + \frac{L}{2}\Big( \dist(x, y_+)^2 + \dist(x, y_-)^2 \Big),
\end{align*}
which rearranging and applying inequality~\eqref{usefuleq1forgeometryinffct2} becomes
$$\frac{\mu}{4} c (r-c) = \frac{c}{r-c}\cdot \frac{\mu}{4}(r-c)^2 \leq \frac{c}{r-c}(f(x) - f(\xorigin)) \leq \frac{L}{2}\Big( \dist(x, y_+)^2 + \dist(x, y_-)^2 \Big)$$
provided $r-c \geq \frac{4}{\mu}\norm{\grad f(\xorigin)}$.

For the last step we shall upper bound $\dist(y_{+}, x)^2$ and $\dist(y_{-}, x)^2$.  
Let us focus on $\dist(y_{+}, x)^2$ since the exact same reasoning applies to $\dist(y_{-}, x)^2$.
Consider the geodesic triangle $\xorigin x y_{+}$.  Again, note that the angle at $x$ is $\frac{\pi}{2}$.  So by Proposition~\ref{TopogonovHyperbolicBoundedAbove},
\begin{align*}
\cosh(r \sqrt{-\Kup}) &= \cosh(\dist(\xorigin, y_{+}) \sqrt{-\Kup}) \\
&\geq \cosh(\dist(x, y_+) \sqrt{-\Kup}) \cosh(\dist(\xorigin, x) \sqrt{-\Kup}) \\
&= \cosh(\dist(x, y_+) \sqrt{-\Kup}) \cosh((r-c) \sqrt{-\Kup}).
\end{align*}

Using $e^{q} \cosh(t-q) = \frac{1}{2} (e^{2 q - t} + e^{t}) \geq \frac{1}{2} (e^{-t} + e^{t}) = \cosh(t)$ for any $t \in \reals$ and $q \geq 0$,
$$\cosh(\dist(x, y_+) \sqrt{-\Kup}) \leq \frac{\cosh(r \sqrt{-\Kup})}{\cosh((r-c) \sqrt{-\Kup})} \leq e^{c \sqrt{-\Kup}} = e$$
i.e., $\dist(x, y_+) \leq \frac{1}{\sqrt{-\Kup}} \arccosh(e).$

We conclude that if $r-c \geq \frac{4}{\mu}\norm{\grad f(\xorigin)}$, then
$\frac{\mu}{4} c (r-c) \leq \frac{L}{-\Kup} \arccosh(e)^2.$
Rearranging, 
$$\frac{\sqrt{-\Kup} (r-\frac{1}{\sqrt{-\Kup}})}{8} \leq \frac{\sqrt{-\Kup} (r-\frac{1}{\sqrt{-\Kup}})}{4 \arccosh(e)^2} \leq \frac{L}{\mu},$$
provided $r-\frac{1}{\sqrt{-\Kup}} \geq \frac{4}{\mu}\norm{\grad f(\xorigin)}.$
\end{proof}

\section{Positive definite matrices} \label{PDmatricesappendix}
\begin{lemma} \label{lemmageodesicsubmanifolds}
Let $\calM$ be a Hadamard manifold of dimension $d$ which contains a totally geodesic submanifold $\calN$ of dimension $d_1$.
Assume that all the sectional curvatures of the submanifold $\calN$ are upper bounded by $\Kup$, with $\Kup < 0$.
Then, $\calM$ satisfies the ball-packing property for $\tilde{r} = \frac{4}{\sqrt{-\Kup}}, \tilde{c} = \frac{d_1 \sqrt{-\Kup}}{8}$ and any $\xorigin \in \calN$.
If in addition $\calM$ is a homogeneous manifold, then $\calM$ satisfies the strong ball-packing property with the same constants $\tilde{r}$ and $\tilde{c}$.
\end{lemma}
\begin{proof}
Let $\xorigin \in \calN$.
By Lemma~\ref{lemmaNbig}, there are at least $e^{\frac{d_1 \sqrt{-\Kup}}{8} r}$ points in $B_{\calM}(\xorigin, \frac{3}{4} r)$ which are pairwise separated by a distance of $\frac{r}{2}$, provided $r \geq \tilde{r}$.
Note that here we have used that distance on $\calN$ is equal to distance on $\calM$ because $\calN$ is totally geodesic.

If $\calM$ is homogenous, then by definition for all $x, y \in \calM$ there is an isometry $\phi \colon \calM \rightarrow \calM$ such that $\phi(x) = y$.  In particular, this implies that every $x \in \calM$ is an element of a totally geodesic submanifold isometric to $\calN$.
\end{proof}

\begin{lemma} \label{lemmaproductmanifolds}
Let $\calM_1$ be a $d_1$-dimensional Hadamard manifold whose sectional curvatures are upper bounded by $\Kup$ everywhere, with $\Kup < 0$.  
Let $\calM_2$ be a Hadamard manifold.
Then $\calM = \calM_1 \times \calM_2$ satisfies the strong ball-packing property for $\tilde{r} = \frac{4}{\sqrt{-\Kup}}, \tilde{c} = \frac{d_1 \sqrt{-\Kup}}{8}$ and $\xorigin$ any point in $\calM$.
\end{lemma}
\begin{proof}
This follows from Lemma~\ref{lemmageodesicsubmanifolds} by noting that for every $x_2 \in \calM_2$, $\calM_1 \times \{x_2\}$ is a totally geodesic submanifold, so we can apply the same logic from Lemma~\ref{lemmageodesicsubmanifolds}.
\end{proof}

Let $\calP_n = \{P \in \reals^{n\times n} : P^\top = P, P \succ 0\}$ be the Riemannian manifold of $n\times n$ positive definite matrices (with real entries), endowed with the so-called affine-invariant metric
$$\inner{X}{Y}_P = \trace(P^{-1} X P^{-1} Y) \quad \text{for } P \in \calP_n \text{, and } X, Y \in \T_P \calP_n \cong \Sym(n),$$
where $\Sym(n)$ is the set of $n\times n$ real symmetric matrices.
Let $\mathcal{SLP}_n = \mathrm{SL}(n) / \mathrm{SO}(n)$ be the totally geodesic submanifold of $\calP_n$ consisting of those matrices of determinant one.
Both $\calP_n$ and $\mathcal{SLP}_n$ are important in applications~\citep{skovgaard1984riemgeogaussians,Bhatia,fletchterdiffusion2007,statsofPDmatricesfisher2006,sra2015conicgeometricoptimspd,moakher2005diffgeomspdmeans,moakher2006symmetric,allenzhuoperatorsplitting,ciobotaru2020geometrical}.
We know $\mathcal{SLP}_n$ and $\calP_n$ are symmetric spaces and Hadamard manifolds~\citep[Prop. 3.1]{dolcetti2018differential}
whose sectional curvatures are each between $-\frac{1}{2}$ and $0$~\citep[Prop.~I.1]{criscitiello2020accelerated}.  
Since they are symmetric, $\mathcal{SLP}_n$ and $\calP_n$ are also a homogeneous manifolds~\citep[prob. 6-19]{lee2018riemannian}.

It is well-known that $\mathcal{SLP}_2$ is isomorphic to the hyperbolic plane of curvature $-\frac{1}{2}$~\citep{Chossat_2009,dolcetti2018differential}, and thus satisfies a strong ball property by Lemma~\ref{lemmaNbig}.
For $n \geq 3$, \citet[Ch. II.10]{bridsonmetric} show that $\mathcal{SLP}_n$ contains a totally geodesic submanifold containing the identity matrix $I$ which is isomorphic to an $(n-1)$-dimensional hyperbolic space for some $K < 0$.  We show that $K = -\frac{1}{8}$, see Lemma~\ref{PDmatriceslemma}.
Therefore applying Lemma~\ref{lemmageodesicsubmanifolds}, $\mathcal{SLP}_n$ satisfies the strong ball-packing property with:
\begin{itemize}
\item $\tilde{r} = \frac{4}{\sqrt{1/2}} = 4 \sqrt{2}, \tilde{c} = 2 \frac{\sqrt{1/2}}{8} = \frac{1}{4 \sqrt{2}}$ if $n=2$;
\item $\tilde{r} = \frac{4}{\sqrt{1/8}} = 8 \sqrt{2}, \tilde{c} = \frac{(n-1) \sqrt{1/8}}{8} = \frac{n-1}{16 \sqrt{2}}$ if $n\geq 3$.
\end{itemize}
Since $\calP_n$ is isometric to $\reals \times \mathcal{SLP}_n$~\citep{dolcetti2018differential}, Lemma~\ref{lemmaproductmanifolds} implies the strong ball packing property holds for $\calP_n$ with the same constants $\tilde r, \tilde c$ just given.
This proves Lemma~\ref{lemmaPDmatrices}.
We note that \citet{franksreichenback2021} independently use the observation that $\mathcal{SLP}_n$ contains a hyperbolic \emph{plane} for a similar purpose.

\begin{lemma} \label{PDmatriceslemma}
For $n \geq 3$, $\mathcal{SLP}_n$ contains a totally geodesic submanifold containing $I$ which is isomorphic to the $(n-1)$-dimensional hyperbolic space of curvature $-\frac{1}{8}$.
\end{lemma}
\begin{proof}
Theorem 10.58 and Remark 10.60(4) of~\citep{bridsonmetric} state that 
$\calN = \calP_n \cap O(n-1, 1)$
is a totally geodesic submanifold of $\calP_n$ which is isometric to a $(n-1)$-dimensional hyperbolic space of some constant sectional curvature $K < 0$.
Here, 
$O(n-1, 1) = \{ A \in \reals^{n \times n} : A^\top J A = J \}$
is an indefinite orthogonal group (symmetries of the $(n-1)$-dimensional hyperboloid model in Minkowski space), where $J = \diag(1, 1, \ldots, 1, -1)$~\citep[Ex.~10.20(4)]{bridsonmetric}.

Note that $\calN \subset \mathcal{SLP}_n \cap O(n-1, 1)$ since $A^\top J A = J \implies \det(A)^2 = 1$, and any positive definite matrix has positive determinant.
Thus, $\calN$ is also a totally geodesic submanifold of $\mathcal{SLP}_n$.

We have $\T_I O(n-1, 1) = \{X \in \reals^{n \times n} : X^\top J = - J X\}$.
Therefore,
$$\T_I \calN = \Sym_0(n) \cap \T_I O(n-1, 1) = \Bigg\{\begin{pmatrix} 0_{(n-1)\times (n-1)} & s \\ s^\top & 0 \end{pmatrix} : s \in \reals^{n-1}\Bigg\}$$
where $\Sym_0(n)$ is the set of $n\times n$ real symmetric matrices with vanishing trace.

Let $s_1, s_2 \in \reals^{n-1}$ with $\|s_1\|^2 = \|s_2\|^2 = 1/2, s_1^\top s_2 = 0$.  Let
$$X_1 = \begin{pmatrix} 0_{(n-1)\times (n-1)} & s_1 \\ s_1^\top & 0 \end{pmatrix}, \quad X_2 = \begin{pmatrix} 0_{(n-1)\times (n-1)} & s_2 \\ s_2^\top & 0 \end{pmatrix}.$$
Therefore $\langle X_1, X_2\rangle = 0, \|X_1\|^2 = \|X_2\|^2 = 1$, and
$[X_1, X_2] = \begin{pmatrix} s_1 s_2^\top - s_2 s_1^\top & 0 \\ 0 & 0 \end{pmatrix}.$

By Proposition 2.3 of~\citep{dolcetti2018differential}, the curvature tensor of $\mathcal{SLP}_d$ is
$$\Rmcurv(W, X, Y, Z)(P) = -\frac{1}{4} \trace([P^{-1} W, P^{-1} X] [P^{-1} Y, P^{-1} Z]), \quad \text{for } W, X, Y, Z \in \Sym_0(n)$$
where $[X, Y] = X Y - Y X$ is the matrix commutator of $X, Y$.
Therefore,
\begin{align*}
K &= \Rmcurv(X_1, X_2, X_2, X_1)(I) = -\frac{1}{4} \trace([X_1, X_2] [X_2, X_1]) = \frac{1}{4} \trace([X_1, X_2]^2) \\
&= \frac{1}{4} \trace((s_1 s_2^\top - s_2 s_1^\top)^2) = \frac{1}{4} \cdot \frac{1}{2} \trace(-s_1 s_1^\top - s_2 s_2^\top) = -\frac{1}{8}.
\end{align*}
\end{proof}

\subsection*{For positive definite matrices, $\tilde{c} \leq O({n^{3/2}})$} \label{appendixcanwegetbetterctildeforPd}
We do not know if the constant $\tilde{c}$ stated in Lemma~\ref{lemmaPDmatrices} is the best possible constant (i.e., is as large as possible).
\citet{dolcetti2018differential} show that $\mathcal{SLP}_n$ is an Einstein manifold with constant Ricci curvature $-\frac{n}{4}$.
Therefore, by the Bishop-Gromov volume comparison theorem~\citep[Thm.~11.19]{lee2018riemannian}, the volume of a geodesic ball of radius $r$ in $\mathcal{SLP}_n$ is at most the volume of a geodesic ball in a $\dim(\mathcal{SLP}_n)$-dimensional hyperbolic space of sectional curvature $-\frac{n}{4(\dim(\mathcal{SLP}_n)-1)}$.
Hence, the volume of a geodesic ball of radius $r$ in $\mathcal{SLP}_n$ is at most
\begin{align*}
\exp\bigg(\Theta\bigg({\dim(\mathcal{SLP}_n) r \sqrt{\frac{n}{4(\dim(\mathcal{SLP}_n)-1)}}}\bigg)\bigg) 
= \exp(\Theta( {r n^{3/2}} )).
\end{align*}
On the other hand, for $r$ sufficiently large, the volume of a geodesic ball of radius $\frac{r}{4}$ in $\mathcal{SLP}_n$ is at least $1$.  So the number of disjoint balls of radius $\frac{r}{4}$ we can pack into a ball of radius $r$ in $\mathcal{SLP}_n$ is at most $\exp(\Theta( {r n^{3/2}} ))$, which implies $\tilde{c} \leq \Theta(n^{3/2})$.

\section{Comparison to Riemannian Gradient Descent} \label{curvdependenceandbestratesforRGD}
\emph{The published version of this article includes the argument below (in light gray) regarding the complexity of projected RGD.  It relies on Proposition 15 by~\citet{zhang2016complexitygeodesicallyconvex}, which states a complexity result for project RGD.  Unfortunately the proof of that result does not handle the projection step appropriately, therefore our original statements below are no longer relevant.  Fortunately, there is a version of RGD for constrained optimization which does have the query complexity $\tilde O(\kappa)$: see Proposition 17 in Appendix D of~\citep{davidmr}.  Those authors also provide details on the issue in the original argument of~\citet{zhang2016complexitygeodesicallyconvex}.}


\begin{thisnote}
\citet{zhang2016complexitygeodesicallyconvex} show that, for a bounded g-convex domain $D$ with diameter $2 r$, projected RGD initialized in $D$ finds a point $x$ within $\frac{r}{5}$ of the minimizer of $f$ in no more than $\tilde{O}(\max\{\kappa, {r \sqrt{-\Klo}}\})$ queries.  This rate depends on curvature.
However, if $\calM$ is a hyperbolic space of curvature $K < 0$, Proposition~\ref{geometryinfluencesobjective} implies $\kappa \geq \Omega(r \sqrt{-K})$.  
Hence, RGD uses at most $\tilde O(\kappa)$ queries when $\calM$ is a hyperbolic space---this is a curvature-independent rate.
We have the following proposition.  
\begin{proposition} \label{lemmacurvindependentrateforRGD0}
Let $\calM$ be a hyperbolic space of curvature $K < 0$, and let $\xorigin \in \calM$.
Let $L \geq \mu > 0$, $\kappa = \frac{L}{\mu}$, and $r > 0$.
Let $f \in \mathcal{F}_{\kappa,  r}^{\xorigin}$ be $L$-smooth and have minimizer $x^*$.
Then projected RGD
$$x_{k+1} = \mathrm{Proj}_D\Big(\exp_{x_k}\Big(-\frac{1}{L} \grad f(x_k)\Big)\Big), \quad \quad x_0 = \xorigin, \quad \quad D = B(\xorigin, r)$$
satisfies
$\dist(x_k, x^*)^2 \leq 4 \Big(1-\frac{1}{100} \cdot \frac{1}{\kappa}\Big)^{k-2} \kappa r^2$, for all $k \geq 2.$
Here, $\mathrm{Proj}_D$ denotes metric projection on to the geodesic ball $D$.  
\end{proposition}
\begin{proof}
\citet{zhang2016complexitygeodesicallyconvex} prove
$$f(x_k) - f(x^*) \leq (1-\delta)^{k-2} \frac{1}{2} L (2 r)^2 = 2 (1-\delta)^{k-2} L r^2 \quad \quad \forall k \geq 2$$
where $\delta^{-1} = \max\{\frac{L}{\mu}, \frac{r \sqrt{-\Klo}}{\tanh(r \sqrt{-\Klo})}\}$.
By $\mu$-strong g-convexity, $\frac{\mu}{2} \dist(x_k, x^*)^2 \leq f(x_k) - f(x^*)$.

First, assume $r \sqrt{-K} < 8$.
Then, $\frac{r \sqrt{-K}}{\tanh(r \sqrt{-K})} \leq 1 + r \sqrt{-K} \leq 9 \leq 9 \frac{L}{\mu}$.
Hence, $\delta^{-1} \leq 9 \frac{L}{\mu}$.

Second, assume $r \sqrt{-K} \geq 8$.  
This implies $\frac{r}{4} \sqrt{-K} - 1 \geq \frac{r}{8} \sqrt{-K}$.  
Proposition~\ref{geometryinfluencesobjective} applied to the ball $B = B(x^*, \frac{1}{4} r)$ implies that the condition number of $f$ in $B$ is at least $\frac{1}{8}(\frac{r}{4} \sqrt{-K} - 1) \geq \frac{1}{8}(\frac{r}{8} \sqrt{-K}) = \frac{r}{64} \sqrt{-K}$.
On the other hand, we know $x^* \in B(\xorigin, \frac{3}{4} r)$ because $f \in \mathcal{F}_{\kappa,  r}^{\xorigin}$.
Therefore, $B \subset B(\xorigin, r)$, which implies $\frac{L}{\mu} \geq \frac{r}{64} \sqrt{-K}$.
We conclude $\frac{r \sqrt{-K}}{\tanh(r \sqrt{-K})} \leq 1 + r \sqrt{-K} \leq 1 + 64 \frac{L}{\mu} \leq 65 \frac{L}{\mu} \leq 100 \frac{L}{\mu}$, and so $\delta^{-1} \leq 100 \frac{L}{\mu}$.
\end{proof}
\end{thisnote}

\section{Technical fact from proof of Theorem~\ref{theoremnoacceleration}} \label{apptechnicalfactfromproofofbabytheorem}
We show that the inequality~\eqref{lowerboundonAk} implies $\left|A_k\right| \geq 2$ for all $k \leq T$, where $T$ is given by~\eqref{rangefork}.
We do this by induction on $k \geq 0$.
(\textbf{Base case}) By the ball-packing property, $\left|A_0\right| \geq e^{\tilde{c} r} \geq 2$ since $r \geq \frac{4(d+2)}{\tilde{c}} \geq \frac{4}{\tilde{c}}$.
(\textbf{Inductive hypothesis}) Assume $k+1\leq T$, and $\left|A_{m}\right| \geq 2$ for all $m \leq k$.
Therefore, $\left|A_{m}\right| - 1 \geq \left|A_{m}\right| / 2$ for all $m \leq k$.

The bounds $r \geq \frac{4(d+2)}{\tilde{c}}$ and $k+1 \leq T$ imply that $k+1 \leq \floor{2 w}$ (recall $w = \tilde{c} d^{-1} r/4$).  So we can apply Lemma~\ref{keylemma} to get
$$\left|A_{m+1}\right| \geq \frac{\left|A_m\right| - 1}{(2000 w(3 \mathscr{R} \sqrt{-\Klo} + 2))^d} \geq \frac{\left|A_m\right|/2}{(2000 w(3 \mathscr{R} \sqrt{-\Klo} + 2))^d}, \quad \quad \forall m \leq k.$$
Unrolling these inequalities and using $\left|A_0\right| \geq e^{\tilde{c} r}$, we get
\begin{align} \label{djdnddkjdkd}
\left|A_{k+1}\right| \geq \frac{ e^{\tilde{c} r}/2^{k+1}}{ \big(2000 w(3 \mathscr{R} \sqrt{-\Klo} + 2)\big)^{(k+1) d}} \geq \frac{ e^{\tilde{c} r}/2^{(k+1)d}}{ \big(2000 w(3 \mathscr{R} \sqrt{-\Klo} + 2)\big)^{(k+1) d}}.
\end{align}
On the other hand, using the formula~\eqref{rangefork} for $T$, $k+1 \leq T$ implies
\begin{align} \label{inequua1}
\frac{ e^{\tilde{c} r/2} / 2^{(k+1)d}}{ \big(2000 w(3 \mathscr{R} \sqrt{-\Klo} + 2)\big)^{(k+1) d}} \geq 1.
\end{align}
Combining inequalities~\eqref{inequua1} and~\eqref{djdnddkjdkd} (and using that $e^{\tilde{c} r/2} \leq e^{\tilde{c} r} / 2$), we determine that $\left|A_{k+1}\right| \geq 2$.

\section{Deriving Theorems~\ref{cor1} and~\ref{cor3} from Theorem~\ref{maintheoremunboundedqueries}}
Theorem~\ref{cor1} from the introduction follows from Theorem~\ref{maintheoremunboundedqueries} and Lemma~\ref{lemmaNbig} (see Appendix~\ref{helper1}).
Theorem~\ref{cor3} follows from Theorem~\ref{maintheoremunboundedqueries}, Lemma~\ref{lemmaPDmatrices}, and the fact that $\mathcal{SLP}_n$ has sectional curvatures in the interval $[-\frac{1}{2}, 0]$~\citep[Prop.~I.1]{criscitiello2020accelerated} (see Appendix~\ref{helper2}).

\subsection{Deriving Theorem~\ref{cor1} from Theorem~\ref{maintheoremunboundedqueries} and Lemma~\ref{lemmaNbig}} \label{helper1}
We use the values for $\tilde r$ and $\tilde c$ given by Lemma~\ref{lemmaNbig}.
First, we have to check that the assumptions of Theorem~\ref{cor1} imply $r \geq \max\big\{\tilde r, \frac{8}{\sqrt{-\Klo}}, \frac{4 (d+2)}{\tilde{c}}\big\}$.
Indeed, the bound $\kappa \geq 1000 \sqrt{\frac{\Klo}{\Kup}}$ implies $\kappa \geq 1000$, and so
\begin{align*}
r &= \frac{\kappa - 9}{12 \sqrt{-\Klo}} \geq \frac{\frac{99}{100} \kappa}{12 \sqrt{-\Klo}} 
\geq \frac{990}{12 \sqrt{-\Kup}} \geq \frac{64}{\sqrt{-\Kup}} \\
&\geq \max\bigg\{\tilde r, \frac{8}{\sqrt{-\Klo}}, \frac{4 \cdot 8 (d+2)}{d \sqrt{-\Kup}}\bigg\} = \max\bigg\{\tilde r, \frac{8}{\sqrt{-\Klo}}, \frac{4 (d+2)}{\tilde{c}}\bigg\}.
\end{align*}

Second, we have to verify the lower bound in Theorem~\ref{cor1} follows from the lower bound for $T$ given in Theorem~\ref{maintheoremunboundedqueries}.  We have
\begin{align*}
\frac{\sqrt{-\Kup}}{8} \frac{\kappa}{12 \sqrt{-\Klo}} \geq 
\tilde{c} (d+2)^{-1} r = \frac{d \sqrt{-\Kup}}{8(d+2)} \frac{\kappa - 9}{12 \sqrt{-\Klo}} 
\geq \frac{\sqrt{-\Kup}}{16} \frac{\frac{99}{100} \kappa}{12 \sqrt{-\Klo}}.
\end{align*}
Therefore,
\begin{align*}
T &\geq \Bigg\lfloor \frac{\sqrt{-\Kup}}{16} \frac{\frac{99}{100} \kappa}{12 \sqrt{-\Klo}} \cdot \frac{1}{\log(2\cdot 10^6 \cdot \frac{\sqrt{-\Kup}}{8} \frac{\kappa}{12 \sqrt{-\Klo}} (r \sqrt{-\Klo})^2)} \Bigg\rfloor \\
&\geq \Bigg\lfloor \frac{\sqrt{-\Kup}}{16} \frac{\frac{99}{100} \kappa}{12 \sqrt{-\Klo}} \cdot \frac{1}{\log(2^{-2}\cdot 10^6 \cdot \frac{\kappa}{12} (\frac{\kappa}{12})^2)} \Bigg\rfloor \\
&\geq \Bigg\lfloor \frac{\sqrt{-\Kup}}{16} \frac{\frac{99}{100} \kappa}{12 \sqrt{-\Klo}} \cdot \frac{1}{3 \log(10 \kappa)} \Bigg\rfloor
 \geq \Bigg\lfloor \sqrt{\frac{\Kup}{\Klo}} \cdot \frac{\kappa}{1000 \log(10 \kappa)} \Bigg\rfloor.
\end{align*}

\subsection{Deriving Theorem~\ref{cor3} from Theorem~\ref{maintheoremunboundedqueries} and Lemma~\ref{lemmaPDmatrices}} \label{helper2}
We use the values for $\tilde r$ and $\tilde c$ given by Lemma~\ref{lemmaPDmatrices}, and $\Klo = -\frac{1}{2}$.
First, we have to check that the assumptions of Theorem~\ref{cor3} imply $r \geq \max\big\{\tilde r, \frac{8}{\sqrt{-\Klo}}, \frac{4 (d+2)}{\tilde{c}}\big\}$.
Indeed, the bound $\kappa \geq 1000 n$ implies $\kappa \geq 1000$, and so
\begin{align*}
r &= \frac{\kappa-9}{6 \sqrt{2}} 
\geq \frac{\frac{99}{100}\kappa}{6 \sqrt{2}} \geq \frac{990 n}{6 \sqrt{2}} 
\geq \max\bigg\{8 \sqrt{2}, \frac{8}{\sqrt{1/2}}, \frac{2 (n(n+1)+2)}{\tilde{c}}\bigg\} \\
&\geq \max\bigg\{\tilde r, \frac{8}{\sqrt{-\Klo}}, \frac{4 (d+2)}{\tilde{c}}\bigg\}.
\end{align*}
For the second to last inequality, we used (a) $\frac{990 n}{6 \sqrt{2}} \geq \frac{2 (n(n+1)+2) \cdot 16 \sqrt{2}}{n-1}$ for all $n \geq 3$, and
(b) $\frac{990 n}{6 \sqrt{2}} \geq 2(n(n+1)+2) 4 \sqrt{2}$ if $n=2$.
For the last inequality, we used $d = \dim(\mathcal{SLP}_n) = \frac{n(n+1)}{2}-1$.

Second, we have to verify the lower bound in Theorem~\ref{cor3} follows from the lower bound for $T$ given in Theorem~\ref{maintheoremunboundedqueries}.  We have for $n \geq 2$
\begin{align*}
\frac{n-1}{4 \sqrt{2}} \frac{2}{n(n+1)+2} \cdot \frac{\kappa}{6\sqrt{2}} &\geq 
\tilde{c} (d+2)^{-1} r = \frac{n-1}{c_n \sqrt{2}} \frac{2}{n(n+1)+2} \cdot \frac{\kappa-9}{6\sqrt{2}} 
\\ &\geq \frac{n-1}{16 c_n \sqrt{2}} \frac{2}{n(n+1)+2} \cdot \frac{\frac{99}{100} \kappa}{6\sqrt{2}},
\end{align*}
where $c_n = 1$ if $n \geq 3$ and $c_2 = 1/4$.
Therefore,
\begin{align*}
T &\geq \Bigg\lfloor \frac{n-1}{16 c_n \sqrt{2}} \frac{2}{n(n+1)+2} \cdot \frac{\frac{99}{100} \kappa}{6\sqrt{2}} \cdot \frac{1}{\log(2\cdot 10^6 \cdot \frac{n-1}{4 \sqrt{2}} \frac{2}{n(n+1)+2} \cdot \frac{\kappa}{6\sqrt{2}} (\frac{\kappa}{6\sqrt{2}})^2)} \Bigg\rfloor \\
&\geq 
\Bigg\lfloor \frac{1}{16 \sqrt{2}} \frac{2}{\frac{7}{3} n} \cdot \frac{\frac{99}{100} \kappa}{6\sqrt{2}} \cdot \frac{1}{\log(2\cdot 10^6 \cdot \frac{1}{4 \sqrt{2}} \frac{2}{7} \cdot \frac{\kappa}{6\sqrt{2}} (\frac{\kappa}{6\sqrt{2}})^2)} \Bigg\rfloor \\ &\geq 
\Bigg\lfloor \frac{1}{16 \sqrt{2}} \frac{2}{\frac{7}{3} n} \cdot \frac{\frac{99}{100} \kappa}{6\sqrt{2}} \cdot \frac{1}{3 \log(10 \kappa)} \Bigg\rfloor
\geq 
\Bigg\lfloor \frac{1}{n} \cdot \frac{1}{1000 \log(10 \kappa)} \Bigg\rfloor.
\end{align*}
For the third to last inequality, we used that $\frac{3 c_n}{7 n} \leq \frac{n-1}{n(n+1)+2} \leq 3$ for all $n \geq 2$.

\end{document}